\documentclass[a4paper,reqno,11pt]{amsart}
\usepackage[utf8]{inputenc}
\usepackage[normalem]{ulem}
\usepackage{amsmath,amsthm,amsfonts,amssymb,mathtools}
\usepackage{graphicx}
\usepackage{color}
\usepackage{float,dsfont}
\usepackage[square,numbers]{natbib}
\usepackage{thm-restate}
\usepackage{hyperref}
\usepackage{comment}
\mathtoolsset{showonlyrefs}
\usepackage[title]{appendix}
\usepackage[reversemp,
    paperwidth=210mm,
    paperheight=297mm,
    top={26mm},
    headheight={5.5pt},
    headsep={5.6mm},
    text={150mm,227mm},
    marginparsep=5mm,
    marginparwidth=12mm,
    bindingoffset=10mm,
    footskip=10mm]{geometry} 

\numberwithin{equation}{section}
\newtheorem{theorem}{Theorem}[section]
\newtheorem{corollary}[theorem]{Corollary}
\newtheorem{lemma}[theorem]{Lemma}
\newtheorem{proposition}[theorem]{Proposition}
\theoremstyle{definition}
\newtheorem{definition}[theorem]{Definition}
\newtheorem{remark}[theorem]{Remark}

\newcommand{\black}{\color{black}}

\DeclareMathOperator*{\esssup}{ess\,sup}
\begin{document}
\font\myfont=cmr12 at 20pt
\title[Anisotropic rigid body interacting with a Poiseuille flow]{Long-time behavior of an anisotropic rigid body interacting with a Poiseuille flow\\ in an unbounded {2D} channel}
\author{Denis Bonheure}
\address{
	Département de Mathématique\\
	Université Libre de Bruxelles\\
	Boulevard du Triomphe 155\\
	1050 Brussels - Belgium
}
\email{denis.bonheure@ulb.be}
\author{Matthieu Hillairet}
\address{
	Institut Montpelliérain Alexandre Grothendieck\\
	Université de Montpellier\\
	Place Eu\-gène Bataillon\\
	34090 Montpellier - France}
\email{matthieu.hillairet@umontpellier.fr}
\author{Clara Patriarca}
\address{Département de Mathématique\\
	Université Libre de Bruxelles\\
	Boulevard du Triomphe 155\\
	1050 Brussels - Belgium}
\email{clara.patriarca@ulb.be}
\author{Gianmarco Sperone}
\address{
	Dipartimento di Matematica\\
	Politecnico di Milano\\
	Piazza Leonardo da Vinci 32\\
	20133 Milan - Italy}
\email{gianmarcosilvio.sperone@polimi.it}
\date{\today}
\thanks{This research was initiated in Brussels when M.H. was a visiting fellow, C.P. was on a PhD internship and G.S. was post-doc, all supported by the ARC Advanced 2020-25 grant ``PDEs in interaction'' at ULB. D.B. is also supported and by the Francqui Foundation as Francqui Research Professor 2021-24. M.H. is supported by the Institut Universitaire de France. }
\maketitle

\begin{abstract}
\noindent
We study the long-time behavior of an elliptic rigid body which is allowed to vertically translate and rotate in a 2D unbounded channel under the action of a {Poiseuille} flow at large distances. The motion of the fluid is modelled by the incompressible Navier-Stokes equations, while the motion of the solid is described through Newton's laws. In addition to the solid  inertia and the hydrodynamic forces, we assume the dynamics of the solid is driven by internal elastic restoring forces but without any structural damping. Through a precise description of the motion of the elliptic body whenever it comes close to the channel boundaries, we prove global-in-time existence of weak solutions. Our second main contribution is a proof of return to equilibrium in case the amplitude of the Poiseuille flow is small. \black To our knowledge,  this represents the first long-time analysis of fluid-solid interaction problems with a given non-trivial final state.  

\end{abstract} 
\tableofcontents


\newcommand{\tildex}{\tilde x}
\section{Introduction}
\subsection{Context}
Because of its relevance to applications,  describing the motion of 
particles in a channel is one of the crucial problems in the field of fluid-solid interactions \cite[Section 7-3]{Happel&Brenner65}.  From the analytical standpoint,  this is a simple system to state that contains a wide variety of intriguing phenomena.  Consider for instance the fall of a particle inside a vertical channel.  If the particle is sufficiently small and close to the centerline of the channel,  interactions with the boundaries are neglectable and the equations of motion are fully integrable (for simple particle shape). This leads for instance to (by now classical) Jeffery's orbits \cite{Jeffery22} in case of elliptical particles.
The situation becomes more involved when the particle moves close to the channel boundaries.  One must then take into account the influence of the walls to compute the hydrodynamic forces exerted on the particle.  When fluid inertia is neglected, several studies provide either exact formulas with suitable transformations \cite{ONeill&Stewartson67,Cooley&ONeill68,Cooley&Oneill69,Dean&Oneill64} or approximations when the distance between the obstacle and the wall is small  with respect to its dimensions \cite{Cox67}. 
However,  such descriptions lead to paradoxical results.  For instance,  assuming the particle is spherical prevents from a lateral migration of the particle relative to the wall.  This is in contradiction with the observed Segré-Silberberg effect  that predicts particles falling in a cylindrical channel have privileged orbits \cite{SegreSilberberg62}.  One common explanation of this paradox is to reintroduce fluid convection \cite{Hogg}: relying on matched aysmptotic methods, approximate solutions enable to recover the privileged orbits of Segré and Silberberg (see the introduction of \cite{Matasetal} for a review of  known results).  However, to  our knowledge, handling the full nonlinear problem is still an open issue since it makes  the mathematical analysis more subtle with no explicit exact analytical formulas for solutions.  

One ambition of our study is to contribute to the mathematical analysis of solid/wall interactions in this fluid/solid context.  We aim in particular to provide a sharp description of particle trajectories without relying on exact
or approximate formulas.  To this end,  we start from another standpoint on particle motion in channels.  Following some related models used in the engineering literature \cite{cossu2000instability, dolci2019bifurcation, fani2015motion,obligado2013bi}, we consider a spring-mounted, undamped, rigid elliptic body which is immersed in an incompressible viscous flow modelled by Navier-Stokes equations and which is free to move in the direction orthogonal to an incoming Poiseuille flow (see the precise definition \eqref{eq:poiseuille} below).  This toy model serves as a paradigm for many problems in physics and engineering sciences such as the wind blowing on a bridge or the oscillations of a submarine communication cable (see \cite{williamson2004vortex} for more examples).  In this apparatus, the elliptic body is at rest if it lays on the channel centerline with its axis parallel and orthogonal to the channel axis.  In case the body shifts or rotates from this rest position,  the created asymmetry interacts with the Poiseuille flow and makes the body to move. 
In the absence of restoring forces,  the particle would leave the channel centerline possibly to reach a Segré-Silberberg orbit.  There is then a competition between the fluid forces and the restoring forces.  Numerical evidence \cite{patriarca2022numerical} shows that, in this regard, if the intensity of the inflow is sufficiently small, in the long time the streamlines of the fluid are characterized by an upstream-downstream symmetry. In particular, it is possible to observe the presence of two symmetrical eddies in a closed re-circulation zone in the wake of the obstacle and, thus, the flow can be considered steady. As soon as the inflow's intensity is above a critical threshold value, the flow pattern makes a transition from an $x_2$-symmetrical steady configuration to an unsteady periodical one, where it exhibits a vortex shedding phenomenon \cite{blackburn1999study}. 
While this issue has been extensively studied from the numerical and experimental point of view,  a mathematical approach is still lacking. In this paper, we propose a mathematical approach of the first regime  of small inflow in the 2D setting.

\subsection{Model.}   To describe precisely the model, we set
\begin{equation}\label{prova}
B_{eq} = \left\{ (x_1,x_2) \in \mathbb{R}^{2} \ \Bigg| \ \dfrac{x_{1}^{2}}{d^2} + \dfrac{x_{2}^{2}}{\delta^{2}} \leq 1 \right\}
\end{equation}
which describes an elliptic rigid body $B_{eq}$ of mass $\mathcal{M}>0$ and moment of inertia $\mathcal{J}>0.$ To fix the ideas,  we assume that $\delta < d$ below  so that $d$ is the large axis of $B_{eq}$.   We recall that, in this 2D setting the body the moment of inertia is a time-independent scalar.  We assume that the body is immersed in a two-dimensional channel $A$ of width $2\mathcal L.$  Up to a good choice of coordinates we fix $A \doteq \mathbb{R}\times (-\mathcal L,\mathcal L).$ The upper and lower boundaries of the channel are therefore given by $\Gamma=\mathbb{R}\times\{-\mathcal L\}\cup \mathbb{R}\times\{\mathcal L\}.$ We restrict to the case $\mathcal L > 2d$ so that -- with the assumption that the large axis of $B_{eq}$ has length $2d$ -- the body cannot touch simultaneously the upper and lower part of $\Gamma.$ We will comment on  this restriction further on.  What remains of the channel is filled with a
homogeneous incompressible viscous fluid. 

The rigid body is free to rotate and move vertically but it does not translate horizontally inside the channel.  We respectively denote by $h$ and $\theta$ the vertical displacement of the barycenter of the rigid body and its anticlockwise  rotation  from the central line $x_2=0$.  Thus,
\begin{equation}\label{eq:Bh}
 B(h,\theta)= Q(\theta)B_{eq}+h\,\widehat{e}_2 \doteq \begin{bmatrix} \cos\theta & -\sin\theta \\ \sin\theta & \cos\theta \end{bmatrix} B_{eq}+h\,\widehat{e}_2\,, \qquad \forall\, (h,\theta) \in A_{d,\delta}\,,
\end{equation}
tracks the position of the body after the vertical translation of height $h$ and the rotation of angle $\theta$.
{Here,} $A_{d,\delta}$ denotes the set of admissible values for $(h,\theta)$ {that imposes $B_{eq} \Subset A:$} 
$$
A_{d,\delta} \doteq \left\{(h,\theta) \in \mathbb{R}^{2} \ \Big| \ {|h|+\sqrt{(d\sin\theta)^2+(\delta\cos\theta)^2}<\mathcal L} \right\} \, .
$$

\begin{figure}[H]
    \centering
    \includegraphics[scale=0.5]{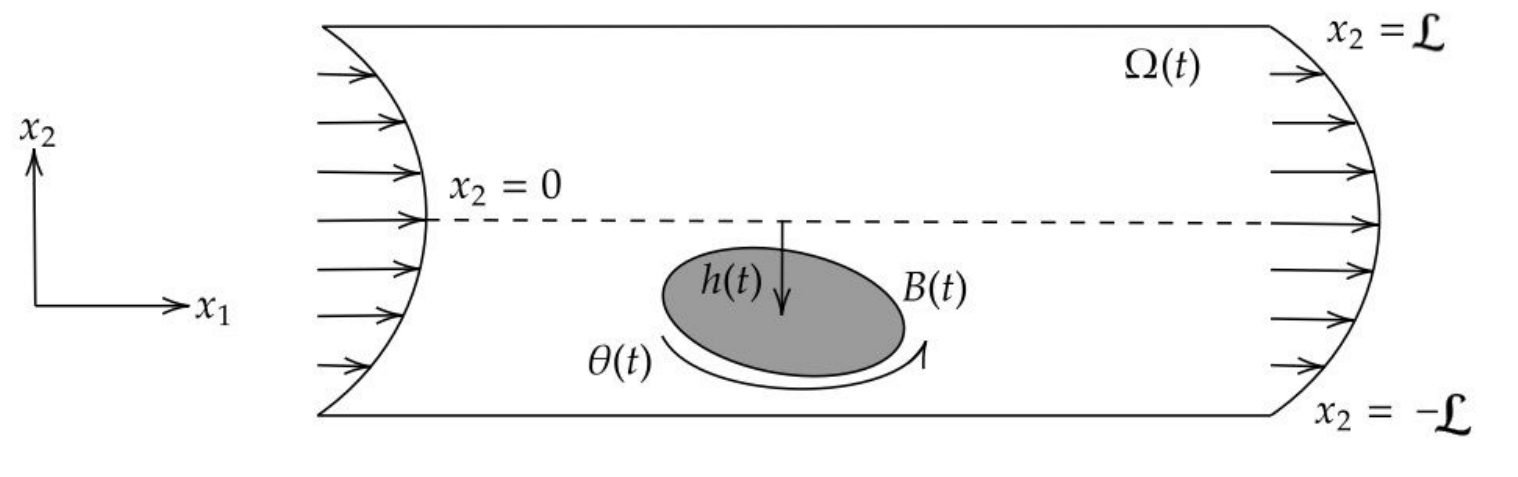}
    \vspace{-4mm}
    \caption{Summary of the notations.} 
    \label{fig:my_label}
\end{figure}

{We consider the displacement of the elliptical particle on a timespan $(0,{\mathcal T}).$} Due to the motion of the rigid body inside $A$, the domain occupied by the fluid
is variable in time and is given at time $t$ by $\Omega(t)=A\setminus \overline{B(t)}$, where 
\begin{equation}\label{eq:domain}
B(t) = \{(x_1,x_2)\in  \mathbb{R}^{2} \mid Q(\theta(t))^T(x_1,x_2-h(t))\in B_{eq}\}.
\end{equation}
For simplicity, in the sequel we will sometimes omit emphasizing the dependence on $t\in(0,\mathcal T)$ and, with an abuse of notation, we will denote through the Cartesian product $\Omega(t)\times (0,\mathcal T)$ the space-time domain defined by
\begin{equation}
\{(x,t)\in\mathbb{R}^2\times\mathbb{R}_+ \, | \, x\in\Omega(t), \ t \in (0,\mathcal T) \}\,.
\end{equation}
We will do the same abuse while writing $\partial B(t) \times (0,\mathcal T)$ instead of $\cup_{t\in(0,\mathcal T)}\partial B(t)\times\{t\}$.
{All notations are summarized in Figure \ref{fig:my_label}.}

In the fluid domain $\Omega(t),$ we introduce
$U:\Omega(t)\times (0,\mathcal T)\to\mathbb{R}^2$ and $P:\Omega(t)\times(0,\mathcal T)\to~\mathbb{R}$ that denote respectively the velocity field and the pressure of the fluid. We denote by $\mu>0$ and  $\rho>0$ the viscosity and the density of the fluid. We recall that Newtonian laws state that the stress tensor of the fluid reads
\begin{equation}\label{eq:stress}
\Sigma_\mu(U,P)=-P\mathbb{I}+2\mu D(U)\quad\mbox{with}\quad
D(U)=\frac{\nabla U+(\nabla U)^{\top}}{2}\, ,
\end{equation}
where $\mathbb{I}$ is the $2\times 2$-identity matrix, and that we have the algebraic relation
\[
\nabla \cdot \Sigma_\mu(U,P) = \mu \Delta U - \nabla P\,
\]
for any divergence free vector field $U$. 
In the absence of a body in the channel $A$,  the fluid flow would be driven by a (prescribed) pressure drop {per unit length} $\mathcal P_0>0$. We would have then $(U,P)=(v_p, \pi_p)$
where the so-called stationary Poiseuille flow $v_p$ associated with the pressure $\pi_p$ solves
\begin{equation}\label{eq:NS_Poiseuille}
	-\mu \,\Delta v_p+\rho (v_p\cdot \nabla)v_p+\nabla \pi_p=0 \, , \qquad \nabla \cdot v_p=0 \qquad \text{in} \ \  A\,,
\end{equation}
with no-slip boundary conditions, i.e. $v_p=0$ on $\Gamma$. With this boundary condition, the equation \eqref{eq:NS_Poiseuille} has a unique unidirectional solution whose expression is explicit, namely
\begin{equation}\label{eq:poiseuille}
\left\{
\begin{aligned}
	& v_p(x_1,x_2)= v_{0}(x_2) \widehat{e}_{1} = \frac{\mathcal P_0\,\mathcal L^2}{2\mu}\left(1-\frac{x_2^2}{\mathcal L^2}\right) \widehat{e}_{1} \qquad \forall\, x_{2} \in [-\mathcal L,\mathcal L] \,,\\
	& \pi_p(x_1,x_2) = -\mathcal P_0 x_1	\,.
\end{aligned}	
\right.
\end{equation} 
In particular, the pressure drop  per unit length $\mathcal P_0$ is directly related to the flow rate $2\mathcal P_0 \mathcal L^3/(3\mu)$.
In the presence of a solid body inside $A$, the couple velocity-pressure $(U,P)$ is a solution to the Navier-Stokes equations on $\Omega (t)\times (0,\mathcal T)$ forced by the Poiseuille flow $v_p$ at spatial infinity,  the body velocity-field on $\partial B(t)$ and no-slip boundary conditions on $\Gamma.$

As for the elliptical body, its motion follows Newton's laws of solid dynamics taking into account the action of the fluid and restoring forces only.  Assuming continuity of normal stress at the fluid/solid interface, the hydrodynamics forces produced by the fluid on the rigid body $B$ are given by 
$$-\int_{\partial B(t)} \Sigma_\mu(U,P)\widehat{n} \, d\sigma,$$
where the minus sign is due to the fact that we choose $\widehat{n}$ to be the outward normal to $\partial\Omega(t)$, thus directed towards the interior of $B$.  The lift force that enters in Newton's equation along $\widehat{e}_2$ reads then
$$-\widehat{e}_2\cdot\int_{\partial B(t)} \Sigma_\mu(U,P)\widehat{n} \, d\sigma\,.$$
The torque  produced by the fluid on the rigid body $B$ is given by
$$-\int_{\partial B(t)} (x-h\widehat{e}_2)^{\perp} \cdot \Sigma_\mu(U,P)\widehat{n} \, d\sigma\,,$$
and contributes to the the angular part of Newton's equations.  
We recall that we assume that $B$ cannot move horizontally (this can be achieved by applying an active drag  force on $B$). 
In Newton's equations for linear and angular momentum we also add respectively $F_h, F_\theta : \mathbb{R}^{2} \to \mathbb{R}$ that represent smooth elastic restoring forces, the model case being purely linear forces given by Hooke's law.  These forces are obtained as partial derivatives of a  smooth potential 
$F:\mathbb{R}^2\to \mathbb{R}$, namely
\[
F_h(h,\theta)\doteq \frac{\partial F}{\partial h}(h,\theta)\,, \qquad F_\theta(h,\theta)\doteq \frac{\partial F}{\partial \theta}(h,\theta)\,.
\]
We take $F(0,0)=0$ and we further assume that 
\begin{itemize}
\item for any ${\theta}_M>0$, there exist ${\varpi_M},r_M>0$ such that
\begin{equation}\label{hp_F}
   h \partial_h F(h,\theta) + \theta \partial_\theta F(h,\theta) \ge  {\varpi_M} F(h,\theta),\qquad  F(h,\theta) \ge \dfrac{r_M}{2} (h^2+\theta^2),
\end{equation}
for all $(h,\theta) \in A_{d,\delta}$ with  $|\theta|<\theta_M$;
\item there exists an increasing function $\Theta:\mathbb{R}^+\to \mathbb{R}^+: \ell\mapsto\theta(\ell)$ such that
\begin{equation}\label{coercivity}
\{(h,\theta)\in A_{d,\delta} : F(h,\theta)\le \ell \}\subseteq \mathbb R\times(-\theta(\ell),\theta(\ell))\,.
\end{equation} 
\end{itemize}
Such assumptions are in particular satisfied when the restoring force satisfies linear Hooke's law. They are also compatible with more realistic nonlinear restoring forces saturating at some values of $h$ and $\theta$. 

In a nutshell,  the fluid-solid interaction evolution problem gives rise to the coupled system 
\begin{equation} \label{eq:evolution_pb}
\left\{
\begin{aligned}
& \rho(\partial_{t} U+(U\cdot \nabla)U) = \nabla \cdot \Sigma_\mu(U,P)\, ,\quad  \nabla \cdot\, U=0 \text{ in} \bigcup_{t\in(0,\mathcal T)}\Omega(t)\times\{t\}, \\ \medskip
& \lim_{|x_1|\to\infty} U(x_1,x_2,t) = v_p(x_2) \text{ in } [-\mathcal L,\mathcal L] \times (0,\mathcal T),  \\ \medskip
& U=0 \text{ on} \ \ \Gamma\times (0,\mathcal T),\ U={h'}\,\widehat{e}_2+\theta' (x-h\widehat{e}_2)^{\perp} \text{ on} \bigcup_{t\in(0,\mathcal T)}\partial B(t)\times\{t\} , \\ \medskip
& \mathcal M{h''} +F_{h}(h,\theta)=-\widehat{e}_2\cdot\int_{\partial B} \Sigma_\mu(U,P)\widehat{n} \, d\sigma \text{ in } (0,\mathcal T) , \\ \medskip
& {\mathcal J} \theta'' +F_{\theta}(h,\theta)=-\int_{\partial B} (x-h\widehat{e}_2)^{\perp} \cdot \Sigma_\mu(U,P)\widehat{n} \, d\sigma \text{ in } (0,\mathcal T).
\end{aligned}
\right.
\end{equation}
We complement the system with initial conditions 
\begin{equation}\label{eq:ic}
\left\{
\begin{aligned}
   & (h,{h'})(0)=( h_0,{h'_0}),\\
   & (\theta,\theta')(0)= ( \theta_0,{\theta'_0}),\\
  &   U(x,0)=U_{0}(x)\text{ in }{\Omega(0) = A \setminus \overline{B(0)}},
\end{aligned}
\right.
\end{equation}
for some initial  position $(h_0,\theta_0) \in A_{d,\delta},$ velocities $(h'_{0}, \theta'_{0}) \in \mathbb{R}^2$ and $U_0 : \Omega_0 \to \mathbb R^2$, where, here and in the sequel, $\Omega_0 \doteq \Omega(0) = A\setminus \overline{B_0}$ with $B_0 \doteq B(0)=B(h_0,\theta_0).$ Compatibility conditions on the initial conditions will be imposed in \eqref{eq:ic2}. 

In order to restrict the number of physical parameters, it is customary to work with a dimensionless form of the system \eqref{eq:evolution_pb}. The choice is usually motivated by the physics of the problem and the phenomena to be highlighted. We emphasize that our choice here is purely mathematically intended as it serves to normalize concomitantly the Stokes and Reynolds numbers. We take ${\mu}/{\rho d}$ as our reference speed, ${\rho d^2}/{\mu}$ as characteristic time and $d$ as characteristic length.  As a consequence, the physical parameters in the fluid equation then appear only in the inflow/outflow condition at spatial infinity. In the new variables, the channel's width is ${L} = \mathcal L/d$, the axes of the ellipse are respectively of lengths $1$ and ${e}=\delta/d$, whereas the velocity at infinity is now given by the dimensionless Poiseuille velocity field
$$
{{\tilde v_p}({\tildex}_2)}=  \frac{p_0 L^2}{2}\left(1-\frac{\tildex_2^2}{{L}^2}\right) \widehat{e}_{1} \qquad \forall {\tildex}_{2} \in [-L,L]\,,
$$
with 
\begin{equation}\label{eq:tildep0}
{p_0} = \frac{\mathcal P_0\rho d^3}{\mu^2}.
\end{equation}
In the new variables, the dimensionless system reads, with obvious notations,
\begin{equation} \label{eq:evolution_pb-scaled}
\left\{
\begin{aligned}
& \partial_{\tilde t} {\tilde U}+({\tilde U}\cdot \tilde \nabla){\tilde U} = \tilde \nabla \cdot \Sigma_1({\tilde U},{\tilde P})\, ,\quad  \tilde \nabla \cdot\, {\tilde U}=0 \text{ in} \bigcup_{\tilde t\in(0,T)}\tilde{\Omega}(\tilde t)\times\{\tilde t\} \, , \\ \medskip
& \lim_{|\tildex_1|\to\infty} {\tilde U}(\tildex_1,\tildex_2,\tilde  t) = \tilde  v_{p}(\tilde x_2) \text{ in } [-L,L] \times (0,T),  \\ \medskip
& {\tilde U}=0 \text{ on} \ \ {\tilde \Gamma\times (0,T),\ {\tilde U}= {\tilde h'}}\,\widehat{e}_2+{\tilde\theta' (\tildex-\tilde h}\widehat{e}_2)^{\perp} \text{ on} \bigcup_{\tilde t\in(0,T)}\partial \tilde B(\tilde t)\times\{\tilde t\}   \, , \\ \medskip
&  \frac{\mathcal M}{\rho d^2} {\tilde h''} +  \frac{\rho}{\mu^2}  G_{\tilde h} (\tilde h, \tilde \theta)=- {\widehat{e}_2}\cdot{\int_{\partial \tilde B} \Sigma_1({\tilde U},{\tilde P})\widehat{n} \, }d{\tilde\sigma \text{ in } (0,T)} \, , \\ \medskip
&   { \frac{\mathcal J}{{\rho d^4}}	\tilde \theta'' + \frac{\rho}{\mu^2}  G_{\tilde \theta}(\tilde h,\tilde \theta)=- \int_{\partial \tilde B} (\tildex- \tilde h}\widehat{e}_2)^{\perp} \cdot \Sigma_1({\tilde U},{\tilde P})\widehat{n} \, d\tilde \sigma \text{ in } (0,T) \, ,
\end{aligned}
\right.
\end{equation}
where $G(\cdot,\cdot) = F(d\,\cdot,\cdot)$ and $T = ({\mu}/{\rho d^2})\mathcal T$. 
For the initial conditions, we have
\begin{equation}\label{eq:ic-scaled}
\left\{
\begin{aligned}
   & {(\tilde h,{\tilde h'})}(0)=\frac1d (h_0, \frac{\rho d^2}{\mu}h'_0){\doteq (\tilde{h}_0,\tilde{h}'_0)},\\
   & {(\tilde\theta,\tilde\theta')}(0)= (\theta_0,\frac{\rho d^2}{\mu}\theta'_0){\doteq (\tilde{\theta}_0,\tilde{\theta}'_0)},\\
  &   {\tilde U(\tilde x,0)}=\frac{\rho d}{\mu}U_{0}(d \tildex){\doteq \tilde{U}_0}\text{ in }{\tilde\Omega_0 = \tilde A \setminus \tilde B_0}.
\end{aligned}
\right.
\end{equation}
In the fluid equation,  we remark that both the fictive density and viscosity are now set to $1$ as desired while ${p_0}$ contains all physical dependencies. The quantities that now appear in the Newton equations have a clear physical meaning: $\frac{\mathcal M}{{\rho d^2}}$ and $\frac{\mathcal J}{{\rho d^4}}$ are density ratios
whereas $\frac{\rho}{\mu^2}G$ is a dimensionless potential.

\medskip

To end the introduction of our model,  we remark that, in the absence of a Poiseuille flow (say $p_0= 0$) we have a formal {\em a priori} estimate by multiplying the fluid equation in \eqref{eq:evolution_pb-scaled} with $\tilde U$. Integrating by parts, this entails:
\begin{multline} \label{eq_energy_intro}
\dfrac{1}{2}\dfrac{{d}}{{d}\tilde t} \left[  {\frac{\mathcal M}{\rho d^2}} |\tilde h'(\tilde t)|^2 + 
\frac{\mathcal J}{\rho d^4}  |\tilde \theta'(\tilde t)|^2 + \int_{\tilde \Omega(\tilde t)} |\tilde U(\tilde t,\tilde x)|^2 { d}\tilde x   + \frac{2\rho}{\mu^2}{G}(\tilde h(\tilde t),\tilde \theta(\tilde t)) \right] \\
+ \int_{\tilde \Omega(\tilde t)} 2 |D(\tilde U)(\tilde t,\tilde x)|^2 { d}\tilde x { d}\tilde t = 0.
\end{multline}
The system \eqref{eq:evolution_pb-scaled} can then be interpreted as a partially dissipative (only fluid viscosity involves dissipation) system forced by the Poiseuille flow. 
Observe also that the energy scales uniformly with the factor $\frac{\rho}{\mu^2}$ in the adimensionalization, namely
\begin{equation}\label{eq:energyscaling}
E^{adim}(\tilde t) = \frac{\rho}{\mu^2} E^{phys}(t),
\end{equation}
where 
\[
\begin{aligned}
& E^{adim} (\tilde t) = \frac12\left(\frac{\mathcal M}{\rho d^2}  |\tilde h'(\tilde t)|^{2} + 
\frac{\mathcal J}{\rho d^4} |\tilde \theta'(\tilde t)|^{\rm 2} + \int_{\tilde \Omega(\tilde t)} |\tilde U(\tilde t,\tilde x)|^2 { d}\tildex  + 2\frac{\rho}{\mu^2}  G(\tilde h(\tilde t),\tilde \theta(\tilde t))\right),\\
& E^{phys} (t) = \frac{1}{2}\left(  \mathcal M| h'(t)|^{2} + {\mathcal J} |\theta'(t)|^{2} + \int_{\Omega(t)} |U(t,x)|^{2} { d}x+ 2{F}(h(t),\theta(t))  \right).
\end{aligned}
\]%
All dimensionless variables and rescaled sets have been surmounted by a ``tilda'' to highlight the changes in \eqref{eq:evolution_pb-scaled}-\eqref{eq:energyscaling} but we will return to the original notations in the sequel to lighten the presentation. On the contrary, we will keep for clarity the new notations of the adimensional constants $L,e,p_0,T$. We refer to \eqref{eq:adim-const} for further simplifications of the notations. From now on, we also set $ \Sigma= \Sigma_1$. 

\subsection{Previous related references}
In this paper we aim to construct a global-in-time solution to \eqref{eq:evolution_pb-scaled} and analyze its long-time behavior. 
The Cauchy theory for fluid/solid problem has been thoroughly studied around the beginning of the 21st century.
Mostly two standpoints are proposed. The first one is a pseudolagrangian formulation in which characteristics associated with the solid motion are constructed in order to fix the fluid domain.  This enables to transform the Navier-Stokes system into a quasilinear system that is coupled with the Newton laws.  Classical solutions can then be constructed via a perturbative method.  Various results are obtained, with possible restrictions on the physical parameters  depending on the way the linearized problem is handled \citep{conca,CumsilleTakahashi,CumsilleTakahashi2,Hieber_etal,Grandmont&Maday00,Takahashi,TakahashiTucsnak}.  The second approach relies on a fully eulerian formulation of the coupled problem and energy estimate \eqref{eq_energy_intro}. This formulation is obtained by remarking that the fluid/solid system can be seen as a plain fluid with constrained velocity-field on the solid domain.  The constraint can then be mollified to construct approximate solutions that lead to solutions to the target system via a compactness method.  This yields naturally weak solution to the coupled problem (with an appropriate notion of weak solution as we shall see below) \cite{DejEste, DejEste2,GunzLeeSeregin,Hoffmann&Starovoitov99,Hoffmann&Starovoitov00}.  Eventually, these methods yield existence and uniqueness of classical solutions until possible collision between the bodies or between one body and the container boundaries (under smallness assumption on the initial data in the 3D case) while weak solutions are proven to exist globally regardless collisions \cite{Feireisl03,SMStarTuck}.  In the 2D case,  the picture is complete since weak solutions are known to be unique and coincide with the classical one before collision \cite{GlassSueur}.  All these results are obtained either in a bounded domain or in the whole space.  Semi-bounded containers containing one solid body are considered in \cite{HillTaka,patriarca} extending the existence of weak solutions.  In all these previous references, rotations of the body were either neglected or did not influence the fluid domain.

Collision in fluid/solid problems is a delicate issue.  It should be noted here that, in systems like \eqref{eq:evolution_pb-scaled} where we assume no-slip of fluid on solid boundaries,  collision implies a crossing of the characteristics associated with the fluid velocity-field.  This prevents from keeping a smooth fluid velocity-field through a collision.  Thus,  first computations discuss which norms of classical solutions should blowup in case of collision \cite{Star}. These computations are then complemented in \cite{Star2} showing that a finite-time collision (under the action of sufficiently singular forces) would lead to the existence of several (weak) solutions after collision.  Evidence that no collision in finite time is possible (whether with weak or strong solutions) are provided in \cite{Hesla,Hillairet2007,hillairet2023global,HillTaka,sabbagh2019motion}  except for very peculiar 3D geometries \cite{HillairetTakahashi10b}.  The result was recently  extended to the motion of spheres in $\mathbb R^3$ \cite{hillairet2023global}.  These rigorous results extend  previous heuristic computations relying on lubrication approximation \cite{Cox67,Happel&Brenner65}.  Explanations for the discrepancy between a no-collision result and what is observed in  real-life experiments are also discussed in following references: influence of solid geometry or boundary conditions \cite{GeVaHill,GeVaHill1,GVHW,HillTaka3},  influence of viscosity \cite{Sokhna,Munnier1,Munnier2},  influence of incompressibility assumption \cite{Sarka}.  

The above analysis of Cauchy theory yields that weak solutions exist globally in time (regardless collision) and that, at least in simple configurations,  these solutions have a reasonable sense since no collision occurs in finite time.  There are fewer references regarding the large-time behavior of these solutions.  This might be due to the fact that, in case the container is bounded and there are no forcing term, a simple energy/dissipation estimate yields that the kinetic energy of the system decays exponentially fast so that the solution converges to a trivial rest state.  Without forcing term,  the situation of the all space is more complex.  An algebraic decay of solutions is computed in \cite{MR4643426,Ervedoza1,Ervedoza2} but no asymptotic profile of the solution is known to date.  The case of a spherical particle moving in a compressible fluid under the action of gravity is analyzed in \cite{FereislNecasova}.  Regarding fluid/solid systems forced by an inflow,  a first issue to analyze the large time behavior is the existence of stationary states.  This issue is discussed in \cite{BoGaGa}.  We point out that for symmetry reasons,  we have trivially a stationary solution to our coupled problem when the solid is in the position $B_{eq}$. 

\subsection{Main contributions}
 In this paper,   our first contribution is to develop a weak solution framework for \eqref{eq:evolution_pb-scaled}. 
We start with existence of a solution until the first collision between $B(t)$ and $\partial A$ when the initial data matches the compatibility conditions, that is
\begin{equation}\label{eq:ic2}
\left\{
\begin{aligned}
 & \nabla \cdot U_{0}  = 0 && \text{in} \ \  \Omega_0 \, , \\[3pt]
 & \lim_{|x_1|\to\infty} U_{0}(x_1,x_2) = v_p(x_2) &&  \forall x_{2} \in [-L,L] \, , \\
& U_{0}  \cdot \widehat{e}_2 =0 &&  \text{on} \ \ \Gamma\, , \\[3pt]
 & U_{0} \cdot \widehat{n}=  (h'_{0} \,\widehat{e}_2+\theta'_{0}(x-h_0)^{\perp}) \cdot \widehat{n} && \text{on} \ \ \partial  B_0 \,. \\[3pt]
\end{aligned}
\right.
\end{equation}
 We also obtain that these weak solutions are time-continuous and satisfy a suitable energy estimate (see Theorem \ref{th:existence_weak_local}).
We complement then the study of the Cauchy problem with a careful analysis of the collision issue.   For this we obtain non-trivial extensions of computations from the above-cited references in two directions.  First,  in comparison to our model, an important simplification shared by all the references mentioned above is that the geometry of the gap (understood as the thin contact region between the body and the walls of the container) depends only on a distance parameter $\mathfrak d=\mathfrak d(h).$ This was possible since either the motion of the body is confined to a vertical translation or the bodies are assumed to be spherically symmetric.  In the case at consideration, where an elliptic body  is allowed to rotate, the geometry of the gap depends both on $h$ and $\theta$.   Hence,  we need to generalize previous computations to non-symmetric configurations in which we must capture the time-evolution of the gap geometry.  
Secondly,  previous computations only care for finite time-horizon of solutions.  Herein,  we obtain a global-in-time bound from below on the distance between the elliptic body $B(t)$ and $\Gamma.$ This is a crucial step toward the analysis of the long-time behavior of the body.  Indeed,  with our dissipative system of equations, we expect that after a transition time, during which the excess of kinetic energy of the full system is dissipated, the body returns to its rest state with the fluid unknowns converging to the stationary solution outside this domain. However, 
this is possible only if the particle does not remain stuck in a viscous layer close to one part of $\Gamma.$ Such a scenario is impossible for small Poiseuille flow when fluid inertia is neglected, that is setting $\rho=0$ in \eqref{eq:evolution_pb}.  In this regime,  hydrodynamic forces applied to the body split into the sum of a contribution proportional to the body velocities ($h'$ and $\theta'$) and a contribution of the Poiseuille flow proportional to $p_0.$ When the elliptic body slows down  due to fluid viscosity,  it remains to compare the amplitude of Poiseuille flow contribution (whose proportionality coefficient of $p_0$ turns out to be bounded independent of the position of the body in the channel)  and the restoring forces.  If $p_0$ is sufficiently small, the restoring force dominates and the body goes  back to rest state.  When one reintroduces fluid inertia, i.e. the full Navier-Stokes equations are considered, the situation is more involved.  We cannot rule out {\em a priori} that fluid convection combines with the Poiseuille flow to retain the body close to one boundary.  Hence,  we provide a fine comparison between the effects of convection combined with Poiseuille flow and restoring forces to yield the following theorem.
\begin{theorem} \label{th:uniform_distance-intro}
For any initial data $((h_0,\theta_0),U_0,(h'_0,\theta'_0)) \in A_{1,e} \times \{L^2(\Omega_0)+v_p\}\times \mathbb R^2$ satisfying \eqref{eq:ic2}, there exists ${p}_*^{(1)}>0$ such that, if $p_0\le {p}_*^{(1)}$,
 then 
\begin{enumerate}
	\item[(i)] the weak solutions $(U,h,\theta)$ to \eqref{eq:evolution_pb-scaled}-\eqref{eq:ic-scaled} are global-in-time ;
	\item[(ii)] there exists a constant  ${\mathfrak d}_{min}^0 >0$, such that  
$$
		{\rm dist}(B(t), \partial A) \geq {\mathfrak d}_{min}^0\,, \qquad \forall t \geq 0 \, .
$$
\end{enumerate}
\end{theorem}	

The smallness assumption $p_0\le {p}_*^{(1)}$ is required only to deduce the uniform lower bound in assertion (ii). Without this condition, weak solutions are still global-in-time but we cannot rule out collisions in infinite time. 

Theorem \ref{th:uniform_distance-intro} naturally allows to address the large-time behavior of solutions to \eqref{eq:evolution_pb-scaled} for small $p_0$. To this respect, our second main result shows that any weak solution of \eqref{eq:evolution_pb-scaled} converges, as time goes to infinity and provided $p_0$ is small enough, to the equilibrium configuration $(U_{\rm eq},P_{\rm eq},0,0)$,
where $(U_{\rm eq},P_{\rm eq})$ is the unique steady state in $\Omega_{eq}=A\setminus B_{eq}$, i.e.
\begin{equation} \label{eq:st_pb-intro}
\left\{
\begin{aligned}
& -\Delta U_{\rm eq}+ (U_{\rm eq}\cdot \nabla)U_{\rm eq}+\nabla P_{\rm eq}=0\,,\ \nabla \cdot\, U_{\rm eq}=0 \text{ in } \Omega_{eq}  \, , \\[3pt]
& U_{\rm eq}=0 \text{ on }  \partial B_{eq}\cup \Gamma,\ \lim_{|x_1|\to\infty} U_{\rm eq}(x_1,x_2) = v_p(x_2),\  \forall x_{2} \in [-L,L] \,. 
\end{aligned}
\right.
\end{equation}
The configuration $(U_{\rm eq},P_{\rm eq},0,0)$ is indeed an equilibrium of \eqref{eq:evolution_pb-scaled} since, for symmetry reasons,  
\begin{equation} \label{eq:st_pbc-intro}
\widehat{e}_2\cdot\int_{\partial B_{eq}} \Sigma(U_{\rm eq},P_{\rm eq})\widehat{n} \, d\sigma = \int_{\partial B_{eq}} x^{\perp} \cdot \Sigma(U_{\rm eq},P_{\rm eq})\widehat{n} \, d\sigma =0 \,.
\end{equation}
The second main result then reads as follows.
\begin{theorem} \label{lerayintro}
For any initial data $((h_0,\theta_0),U_0,(h'_0,\theta'_0)) \in A_{1,e} \times \{L^2(\Omega_0)+v_p\} \times \mathbb R^2$  satisfying \eqref{eq:ic2}, there exists ${p}_*^{(2)}\in(0,{p}_*^{(1)}]$ such that, if $\,p_0 \le {p}_*^{(2)}$ and $(U,h,\theta)$ is a weak solution to \eqref{eq:evolution_pb-scaled}-\eqref{eq:ic-scaled},  
then $(U,h,\theta)$ converges to the equilibrium $(U_{\rm eq},0,0)$ as $t\to\infty$ in the following sense:
$$
\lim_{t\to\infty}\|U(t)-U_{\rm eq}\|^2_{L^2(A)}=\lim_{t \to +\infty} (|h(t)|^2+|\theta(t)|^2) =\lim_{t \to +\infty} (|h'(t)|^2+|\theta'(t)|^2)= 0\, .
$$
\end{theorem}
We point out that in this statement, we compare the fluid velocity fields $U(t)$ and $U_{\rm eq}$ on the full channel up to extending them by their solid counterparts on the associated solid domains.
The precise dependence of the threshold values $p_*^{(1)}, p_*^{(2)}$ and of the minimal distance $\mathfrak{d}_{min}^0$ on the relevant parameters of the problem (the initial data and the relevant physical parameters) will be clearly given in Theorem \ref{th:uniform_distance} and Theorem \ref{leray}.  We point out that Theorem \ref{lerayintro} is much stronger than a local asymptotic stability result (in particular Theorem \ref{th:uniform_distance} would be much easier to prove for small data) because \textit{any} initial data in a bounded region of the phase space is covered by the statement, i.e. whatever $r>0$, there exist thresholds $p_*^{(1)}(r), p_*^{(2)}(r)$ that apply to any initial datum whose energy norm is less than $r$. 
Theorem \ref{lerayintro} also shows that, in the absence of an external forcing, i.e. when $p_0=0$, any possible trajectory of the system \eqref{eq:evolution_pb-scaled} will go to rest, which proves the presence of an intrinsic damping mechanism of the fluid.

Theorem \ref{th:uniform_distance-intro} and Theorem \ref{lerayintro} are written for the dimensionless Problem \eqref{eq:evolution_pb-scaled}. If we go back to the system written in the physical variables, i.e. to Problem \eqref{eq:evolution_pb}, one naturally wonders the influence of the physical parameters on the thresholds $p_*^{(1)}$ and $p_*^{(2)}$.  It is in fact interesting to notice that two mechanisms enter the game. The pressure drop $\mathcal P_0$  needs to be small compared to the ratio $\mu^2/\rho$ in order to get a control of the energy. But this is not enough, $\mathcal P_0$ needs also to be small compared to the stiffness of the elastic force $F_h$ in order to get a control of  the distance to the walls of the channel. This means that Theorem \ref{lerayintro} should be seen as a stability result for small pressure drop rather than for small Reynolds number or high viscosity (compared to the density). We point out that it is by no means our intention to study the asymptotic regimes of small or large viscosity but some comments are given in Remark \ref{remark_viscosity2} and Remark \ref{remark_viscosity3}. 

\par

The proofs of Theorem \ref{th:uniform_distance-intro} and Theorem \ref{lerayintro} contain severe difficulties associated with both fluid and solid parts of system \eqref{eq:evolution_pb-scaled}.  Concerning the solid part,  we emphasize that, with the elastic restoring forces that are included here,  Newton's laws can be seen as  dispersive equations with an additional implicit fluid damping.  This further dispersive structure prevents from a simple application of energy/dissipation methods to tackle the long-time behavior.  
Using viscosity dissipation combined with trace theorems, one can (classically) show that the kinetic energies of the solid and the fluid are damped. Obtaining damping terms for the potential energy of the solid to control the amplitude of its oscillations is much more delicate. To this aim, we import an idea that we learned from \cite{haraux1985non,haraux1985two}, originally devised for the wave equation (and by now commonly used for other dissipative models), to our setting:  we introduce a perturbed energy functional that improves the dissipative estimate for the energy of the system.

As for the fluid part,  a critical difficulty is that the fluid problem is set on an unknown moving domain with non-homogeneous boundary conditions. In particular, to prove the convergence to the equilibrium configuration,  it is natural to compare the solutions of two partial differential equations that are set on different domains: the (unknown) time-moving domain and the stationary configuration $h=\theta=0$.  The proof of Theorem \ref{lerayintro} therefore requires different ingredients to overcome these difficulties. 

The main idea behind the proof of Theorem \ref{th:uniform_distance-intro} lies in the introduction of a potential energy of contact that we extract by analyzing the asymptotics of the Stokes problem in thin domains.  As already mentioned, our asymptotic analysis is built on previous results, \cite{GeVaHill} for instance. The main challenge in our situation consists in dealing with general solid motions and gap geometries. A tentative in that direction has been done in \cite{MR4643472} but the obtained formulas are not tractable for our purposes. In our analysis, we use a genuine variational point of view in the spirit of \cite{Sokhna}. The new crucial step is then to identify the influence of an optimized constant (see the quantity $c_*$ appearing in {Section \ref{sec:test-function}}) that is vanishing when considering simple motions and symmetric geometries. We then compute a time-evolution equation for this potential energy of contact by using suitable multipliers for the Navier-Stokes equations.  

Our proof of Theorem \ref{lerayintro} is made in  two steps. We  first use improved dissipative estimate to conclude that the position of the ellipse is close to its expected asymptotic value for large times. This enables us to proceed with an asymptotic stability analysis starting from small initial data. To this aim, we define an ad-hoc change of variables to bring the equilibrium position $B_{eq}$ to the position of the solid at time $t$. Namely, we locally transform $B_{eq}$ into $B(t)$ by using the isometry $x\mapsto h(t)\widehat{e}_2 + Q(\theta(t)) x$ in a compact neighbourhood of $B_{eq}$ that we connect to the identity outside a compact neighbourhood of $B(t)$. Using this change of variables to rewrite the steady solution $U_{eq}$ in the time-depending fluid domain $\Omega(t)$, say $V_{eq}(t,\cdot)$, we then workout a stability estimate on the equation satisfied by $U-V_{eq}$. In this argument, the space-regularity of $U_{eq}$ and the time-regularity of $(h,\theta)$ are the key ingredients to control the difference $U-V_{eq}$ far from the solid whereas the Galilean invariance of the Navier-Stokes equations are used to deal with the estimates of the forces and of the fluid velocity close to $\partial B(t)$.

We emphasize that, even if the symmetry simplifies drastically the analysis, Theorem \ref{th:uniform_distance-intro} and Theorem \ref{lerayintro} are new also in the case of a disk. Any asymptotic stability result is a challenge for the Navier-Stokes equations and it becomes even more challenging in the context of a fluid-solid interaction models. 
Nevertheless,  one of the originality and novelty of our  work is to deal with a non-spherical body.  The choice of an elliptic body is justified by the aim of breaking the full symmetry of a disk without making computations impossible (as it would be for a body of arbitrary shape). We believe that our analysis should apply to any regular strictly convex shape at the cost of some analytical technicalities but without further conceptual difficulties.  Our methodology is however limited to sufficiently small convex body.  Indeed,  when the body is large (keeping the elliptical assumption for simplicity) $B(t)$ can touch simultaneously the top and the bottom of the channel.  If $B_{eq}$ is not strictly convex,  the geometry of gaps in close-to-contact configurations can also be complicated.  For these reasons, the collision issue would require another approach when relaxing this assumption.  Another limitation of our approach is that the translation of $B_{eq}$ is chosen to be vertical only.  In this respect,  this study can be seen as a first simplified step towards the understanding of the full complex dynamics.   Such extension would be relevant to other applications, like the motion of red blood cells flowing in the circulatory system. In that case, there is no spring-type forces applying on each cell but cell-cell and cell-vessel interactions.
We do not expect that the collision issue would be affected by such a generalization but the existence and stability of rest states could be different.  Finally, in this paper we do not deal with uniqueness of solutions to \eqref{eq:evolution_pb-scaled}: the understanding of this issue strongly relies on the characterization of the behavior of the semigroup associated to the fluid-structure system, which is clear for 2D bounded domains \cite{GlassSueur}, but still not understood in partially bounded domains.

\subsection{Organization of the paper and notations}
 Section \ref{sec:preliminaries} provides some preliminary notions: we describe the functional setting and we construct a solenoidal extension for the Poiseuille flow, which enables us to introduce an equivalent formulation of \eqref{eq:evolution_pb-scaled}, given by problem \eqref{eq:hom_pb1}-\eqref{eq:ODE_newref1}, where the velocity field vanishes at infinity. 

In Section \ref{sec:existence}, we give a global-in-time (up to collision) existence result of weak solutions to problem \eqref{eq:hom_pb1}-\eqref{eq:ODE_newref1} and we derive bounds for the kinetic energy of the system independent of the distance between $B(t)$ and $\partial A$.  
 
 Sections \ref{sec:distance1} and \ref{sec:distance2} are devoted to the proof of Theorem \ref{th:uniform_distance}, which yields the existence of a uniform lower bound for the distance between the obstacle and the boundary of the channel for all times and implies Theorem \ref{th:uniform_distance-intro}. In Section \ref{sec:distance1}, we give a detailed description of the geometry of the gap between $B$ and $\partial A$ at small distances and we construct some suitable test functions required to compute the repulsive force exerted by the boundary $\partial A$ on $B$ through the fluid. In Section \ref{sec:distance2}, we proceed to the actual proof of Theorem \ref{th:uniform_distance}, which consists in providing an upper bound for the potential energy of contact. In order to focus on the main ideas of the argument, technical details required for the proof of Theorem \ref{th:uniform_distance} are contained in Sections \ref{app_difftheta}, \ref{app_asymptotics} and \ref{app_P2} of the Appendix.

Finally, Section \ref{sec:equil} contains Theorem \ref{leray}, where we prove convergence to the equilibrium state as time goes to infinity when $p_0$ is sufficiently small and therefore deduce Theorem \ref{lerayintro}.

Our problem contains a large list of parameters. In the forthcoming computations, some are of primary interest, namely the initial energy and the pressure drop, while others are less critical as for instance $L,e,m,J$ (see \eqref{eq:adim-const} for the definitions of $m$ and $J$). In our computation, we consider the less critical parameters as fixed and we will be cautious with the dependencies of constants in terms of the one of primary interest. To this aim, we will reserve two notations for the constants appearing in the estimates: $C_{dyn}$ will denote a ``dynamical constant'' that depends possibly on $L,e,m,J$ and $F$; $C_{geo}$ will denote a ``geometric constant'' that allows dependencies in $L$ and $e$ only.

Other notations will be used for universal constants and whenever critical, we explicit the dependencies. All constants denoted with a symbol $C$ may vary from one line to another.


\section{Weak formulation of the fluid-solid system}\label{sec:preliminaries}
In this section we give a notion of weak solution to \eqref{eq:evolution_pb-scaled} following \cite{denisetal,BoGaGa,patriarca}. From now on, we shorten some notations as follows: 

\begin{equation}\label{eq:adim-const}
m=\frac{\mathcal M}{\rho d^2}\,,\ J=\frac{\mathcal J}{{\rho d^4}}\,,\ {\lambda_0} = \frac{p_0 L^2}{2} = \frac{\mathcal P_0\rho \mathcal{L}^2}{2\mu^2d}
\text{ and } H(\cdot,\cdot)=\frac{\rho}{\mu^2}  G(\cdot,\cdot) = \frac{\rho}{\mu^2}  F(d\cdot,\cdot) \,. 
\end{equation}

We emphasize that $H$, up to adimensionalization, satisfies both bounds in \eqref{hp_F}
with 
$
\overline{\varrho}_M=(\rho/{\mu^2})\min(d^2,1){r_M}
$
and some 
$
\overline{\varpi}_M>0\,
$
depending on $\theta_M >0$:
\begin{equation}\label{hp_H}
H(h,\theta) \ge \dfrac{\overline{\varrho}_M}{2} (h^2+\theta^2),\quad |\theta|<\theta_M.
\end{equation}
 We start by recalling how to treat the non-constant behavior at infinity.  We then introduce the associated function spaces and weak formulation of \eqref{eq:evolution_pb-scaled}.   \par To explain our construction,  we assume in the whole section  that $T>0$ and that $(U,P,h,\theta)$ is a smooth solution to \eqref{eq:evolution_pb-scaled} on $(0,T)$
such that $(h(t),\theta(t)) \in { A_{1,e}}$ for all $t \in (0,T). $  We point out that, since the larger axis of $B_{eq}$ is ${1}$ and ${L > 2}$ we have:
\begin{equation} 
\left\{
\begin{aligned}
& {\rm dist}(B(t), \{ x_2 = L\}) > \frac{ {1}}2\, && \text{when }h(t) <\frac{{1}}{2} \, , \\[4pt]
& {\rm dist}(B(t),\{x_2  =-L\}) > \frac{ {1}}2  && \text{when }h(t) > -\frac{{1}}{2} \, , \\[4pt]
& \min( {\rm dist}(B(t),\{x_2  =-L\}), {\rm dist}(B(t), \{ x_2  = L\}) ) \geq \frac{{1}}2 && \text{when } h(t) \in \bigg\{-\frac{{1}}2,\frac{{1}}2\bigg\}  \, .
\end{aligned}
\right. \label{eqgeometric}
\end{equation}
Below, we fix $\varepsilon_0 =\frac{{1}}4$.

\subsection{Reformulation of the fluid equation}
The fluid equation in Problem \eqref{eq:evolution_pb-scaled} is set in a two-dimensional unbounded channel with a prescribed non-zero velocity field at infinity, namely the  Poiseuille flow in \eqref{eq:poiseuille}.   To construct a finite-energy setting,  we begin by capturing the flow for large values of $|x_1|$.  We construct an {\em ad hoc} truncation around the rigid body of the Poiseuille velocity profile by adapting (by now classical) Ladyzhenskaya's procedure \cite{Ladyzhenskaya1969} (see also \cite{amick} and \cite[Chapter VI]{galdi}) to our time-dependent problem. For this, we divide the channel $A$ into three parts, {\em i.e.} we set
\begin{equation} \label{eq:partition_A}
A = \bigcup_{i=0}^{2} A_{i} \, ,
\end{equation}
where
\begin{equation}\label{eq:partition}
A_{0} \doteq A \cap ([-{3},{3}] \times \mathbb{R})\,,
\quad A_{1} \doteq A \cap ((-\infty,-{3}) \times \mathbb{R})\,,\quad A_{2} \doteq A \cap (({3},\infty) \times \mathbb{R}) \, ,
\end{equation} 
and then prove the following lemma.

\begin{lemma}\label{lemma:extension}
 There exists a  smooth divergence-free vector-field $s:(0,T)\times A\to \mathbb R^2 : (t,x) \mapsto s(t,x)$ such that 
\begin{equation} \label{eq:props}
\left\{
\begin{aligned}
& s(t,\cdot)=(0,0) && \text{in a neighborhood of $B(t) $ and on $\partial A,$} \\[3pt]
& s(t,\cdot)=v_{p} && \text{in} \ \ A_{1} \cup A_{2} \, , 
\end{aligned}
\right.
\end{equation}
Furthermore, there exists  $C_{geo} >0$, depending only on $L$, such that
\begin{itemize}
\item[(i)]  $\| s(t,\cdot) \|_{W^{2,\infty}(A)} \leq { C_{geo}} \lambda_0$, for all $t \in (0,T)$,
\vspace{6pt}
\item[(ii)] $\|\partial_t s(t,\cdot)\|_{L^2(A)} \leq { C_{geo}} \lambda_0 |h'(t)|$, for all $t \in (0,T)$,
\vspace{6pt}
\item[(iii)] the mapping $h \to \widehat{f}[h]\, \dot{=}\, \partial_t s$ is continuous from $W^{1,\infty}(0,T)$ into $L^2((0,T) \times A)$,\medskip
\vspace{6pt}
\item[(iv)] $\widehat{f}[h]=h' \bar{f}[h]$ where  ${\rm supp}(\bar{f}[h,\theta]) \subset A_0$ and the mapping $h \mapsto \bar{f}(h)$ is Lipschitz from $W^{1,\infty}(0,T)$ into $L^{\infty}((0,T) \times A)$ with
\[
\begin{aligned}
{\rm a)} &  \|\bar{f}[h]\|_{L^{\infty}(A)} \leq {C_{geo}\lambda_0}, &&  \forall \, h  \in W^{1,\infty}(0,T) \, , \\[6pt]
{\rm b)} &  \|\bar{f}[h] - \bar{f}[\tilde{h}]\|_{L^{\infty}((0,T)\times A)} \leq {C_{geo}\lambda_0} \|h - \tilde{h} \|_{L^{\infty}(0,T)},    && 
\forall \, (h,\tilde{h}) \in [W^{1,\infty}(0,T)]^2,
\end{aligned}
\]
\vspace{6pt}
\item[(v)] if $\widehat{g}(s) \,\dot{=} - (s\cdot \nabla)\, s - \nabla \pi_p + \Delta s  \in \mathcal{C}^{\infty}(\bar{A})$, then $\text{supp}(\widehat{g}) \subset A_0$ and 
\[
 \|\widehat{g}\|_{L^2(A)} \leq {C_{geo} \lambda_0 \left(1 + \lambda_0\right)}\,.
\]
\end{itemize}
\end{lemma}
\noindent
\begin{proof}
Let $b : [-L,L] \longrightarrow \mathbb{R}$ the function defined by 
$$
 b(x_2)=-\lambda_0 L\left[ \frac{x_2}{L}-\frac{1}{3} \frac{x_2^3}{L^3} \right] \qquad \forall x_{2} \in [-L,L] \, ,
$$
so that 
$
b'(x_2)= v_{p}(x_{2})\cdot \widehat{e}_1
$
for all $x_{2}\in(-L,L)$.
We take smooth cutoff functions  $\zeta_1,\zeta^+_0,\zeta_0^-,\chi$ such that
\[
\left\{\begin{split}
&\mathds{1}_{[-{2,2}]} \leq \zeta_{1} \leq \mathds{1}_{[{-3,3}]}\,, \ \mathds{1}_{[{\frac{1}2},\infty)} \leq \chi \leq \mathds{1}_{[{\frac{1}4},\infty)}\,,\\
&\mathds{1}_{(-\infty,L-\varepsilon_0]} \leq  \zeta_{0}^-  \leq \mathds{1}_{(-\infty,L-\varepsilon_0/2]}\,, \ \mathds{1}_{[-L+\varepsilon_0,\infty)} \leq  \zeta_{0}^+ \leq \mathds{1}_{[-L+\varepsilon_0/2,\infty)}\,.  
\end{split}\right.
\]
Then for any $(x_1,x_2) \in A$, we write
$$
Z^+(x_1,x_2)=1 - \zeta_1(x_1)\,\zeta_{0}^+(x_2 ) \,, \  Z^-(x_1,x_2) =  1 - \zeta_1(x_1)\,\zeta_{0}^-(x_2),
$$
and
\[
Z[h](x_1,x_2) = \chi(h) Z^+(x_1,x_2) + (1-\chi(h)) Z^-(x_1,x_2).
\]
Finally, we define
\begin{equation*} 
s(t,x_1,x_2) \doteq \left( -\frac{\partial}{\partial x_2}  \big( b(x_{2}) Z[h](x_1,x_2) \big),\frac{\partial}{\partial x_1} \big(b(x_{2})Z[h](x_1,x_2)\big) \right),
\end{equation*}
for $(t,x_1,x_2) \in (0,T)\times  A$.
Since $\varepsilon_0=\frac{{1}}4$ and recalling \eqref{eqgeometric}, standard computations show that $s$ verifies all the stated properties. 
\end{proof}
We emphasize that the vector field $s$ constructed in Lemma \ref{lemma:extension} depends on time through the function $h$ but it is only space-dependent on time interval along which the solid is far away from the centerline of the channel. Indeed, by construction we observe that $s(t,\cdot) = s^{-}(\cdot)$ (with obvious notations) is time-independent so that $\widehat{f}[h]\, \dot{=}\, \partial_t s=0$ when $-\frac{{1}}{2}<  h(t)<\frac{{1}}2$. 
The assertion (iv) is important to treat the term $\widehat{f}$ in the Cauchy theory, see Section \ref{section:Cauchy}.

\medbreak

We now look for solutions to problem \eqref{eq:evolution_pb-scaled} in the form $U=u+s$ and $P=p+\pi_p$.  Then,  $(u,p,h,\theta)$ solves the following problem with source term $\widehat{g} - { \widehat{f}[h]}$
\begin{align}
&
\label{eq:hom_pb1}
 \left\{
\begin{aligned}
& \left.\begin{array}{rl}\medskip
\partial_{t} u-\nabla \cdot \Sigma(u,p)+ (u\cdot \nabla)\,u+(u\cdot \nabla)\,s+(s\cdot \nabla)\,u\!\!
& = \widehat{g} - { \widehat{f}[h]}  \\ \medskip
\nabla \cdot u \!\! & = 0
\end{array}\right\} \text{in } \Omega(t)\times (0,T)\, , \\ \medskip
& \lim_{|x_1|\to\infty}u(x_1,x_2,t)=0,\  \forall x_{2} \in [-L,L] \, , \ t \in (0,T) \, , \\ \medskip
& u=0 \text{ on } \Gamma\times (0,T),\  u={h'}\,\widehat{e}_2 + \theta'(x-h \widehat{e}_{2})^{\perp} \text{ on } \partial B(t) \times (0,T),
\end{aligned}
\right.
\\
\label{eq:ODE_newref1}
& \left\{
\begin{aligned}\medskip
& m{h''} +H_{h}(h,\theta)=-\widehat{e}_2\cdot\int_{\partial  B(t)} \Sigma(u,p)\widehat{n}\, d\sigma \text{ in } (0,T) \, , \\ \medskip
& J{\theta''} +H_{\theta}(h,\theta)=-\int_{\partial B(t)} (x-h\widehat{e}_{2})^{\perp} \cdot \Sigma(u,p)\widehat{n}\, d\sigma \text{ in } (0,T) \, ,
\end{aligned}
\right.
\end{align}
with initial conditions
\begin{equation}\label{ichat}
\left\{
\begin{aligned}
&     (h,{h'})(0)=( h_0,h'_0),\\
& (\theta,\theta')(0)=( \theta_0,\theta'_0),\\
& u_{0}(x) \doteq U_0(x)-s(0,x) \text{ in } \Omega_0 \, .
\end{aligned}
\right.
\end{equation}
{To write down this set of equations, we used the fact that the Poiseuille flow does not produce lift and torque on $B$, {\em i.e.}
\[
\widehat{e}_2 \cdot \int_{\partial B(t)} \Sigma(v_p,\pi_p) \widehat{n} = \int_{\partial B(t)} (x-h\widehat{e}_2)^{\bot}\cdot  \Sigma(v_p,\pi_p)\widehat{n} = 0\,.
\]
Assuming $(u,p,h,\theta)$ is a smooth solution to Problem \eqref{eq:hom_pb1}-\eqref{eq:ODE_newref1}-\eqref{ichat} and taking $\phi \in \mathcal{C}^{\infty}_c([0,T) \times A)$ such that  \black $\phi(t,x) = \ell(t)\widehat{e}_2 + \alpha(t) (x-h(t)\widehat{e}_2)^{\bot}$ for some $(\ell,\alpha) \in \left[ \mathcal{C}^{\infty}_c([0,T))\right]^2$ for $x\in B(t)$, we multiply \eqref{eq:hom_pb1}$_1$ by $\phi$ and integrate by parts over space and time. 
Using \eqref{eq:ODE_newref1} to compute boundary terms, we get
\begin{equation} \label{eq:weak_hom}
\begin{aligned}
 -\int_{0}^{T} &\int_{\Omega(t)} \left[ u \cdot \partial_{t} \phi + (u\cdot \nabla)\,{\phi} \cdot u
- 2  
D(u) : D(\phi)\right] \, dx \, dt   \\[6pt]
 + \int_{0}^{T} &\int_{\Omega(t)} \left[ (u\cdot \nabla)\,s \cdot \phi +(s\cdot \nabla)\,u \cdot \phi  + (\widehat{f}[h] 
- \widehat{g}) \cdot \phi\right] \, dx \, dt  -\int_{A \setminus B_0}u_{0} \cdot \phi(0) 
\\[6pt]  
 & = mh'_{0} \, \ell(0) + J \theta'_{0} \, \alpha(0)+ \int_{0}^{T} \left(m {h'}{\ell'} -H_{h}(h,\theta) \, \ell + J{\theta'}{\alpha'}-H_{\theta}(h,\theta) \, \alpha \right) \, dt \,.
 \end{aligned}
\end{equation}
Furthermore, multiplying \eqref{eq:hom_pb1} by $u$, integrating by parts in space and using \eqref{eq:ODE_newref1} again, we formally infer the energy identity
\begin{equation} \label{eq_estimationenergie1}
\begin{split}
 \dfrac{{d}}{{d}t}
\left[E_{tot}(t)\right] + 2  \int_{\Omega(t)} |D(u)|^2 \, dx
= 
\int_{\Omega(t)} (\widehat{g} - \widehat{f}[h] - (u \cdot \nabla) s  ) \cdot u\,,
\end{split}
\end{equation}
where 
\begin{equation}\label{energy_original}
  E_{tot}(t)\doteq \frac12\left(\|u(t)\|^2_{L^2(\Omega(t))}  +m|h'(t)|^2+J|\theta'(t)|^2\right)+H(h(t),\theta(t))
\end{equation}
is the total energy. 
Since the right-hand side of this identity is bounded for $u \in L^2(\Omega(t))$ we observe that a weak formulation can be based on the velocity-field only that satisfies the classical $L^{\infty}_t L^2_x \cap L^2_t H^1_x$ regularity.

 For any initial data $((h_0,\theta_0),U_0=u_0+s(0,x),(h'_0,\theta'_0)) \in A_{1,e} \times \{L^2(\Omega_0)+v_p\} \times \mathbb R^2$ satisfying \eqref{eq:ic2}, we associate the initial total energy $E_0$ defined by 
\begin{equation}\label{E0}
E_{0}\doteq E_{kin}(0)+H(h_0,\theta_0) = \frac12\left(\|u_0\|^2_{L^2(\Omega(t))}  +m|h_0'|^2+J|\theta_0'|^2\right)+H(h_0,\theta_0). 
\end{equation}
We recall that the energy scales like $\rho/\mu^2$ in the adimensionalisation, as emphasized in \eqref{eq:energyscaling}, i.e. $E_0 = E^{adim}(0) = (\rho/\mu^2) E^{phys}(0)$.

\subsection{Definition of weak solutions to  \eqref{eq:evolution_pb-scaled}} \label{sec:funct}
In order to define the notion of a weak solution to problem \eqref{eq:evolution_pb-scaled}, we introduce the spaces
\begin{equation}
\begin{aligned}
\mathcal{H}(B)=&\{(v,\ell,\alpha)\in L^2(A)\times \mathbb{R}\times \mathbb{R}  \mid \nabla \cdot v=0 \text{ in } A, \ v\cdot\widehat{n}=0 \text{ on } \partial A, \\
& \ v=\ell\,\widehat{e}_2+\alpha\,(x_1,x_2-h)^{\perp}\text{ in } B\},\\[3pt]
\mathcal{V}(B)= &\{(v,\ell,\alpha)\in H^1_0(A)\times \mathbb{R}\times \mathbb{R}  \mid \nabla \cdot v=0 \text{ in } A,  \\ & v=\ell\,\widehat{e}_2+\alpha\,(x_1,x_2-h)^{\perp} \text{ in } B\},
\end{aligned}
\end{equation}
endowed, in view of Poincaré inequality, with the scalar products
\begin{equation}\label{eq:clean_scalarproducts}
\begin{split}
\langle z_1,z_2\rangle_{\mathcal{H}(B)}= &\int_{\Omega}u_1\cdot u_2\,dx+m \ell_1 \ell_2+J\alpha_1\alpha_2, \\ 
\langle z_1,z_2\rangle_{\mathcal{V}(B)}=& \int_{\Omega}\nabla u_1:\nabla u_2\,dx +2|B|\,\alpha_1\alpha_2 \, ,
 \end{split}
\end{equation}
where $z_i=(u_i,\ell_i,\alpha_i)$, $i \in \{ 1,2 \}$. 
We denote the norms induced by the scalar products in \eqref{eq:clean_scalarproducts} by $\|\cdot \|_{\mathcal{H}(B)}, \|\cdot \|_ {\mathcal{V}(B)}$. Observe that the right-hand side in $\eqref{eq:clean_scalarproducts}_2$ can be  replaced by an integral on $A$, i.e. 
$$\langle z_1,z_2\rangle_{\mathcal{V}(B)}= \int_{A}\nabla u_1:\nabla u_2\,dx $$
since $\nabla u_1 : \nabla u_2=2\alpha_1\alpha_2$ inside $B$. Recalling that $D(\cdot)$ denotes the symmetric part of the gradient, we have Korn identity
\begin{equation}\label{eq:splitting}
2\int_A D({u_1}):D( u_2)\,dx = \int_A\nabla{u_1}:\nabla  u_2\,dx \, .
\end{equation}
for all $(u_1, u_2) \in (H_0^1(A))^2$ satisfying $\nabla \cdot u_1=\nabla \cdot u_2=0.$   Such a formula is in particular valid for all $((u_1,\ell_1,\alpha_1),(u_2,\ell_2,\alpha_2)) \in \mathcal V(B)^2.$

We also recall that the Poincaré inequality on $A$ directly gives a trace inequality on $\partial B$, with a constant that does not depend on the position of $B$. Indeed, 
$$\int_A|v|^2\, dx \le \frac{4L^2}{\pi^2}\int_A |\nabla v|^2\, dx,$$
for every $v\in \mathcal{V}(B)$, so that we infer  a trace inequality for functions in $\mathcal{V}(B)$: 
\begin{equation} \label{eq_dissipationsolide}
\mu_1 |\ell|^2+\mu_2|\alpha|^2= \int_B|v|^2\, dx \le \int_A|v|^2\, dx \le \frac{4L^2}{\pi^2}\int_A |\nabla v|^2\, dx,
\end{equation}
with $\mu_1=|B_{eq}|=\pi e$ and $\mu_2=\int_{B_{eq}} |x|^2\, dx= \frac{\pi e}4(e^2+1)$. 
The important feature of this trace inequality is that it yields a universal constant that depends only on the geometry of the solid $B$ and the channel's height but not on the precise position of $B$. In particular, the crucial point that we will use later on is that the inequality can be used when $B$ is arbitrarily close to one of the boundaries of $A$.

Finally,  if $h,\theta : (0,T) \rightarrow \mathbb{R}$ are functions of time such that $(h(t), \theta(t)) \in A_{{1,e}}$ for every $t \in (0,T)$, we define the spaces
\begin{equation}
\begin{split}
L^p(0,T;\mathcal{V}(B(t)))=\bigg\{f:(0,T)\to \mathcal{V}(B(t)) \Big| 
\int_0^T\|f(\tau)\|^p_{\mathcal{V}(B(t))}\,d\tau<+\infty\bigg\} 
\end{split}
\end{equation}
for $1\le p < \infty$, and
\begin{equation}
L^\infty(0,T;\mathcal{H}(B(t)))=\bigg\{f:(0,T)\to \mathcal{H}(B(t)) \ \ \Big| 
\esssup_{\tau \in (0,T)}\|f(\tau)\|_{\mathcal{H}(B(t))}<+\infty\bigg\}\,.
\end{equation}
The norms are defined respectively by 
$$
\|f\|_{L^p(0,T;\mathcal{V}(B(t)))}=
\left(\int_0^T\|f(\tau)\|^p_{\mathcal{V}(B(t))}\,d\tau\right)^{1/p}
$$
for $1\le p < \infty$ and 
$$\|f\|_{L^\infty(0,T;\mathcal{H}(B(t)))}=
\esssup_{\tau \in (0,T)}\|f(\tau)\|_{\mathcal{H}(B(t))}.$$
We are now ready to give a definition of weak solution to \eqref{eq:evolution_pb-scaled}.
\begin{definition} \label{weakoriginal}
Let $T>0$.  Given  $(h_0,\theta_0) \in A_{{1,e}}$ and$(u_0,h'_0,\theta'_0)\in \mathcal{H}(B_0)$ we say that a triplet $(u,h,\theta)$ is a weak solution to problem \eqref{eq:hom_pb1}-\eqref{eq:ODE_newref1}-\eqref{ichat}  on $(0,T)$ if
\begin{itemize}
\item[(i)] $(h,\theta)\in W^{1,\infty}(0,T;\mathbb{R}^2)\cap \mathcal{C}([0,T]; {A_{1,e}})$,  with $h(0)= h_0$ and $\theta(0) = \theta_0$
\item[(ii)] $(u,h',\theta')\in L^2(0,T;\mathcal{V}(B(t)))\cap  L^\infty(0,T;\mathcal{H}(B(t)))$, 
\item[(iii)] for every $(\phi,\ell,\alpha) \in \mathcal{C}^1([0,T];\mathcal{V}(B(t)))$ such that $\phi(\cdot,T)=\ell(T)=\alpha(T)=0$, $(u,h,\theta)$ satisfies \eqref{eq:weak_hom}.

\end{itemize}
\end{definition}

 Some comments are in order.  We recall that we denote $B_0 = B(h_0,\theta_0)$ and $\Omega_0 = A \setminus B_0.$ The compatibility conditions \eqref{eq:ic2} for initial data are encoded in this definition by the statement 
$(u_0,h'_0,\theta'_0) \in \mathcal{H}(B_0).$ 
We point out that our definition of solution depends on the construction of $s$ and in particular on \black  the choice of $\varepsilon_0$ so that the various possible choices could give rise to {\em a priori} different weak solutions to \eqref{eq:evolution_pb-scaled}.   We do not intend to prove uniqueness in the present work. \black 
{ Since $(h,\theta) \in W^{1,\infty}(0,T)$, we can also extend the weak formulation to any $(\phi,\ell,\alpha)  \in W^{1,2}((0,T) \times A) \times H^1((0,T))$  such that 
\[\left\{
\begin{aligned}
\phi(t,x) &= \ell(t) + \alpha(t) (x-h(t)\widehat{e}_2)^{\bot}&&   \text{ in }B(t),\\ 
\phi & = 0&&   \text{ on }\partial A. 
\end{aligned}\right.
\]
}
We infer in particular for such test functions that for any $0\le t_1<t_2\le T$, 
\begin{equation} \label{eq:weak_gen}
\begin{aligned}
& 2  \int_{t_1}^{t_2} \int_{\Omega(t)} D(u) : D(\phi) \, dx \, dt -\int_{t_1}^{t_2} \left( \int_{\Omega(t)}{ \left[ u \cdot \partial_{t} \phi + (u\cdot \nabla)\,{\phi} \cdot u\right]} \, dx\right)dt   \\[6pt]
& + \int_{t_1}^{t_2} \int_{\Omega(t)} \left[ (u\cdot \nabla)\,s \cdot \phi +(s\cdot \nabla)\,u \cdot \phi { + \widehat{f}[h] \cdot \phi}\right] \, dx \, dt+\int_{\Omega(t_2)} u(t_2) \cdot \phi(t_2)dx 
 \\[6pt]
&- \int_{\Omega(t_1)} u(t_1) \cdot \phi(t_1)dx -\int_{t_1}^{t_2} \left( m {h'}{\ell'} -H_{h}(h,\theta) \, \ell + J{\theta'}{\alpha'}-H_{\theta}(h,\theta) \, \alpha \right) \, dt\\[6pt]
& + m \left(h'(t_2)\ell(t_2)  -  h'(t_1)\ell(t_1)\right) + J\left( \theta'(t_2) \alpha(t_2) -  \theta'(t_1) \alpha(t_1)\right)\\
&  =  \int_{t_1}^{t_2} \int_{\Omega(t)} \widehat{g} \cdot \phi \, dx \, dt.    
 \end{aligned}
\end{equation}


\section{Cauchy theory for weak solutions} \label{section:Cauchy}
\label{sec:existence}
The purpose of this section is to 
analyse weak solutions (up to collision) to the Cauchy Problem \eqref{eq:hom_pb1}-\eqref{eq:ODE_newref1}-\eqref{ichat}, in the sense of Definition \ref{weakoriginal}. 
Our first result in this section is the existence of at least one solution and a control of the growth  
of the total energy. As previously said, we do not tackle the uniqueness issue. 
\begin{theorem}\label{th:existence_weak_local}
There exists  $\lambda_0^{(0)}>0$, {depending on $L$ and $e$}, such that if $\lambda_0 \leq \lambda_0^{(0)}$ and  $(h_0,\theta_0) \in A_{1,e},$ $(u_0,{ h'_0,\theta'_0)\in\mathcal{H}(B_0)}$, then
\begin{itemize}
\item [(i)] there exist $T >0$ and at least one weak solution $(u,{h},{\theta})$ to problem \eqref{eq:hom_pb1}-\eqref{eq:ODE_newref1}-\eqref{ichat}  on $(0,T).$
\end{itemize}
Furthermore, for any $T >0$,   there exists a constant $C_{geo}$
such that for any weak solution $(u,h,\theta)$ on $(0,T)$, we have:
\begin{itemize}
\item[(ii)] $u\in \mathcal{C}( [0,T];L^2(A))$ and for all $0 \leq t_1 \leq t_2\leq T$, 
\begin{equation}\label{eq:energy_estimate_weak}
E_{tot}(t_2) -  E_{tot}(t_1) + \dfrac{1}{4} \int_{t_1}^{t_2} \|\nabla u(\tau)\|^2_{L^2(A)} \, d\tau  \leq  C_{geo} \lambda_0^2(t_2 -t_1) ;
\end{equation} 
\end{itemize}
Finally, we have the following blow-up alternative:
\begin{itemize}
\item[(iii)]
$\text{either }T=\infty\text{ or } T<\infty \text{ and }\lim\limits_{t \to T} (h(t),\theta(t))  \notin A_{1,e}.$
\end{itemize}
\end{theorem}
\begin{proof}
Let  $(h_0,\theta_0) \in A_{1,e}$ and $(u_0,h'_0,\theta'_0)\in\mathcal{H}(B_0).$ Since the implicit term $\widehat{f}[h]$ is Lipschitz, the existence of at least one weak solution $(u,{h},{\theta})$ of \eqref{eq:hom_pb1}-\eqref{eq:ODE_newref1}-\eqref{ichat}, i.e. Assertion (i), on some timespan $(0,T)$ -- depending decreasingly on the initial total energy $E_0$ and increasingly on the distance between $B_0$ and $\partial A$ -- 
can be proved following the standard methods in fluid-solid interaction problems, see e.g. \cite{conca,DejEste,GunzLeeSeregin,SMStarTuck}. 
Concerning Assertion (ii), the continuity in time of the weak solutions with values in $L^2(A)$ and the energy estimate \eqref{eq:energy_estimate_weak} can be proved following \cite[Theorem 2.1]{bravin}.  We only provide here a formal computation of the energy estimate to motivate the assumption that $\lambda_0$  should be sufficiently small.
Indeed the formal identity \eqref{eq_estimationenergie1} yields that, for all $0\leq t_1\leq t_2 \leq T$
\begin{multline} \label{eq_energ0}
\frac12\left( E_{tot}(t_2) -  E_{tot}(t_1)\right) +  \!\int_{t_1}^{t_2}\! \int_{A} |Du|^2dxd\tau 
\\
= \frac12 \int_{t_1}^{t_2}\! \int_{\Omega(\tau)}
 \left(\widehat{g}  - (u \cdot \nabla) s    - \widehat{f}[h]\right) \cdot u\, dxd\tau\,.
\end{multline}
This identity, based on multipying \eqref{eq:hom_pb1} by $u$, can be justified by adapting the regularization arguments of \cite[Theorem 2.1]{bravin}. Then, on the left-hand side, we have, by Korn identity, 
\[
2 \int_{A} |Du(\tau,x)|^2dx = 
\int_{A} |\nabla u(\tau,x)|^2dx\,. 
\]
To estimate the right-hand side of \eqref{eq_energ0}, we use Poincaré inequality, Lemma \ref{lemma:extension} and trace inequality \eqref{eq_dissipationsolide} to obtain successively 
\[
\begin{aligned}
\left| \int_{t_1}^{t_2} \int_{\Omega(\tau)} \widehat{g} \cdot u\,dxd\tau  \right| 
& \leq 
\dfrac{1}{4} \int_{t_1}^{t_2} \int_{A} |\nabla u|^2dxd\tau 
+ {{C_{geo}}}\int_{t_1}^{t_2} \int_{A}|\widehat{g}|^2dxd\tau \\[6pt]
&\leq  
\dfrac{1}{4} \int_{t_1}^{t_2} \int_{A} |\nabla u|^2dxd\tau 
+ {{C_{geo}}\, \lambda_0^2} \left(1 +{\lambda_0}\right)^2  (t_2 -t_1) \, ,
\\[6pt]
\end{aligned}
\]
\[
\begin{aligned}
\left| \int_{t_1}^{t_2} \int_{\Omega(\tau)} ( u \cdot \nabla) s \cdot u\,dxd\tau\right|
& \leq { C_{geo}}\|s\|_{W^{1,\infty}(A)} \int_{t_1}^{t_2} \int_{A} |\nabla u|^2dxd\tau \\[6pt]
&
\leq {C_{geo}}\, \lambda_0\int_{t_1}^{t_2} \int_{A} |\nabla u|^2dxd\tau \, ,
\end{aligned}
\]
\[
\begin{aligned}
\left| \int_{t_1}^{t_2} \int_{\Omega(\tau)} \widehat{f}[h] \cdot u\,dxd\tau 
\right|
& \leq  {C_{geo}}\lambda_0 \int_{t_1}^{t_2}|h'(\tau)|\left(\int_{A} |\nabla u|^2dx\right)^{\frac12}\, d\tau\\[6pt]
&\leq  {C_{geo}}\lambda_0\int_{t_1}^{t_2} \int_{A} |\nabla u|^2dxd\tau \, .\\ 
 \end{aligned}
\]
Using these inequalities and choosing  $\lambda_0$ small enough with respect to $C_{geo}$
, we obtain that the right-hand side of \eqref{eq_energ0} is bounded by
\[
\dfrac{3}{8} \int_{t_1}^{t_2} \int_{A} |\nabla u|^2dxd\tau + {{C_{geo}}\, \lambda_0^2} \left(1 +{\lambda_0}\right)^2 (t_2 -t_1)  \,,
\]
 yielding the estimate \eqref{eq:energy_estimate_weak}.

{ The blow-up alternative is standard and follows from a concatenation method based on the remark that the $L^2$-norm of $(u(t),h'(t),\theta'(t))$ may not blow-up in finite time.}
 \end{proof}

\begin{remark}\label{rem-small}
We remark that, in the above proof, the smallness condition $\lambda_0 \leq \lambda_0^{(0)}$ means physically that $\frac{\mathcal P_0\rho \mathcal{L}^2}{2\mu^2d}\leq \lambda_0^{(0)}$, i.e. the original pressure drop $\mathcal P_0$ should be small compared to the ratio $\mu^2/(\rho d L^2)$. This smallness condition has for only purpose to ensure source terms can be bounded by dissipation in the  energy estimate to yield \eqref{eq_energ0}.  The method that we sketched above yields existence until contact for any intensity of the Poiseuille flow.  However, with no restriction on $\lambda_0$ one must expect that $E_{tot}$ grows exponentially with time.  
\end{remark}
 
From now on, we assume tacitly that $\lambda_0 \leq \lambda_0^{(0)}$. We complement our analysis with a further energy estimate for weak solutions. 
Indeed,  in absence of a potential energy term, the estimate \eqref{eq:energy_estimate_weak} predicts that the kinetic energy of the system, namely 
\[
E_{kin}(t) \doteq\frac12\left(\|u(t)\|^2_{L^2(\Omega(t))}  +m|h'(t)|^2+J|\theta'(t)|^2\right),
\]
as well as its dissipation, grow at most linearly in time.  However, Poincaré inequality implies the dissipation is morally larger than the time integral of the kinetic energy.  So, we may expect  that as long as the potential energy is under control, the kinetic energy remains bounded with time while the dissipation grows linearly. This is the motivation of the next proposition. 
\begin{proposition} \label{prop_Ekin}
Let the assumption of Theorem \ref{th:existence_weak_local} be in force and assume $\lambda_0 \leq \lambda_0^{(0)}$.  There exists a constant $C_{dyn}$  depending  on $e$, $L,$ $m,$ and $J$ such that, for all $t \in (0,T)$, 
\begin{equation} \label{eq_Ekin}
E_{kin}(t) \leq E_{kin}(0) + C_{dyn}\left({\lambda}_0^2 + 
{\max_{s\in(0,t)}{|\nabla H(h(s),\theta(s))|^2}}   \right).
\end{equation}
\end{proposition}
\begin{proof}
We start by rewriting the energy inequality \eqref{eq:energy_estimate_weak} in the following way:
\[
\begin{aligned}
&E_{kin}(t_2) - E_{kin}(t_1)  + \dfrac{1}{4} \int_{t_1}^{t_2} \|\nabla u(\tau)\|^2_{L^2(A)}d\tau \\
 & +  \int_{t_1}^{t_2} \left(h'(\tau)  H_h(h(\tau),\theta(\tau)) \right.  \left. +  \theta'(\tau) H_\theta(h(\tau),\theta(\tau))\right)d\tau   \leq  {C_{geo}} \lambda_0^2 (t_2 -t_1)\,,
 \end{aligned} 
 \]
where $0 \leq t_1 \leq t_2 \leq T$.
By inequality \eqref{eq_dissipationsolide}, we deduce that for all $0 \leq t_1 \leq t_2 \leq T$, 
\begin{equation}\label{eq:yae}
\begin{split}
 E_{kin}(t_2) -E_{kin}(t_1)  & + \dfrac{1}{8} \int_{t_1}^{t_2} \|\nabla u(\tau)\|^2_{L^2(A)}d\tau  \\
&   \leq  {C_{g}}\left(\lambda_0^2(t_2-t_1)+\int_{t_1}^{t_2}{|\nabla H(h(\tau),\theta(\tau))|^2}d\tau \right).  
\end{split}
 \end{equation}

The inequality \eqref{eq_dissipationsolide} again together with Poincaré inequality imply that there exists a constant  ${C_{dyn}}$  depending on  $e$, $L$, $m$ and $J$  so that
\[
{C_{dyn}}  E_{kin}(t) \leq  \|\nabla u(t)\|^2_{L^2(A)}\,, \ \forall \, t \in (0,T).
\]
Eventually, we infer, for $0 \leq t_1 \leq t_2 \leq T$, that
\[
\begin{split}
E_{kin}(t_2) - E_{kin}(t_1)\,  + \, & \dfrac{{C_{dyn}} }{16} \int_{t_1}^{t_2} E_{kin}(\tau)d\tau \\ & \leq
  {C_{geo}}\left(\lambda_0^2(t_2-t_1)+\int_{t_1}^{t_2}{|\nabla H(h(\tau),\theta(\tau))|^2}d\tau \right).
\end{split}
\]
Since, by Theorem \ref{th:existence_weak_local}, $E_{kin} \in \mathcal{C}([0,T]),$ we may apply Gr\"onwall lemma to deduce that
\[
E_{kin}(t) \leq E_{kin}(0) {\rm e}^{-{C_{dyn}} t /16} +
{C_{geo}}\left( {\lambda}_0^2 + {\max_{s\in(0,t)}{|\nabla H(h(s),\theta(s))|^2}} \right)  \int_0^{t} {\rm e}^{- {C_{dyn}} (t-s) /16} ds,
\]
and the bound \eqref{eq_Ekin} follows.
\end{proof}

Observe that, as well known, without any force, the kinetic energy remains bounded along any trajectory. In presence of the elastic restoring forces $H_h$ and $H_\theta$, this remains true if $\theta(t)$ is uniformly bounded in time, but false in general. The restoring force $H_h$ plays a less critical role in this analysis since $h$ is a priori bounded due to the presence of the walls of the channel. A closer inspection of the behaviour of $E_{kin}$ shows that on intervals $(t_1,t_2)$ where $|\theta|$ increases, $E_{kin}(t_2)-E_{kin}(t_1)$ is a priori bounded but an increase of kinetic energy is expected on intervals where $|\theta|$ decreases. 

\par 

In order to bound the kinetic energy over time, we will restrict our attention to trajectories up to collision, either with a wall, or with a fictive chosen bound on the angular displacement. Showing that collisions do not occur will be our main objective from Section \ref{sec:distance1} and Section \ref{sec:distance2}. 
The choice of the fictitious bound for $\theta$ is based on the coercivity assumption \eqref{coercivity} on the potential $H$. This assumption implies that there exists $\alpha_0\in \mathbb R^+$ such that,
\begin{equation}\label{coercivity-bis}
H(h,\theta)\le 3E_0+\overline{\varrho}_M\Rightarrow |\theta|\le \alpha_0\,,
\end{equation}%
whatever $(h,\theta)\in A_{1,e}$.
We then introduce $\theta_M=\alpha_0+1$. This precise choice is motivated by the forthcoming Theorem \ref{cor_estimation1} where we shall prove that when $\lambda_0$ is small enough, actually $H(h(t),\theta(t))\le 3E_0+\overline{\varrho}_M$ on any time interval $(0,T)$ where $(h(t),\theta(t))\in A_{1,e}$ and $|\theta(t)|<\theta_M$. 

With this chosen $\theta_M$, since initially $H(h_0,\theta_0)\le E_0$, we infer that $|\theta_0|\le \alpha_0$ and we can define
\begin{equation}\label{T_M}
T_M\doteq\sup\{t\ge 0 : (h(t),\theta(t))\in A_{1,e}\text{ and } |\theta(t)| < \theta_M\}\in (0,+\infty]\,.
\end{equation}
It then directly follows from Proposition \ref{prop_Ekin} that 
\begin{equation}\label{E_kin^M}
\max_{t\in(0,T_M)}E_{kin}(t) \leq E_{kin}(0) 
+ {C_{dyn}}\left(\lambda_0^2 +S_M \right),
\end{equation}
where 
\begin{equation}\label{F_M}
S_M\doteq \max\{|\nabla H(h,\theta)|^2: (h,\theta)\in A_{1,e} \text{ and } |\theta|\le\theta_M\}\,.
\end{equation}
{The assumption \eqref{coercivity} implies $S_M$ is monotone increasing with respect to $E_0$.}
In the sequel, we will often refer to
\begin{equation}\label{Ekin_max}
E_{kin}^{M} \doteq  E_{kin}(0) + C_{dyn}\left({\lambda^2_0} + S_M   \right)\le  E_0 + C_{dyn}\left({\lambda^2_0} + S_M   \right),
\end{equation}
that is the maximal possible kinetic energy on $(0,T_M)$. Observe that $\theta_M$ is fixed by now on and depends only (increasingly) on the initial energy $E_0$. 
Finally, observe that $E_{kin}^{M}$ is bounded by an increasing function of the initial total energy and $\lambda_0$.
We will drop the subscript $M$ when using the constants $\overline{\varrho}_M$ and $\overline{\varpi}_M$ in the sequel.

\begin{remark}
Since the dimensionless potential $H$, the initial total energy $E_0$ and $\overline{\varrho}_M$ are all proportional to $\rho/\mu^2$, we observe that, if we take initial conditions in the physical variables of the form given in \eqref{eq:ic}, then $\theta_M$ does not depend on $\rho$ and $\mu^2$.
\end{remark}

\black 

\section{Distance estimate -- Preliminary results}\label{sec:distance1}
Estimating the distance between $B(t)$ and $\partial A$ requires to compute the repulsion force exerted by the boundary $\partial A$ on $B(t)$ through the fluid. 
To this aim, it is by now well documented, see \cite{Hillairet2007,HillTaka}, that one method is to use as a dual problem  the stationary Stokes system in $A \setminus \overline{B}$ with a suitable boundary condition on  $\partial B$.  
In this section, we do the preliminary work, namely we derive fine properties of these stationary Stokes solutions whatever the orientation of $B$ when it is close to the lower part of $\partial A$.  These are based first on a precise description of the gap between $B$ and $\partial A.$ We complement this section with several ingredients that will be required for computing distance estimates in the following section. 

\subsection{Description of the gap.} \label{sec_gapdescription}

Firstly, we recall properties of $\partial B$ that result from the fact that $B$ is convex with a smooth, strictly convex and compact boundary.  We remark that we can parametrize the boundary of $B$ around any point $X  \in \partial B$ by a smooth graph $\gamma_{X}$ on some length $4\lambda_*.$ Namely, denoting by $(\tau_X,n_X) \in \mathbb S^1 \times \mathbb S^1$ the local Frenet basis in $X \in \partial B$, we can fix a constant $r >0$ such that for every point $X\in \partial B$, we have
\[
(x_1,x_2) \in \partial B \cap B_{\mathbb R^2}(X,r) \quad \Leftrightarrow 
\quad
(x_1,x_2) = X+ t \tau_X + \gamma_X(t) n_X \quad \text{for some } t \in (-2\lambda_*,2\lambda_*) \,.
\]
Since $\partial B$ is compact, we emphasize that we can choose a similar radius $r$ and length $\lambda_*$ for all points $X\in \partial B$.
Since $\partial B$ is smooth both mappings $\partial B \mapsto \mathbb S^1 \times \mathbb S^1 : X \mapsto (\tau_X,n_X)$ and $\partial B \mapsto \mathcal{C}^m([-2\lambda_*,2\lambda_*]): X \mapsto \gamma_X$ (whatever $m \in \mathbb N^*$) are smooth. 
\begin{figure}[H]
    \centering
    \includegraphics[scale=0.8]{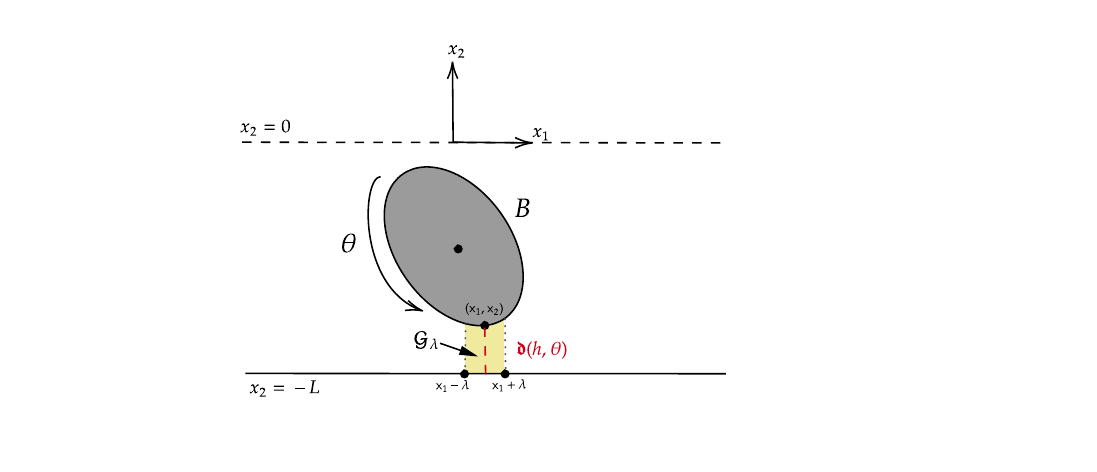}
    \vspace{-4mm}
    \caption{Representation of the gap $   \mathcal{G}_\lambda$ between $B$ and $\partial A$, when $h<0$.}
    \label{fig:contact}
\end{figure}
\black
Let $(h,\theta) \in A_{1,e}$.  We denote by {$X[\theta]$} the point on $\partial B_{eq}$ such that $h \widehat{e}_2 + Q(\theta){X[\theta]}$ achieves the (minimal) distance between $\partial B$  and the lower part of $\partial A$, i.e. the wall $x_2=-L$ . We note that {$X[\theta]$} is uniquely defined since $B$ is strictly convex and we denote $(\mathsf x_1[\theta],\mathsf x_2[\theta]) \doteq Q(\theta){X[\theta]}$.   We set also $\gamma [\theta] \doteq \gamma_{X[\theta]}$ and  we have the local expansion:
\begin{equation} \label{eq_expansiongamma}
\gamma [\theta](\tau) = \mathsf{x}_2[{\theta}] + \kappa_{2}[\theta] \tau^2 + \kappa_3[\theta] \tau^3 + \kappa_4[\theta] \tau^4 +  \varepsilon[\theta](\tau)|\tau|^4 \qquad \forall \, \tau \in [-2\lambda_*,2\lambda_*]\,,
\end{equation}
where $\theta \mapsto (\mathsf x_2[\theta],\kappa_2[\theta],\kappa_3[\theta],\kappa_4[\theta]) \in \mathbb R^4$  is smooth and
\[
|\varepsilon[\theta](\tau)| \leq \varepsilon(\tau)\,, \;\; \forall \, \theta \in [0,2\pi)  \quad \text{where} \quad \lim_{\tau \to 0} |\varepsilon(\tau)| = 0 \, .
\]
In particular, we have:
\begin{equation} \label{eq_expansiongamma_der}
\partial_{\theta} \gamma[\theta](\tau) = \partial_{\theta} \mathsf{x}_2[\theta]+ \partial_{\theta} \kappa_2[\theta] \tau^2 + \partial_{\theta} \kappa_3[\theta] \tau^3 + \varepsilon_{1}[\theta](\tau) |\tau|^3 \qquad \forall \,  \tau \in [-2\lambda_{*},2\lambda_{*}]\,,
\end{equation}
with similar considerations for $\varepsilon_{1}[\theta]$ as for $\varepsilon[\theta].$ {Furthermore, there exist four constants $(c_i^{(2)})_{1\leq i\leq 4}$ such that:
\begin{equation} \label{constants}
\left\{
\begin{aligned}
	c_1^{(2)}\tau^2 & \le \gamma[\theta](\tau)-{\sf{x}_2}[\theta] \le c_2^{(2)}\tau^2\,, \\
	c_3^{(2)}\tau^2 & \le \partial_\theta \gamma[\theta](\tau)-\partial_\theta{\sf{x}_2}[\theta] \le c_4^{(2)}\tau^2\,,
\end{aligned}	
\right.
	\quad \forall \, \tau\in[-2\lambda_*,2\lambda_*]\,. 
\end{equation}
Finally, since $B$ is open and $\partial B$ is strictly convex, we observe that $ \mathsf{x}_2[\theta],\kappa_2[\theta]$ belong to a bounded interval of $(-\infty,0)$ and $(0,\infty)$ respectively whatever $\theta$.  Since $\partial B$ is compact, there exist $0 <d_{min} < d_{max}$ and $0 < \kappa_{2}^{min} \leq \kappa_2^{max}$  so that $\mathsf x_2(\mathbb R)= [-d_{max},-d_{min}]$ and $\kappa_2(\mathbb R) = [\kappa_2^{min},\kappa_2^{max}].$
We refer to the Appendix \ref{app_difftheta} for an explicit description of $\gamma[\theta]$ exploiting that $B$ is an ellipse and providing explicit values for $d_{min},d_{max},\kappa_2^{min},\kappa_2^{max}$.}

\medskip

From now on, we drop the $\theta$-dependencies on all these notations for legibility. However, it is important to keep in mind that these quantities are related to the orientation of $B$ through the angle $\theta$ that changes with time. 
We introduce also the splitting of the fluid domain $\Omega \doteq A \setminus B. $ For $\lambda \in (0,\lambda_*)$
 we set
 \[
 \mathcal G_{\lambda} \doteq \{ (x_1,x_2) \in A \ | \ |x_1-\mathsf x_1| < \lambda \, , \quad -L < x_2 < h +\gamma(x_1 -\mathsf x_1) \, \} \, .
 \]
This domain represents the gap between $B$ and $\partial A$
below $B$, see Figure \ref{fig:contact}. We note that,  when $h<0,$ there exists a strictly positive distance 
$d_0[\lambda]$ depending only on $\lambda$ such that
\[
{dist}(\partial B \setminus \overline{\mathcal G_{\lambda}}, \partial A) > d_0[\lambda] \, .
\]
In particular, we set {$d_* \doteq d_0[\lambda_*]\,.$}

\medskip

For later purpose, we need to find relations between $\mathsf x_2$ and $\kappa_2$ and between $(\partial_{\theta} \kappa_2,\partial_{\theta} \mathsf x_2)$ and $\kappa_3$. We state the required properties in the following lemma. {Its proof exploits that $B$ is an ellipse and is postponed to the Appendix \ref{sec:Proof-appA}}:
\begin{lemma} \label{lem_geom}
The following statements hold true:
\begin{itemize}
\item[(i)] 
there exists a diffeomorphism $\mathsf X_2 : [\kappa_2^{min},\kappa_2^{max}] \to {[-d_{max},-d_{min}]}$ with inverse $\mathcal K_2$ such that:
\begin{equation}\label{eq_K_2etY}
\kappa_2 = \mathcal K_2(\mathsf x_2 ) \quad \text{and} \quad \mathsf x_2 = \mathsf X_2(\kappa_2) \qquad \forall \, {\theta \in \mathbb R} \, .
\end{equation}
\item[(ii)] for all ${\theta \in \mathbb R}$, we have
\begin{equation}\label{eq_kappa_2}
\kappa_3 = -\kappa_2 \mathsf x_2  \partial_{\theta} \left[\dfrac{\kappa_2}{\mathsf x_2}\right] \, . 
\end{equation}
\item[(iii)] there exists a $\mathcal{C}^1$-mapping $\mathcal K_3 : [\kappa_2^{min},\kappa_2^{max}] \to \mathbb R$ such that
\[
\kappa_3[\theta] = \partial_{\theta} \mathcal K_3(\kappa_2[\theta]). 
\]
\end{itemize}
\end{lemma}

\subsection{Construction of test functions.} \label{sec:test-function}

Inspired by the computations in \cite{Hillairet2007,Sokhna}, our next goal is to expand Stokes solutions when the distance between $B(t)$ and $\partial A$ is small.  It will turn out that since the shape of the gap $\mathcal G_{\lambda_*}$ changes with time (because of the distance and also of the curvature of $\partial B(t)$),  we will need to have at hand an expansion of the solution of the Stokes problem for a whole family of boundary conditions. 
Consequently, we proceed with computing a profile for solutions to
 \begin{equation}\label{eq_stokesprofile}
\left\{
\begin{aligned}
 - \Delta v + \nabla q &= 0  && \text{ in $\Omega$} \, , \\
  \nabla \cdot v & = 0 && \text{ in $\Omega$} \, , \\
  v &=0 && \text{ on $\partial A$}\,,  \\
  v &  = v_*&&  \text{ on  $\partial B$}\,,\\
  \lim_{|x_1|\to\infty} v(x_1,x_2) &= 0 && \forall\, x_{2} \in [-L,L] \, , \\ 
\end{aligned}
\right. 
\end{equation}
for a specific class of boundary conditions $v_*$ and 
in the asymptotic regime ${\rm dist}(B,\partial A) \doteq \mathfrak d(h,\theta)  << 1$. We observe that when $h <0$ the distance is achieved below the ellipse and $\mathfrak d(h,\theta) = h+\mathsf x_2 + L$, see Figure \ref{fig:contact}.    In the subsequent computations, we will apply our construction to the following three boundary conditions :
\begin{equation} \label{threebc}
v_{*}^{\bot}(x) = \widehat{e}_2 \, , \qquad v_*^{||}(x) = \widehat{e}_1  \, , \qquad v_*^{\circlearrowleft}(x) = (x- (\mathsf x+h \widehat{e}_2))^{\bot} \qquad \forall x \in \partial B \, ,
\end{equation}
mimicking translation, respectively parallel and orthogonal to $\partial A$, and rotation of $B$ (computed with respect to the point $h\widehat{e}_2 + \mathsf x$).
We remark that, in these three cases, there exists a stream function $\psi_* \in \mathcal{C}^{\infty}(\bar{A})$ such that $v_* = \nabla^{\bot}\psi_* = (-\partial_2 \psi_*, \partial_1 \psi_*)$, namely 
\begin{equation}\label{eq:stream}
\psi_{*}^{\bot}(x)=x_1-{\sf x}_1,\ \psi_*^{||}(x)={\sf x}_2+h-x_2,\text{ and } \psi_*^{\circlearrowleft}(x) =\frac12(x_1-{\sf x}_1)^2+\frac12(x_2-\mathsf x_2-h)^2. 
\end{equation}
 We have chosen {streamfunctions} that are zero at the central point of the gap $\mathcal G_{\lambda_*}$, that is where $(x_1,x_2)=(\mathsf{x_1}, \mathsf{x_2}+h)$.
We therefore present the general method under this further assumption and then its applications to the three specific  cases.  

\medskip

Under the assumption that $v_* = \nabla^{\bot}\psi_*,$ we recall first that there is a unique solution (up to a constant, regarding the pressure) to \eqref{eq_stokesprofile} whose velocity field {(extended by $\nabla^{\bot} \psi_*$ on $B$)} is the vector field achieving 
\begin{equation} \label{min1}
\min_{v \in H^1_0(A)}
\left\{ \int_{\Omega} |\nabla v|^2 \ \bigg| \ \nabla \cdot {v} = 0 \,, \ {v} = \nabla^{\bot} \psi_* \text{ on $B$} \right\}\,. 
\end{equation}
Let $w \in H^1_0(A)$ be a competitor for this minimization problem.  Since $A$ is simply connected, we can find $\psi \in H^2(A)$ such that ${w}= \nabla^{\bot} \psi$ in $A$. Since $w$ prescribes no flux on vertical lines in $A$, we may also normalize $\psi$ by assuming that $\psi = 0$ on $\partial A.$ Doing so, solving \eqref{min1} amounts to solving the problem
\[
\min_{\psi \in H^2_0(A)}
\left\{ \int_{\Omega } |\nabla^2 \psi|^2 \ \bigg| \ \nabla \psi = \nabla \psi_*  \text{ on $B$} \right\} \,.
\] 
We propose to compute an approximate minimum by minimizing the vertical derivatives in the gap with the same boundary conditions:
\begin{equation} \label{eq:min2}
\min_{\psi \in H^2(\mathcal G_{\lambda_*})} \left\{ \int_{\mathcal G_{\lambda_*}} |\partial_{22} \psi|^2 \ \bigg| \ \nabla \psi = \nabla \psi_* \ \text{ on $\partial B$ } , \quad \psi = \partial_2 \psi = 0 \ \text{ on $\partial A$} \right\} \, .
\end{equation}
It turns out that the minimum in \eqref{eq:min2} can be explicitly computed. Before giving the explicit form of the minimizer, we introduce 
the function $\tau \in (-2\lambda_*,2\lambda_*) \mapsto \mathfrak d + \gamma[\theta](\tau) - \mathsf x_2,$ which measures the vertical distance between $\partial A$ and $\partial B$ as a function of the horizontal distance $\tau$ from $(\mathsf x_1,\mathsf x_2),$  and the vertical-distance ratio,  defined for every $x$ in the gap $\mathcal G_{\lambda_*}$ by:
\begin{equation*}
r(x_1,x_2) = \dfrac{x_2 + L }{\mathfrak d + \gamma(x_1 - \mathsf x_1)-\mathsf x_2}. 
\end{equation*}
Following \cite{GeVaHill1}, we observe that the minimization problem \eqref{eq:min2}, under the change of variables 
	\begin{equation}\label{change_variables}
		(x_1,x_2)\mapsto (\tau(x_1,x_2), r(x_1,x_2)) \qquad \forall\,(x_1,x_2)\in \mathcal{G}_{\lambda_*}\,, 
	\end{equation}
	 amounts to find, for every $\tau\in (-2\lambda_*,2\lambda_*)$, the minimizer of the functional 
	$$
	\int_0^1|\psi''_\tau(r)|^2\cdot \frac{1}{(\mathfrak{d}+\gamma (\tau)-{\sf x_2})^3}\,dr\,.
	$$
This is a one-dimensional minimization problem, whose associated Euler-Lagrange  equation yields $\psi''''_\tau=0$.} 

{Applying that $\psi = \partial_2 \psi = 0$ on $\partial A$, t}he minimizer in \eqref{eq:min2} can then be written as
\begin{equation}\label{eq:optimizer}
\begin{split}
{\psi}_{opt}(x) = \ \psi_1(x_1 - \mathsf{x}_1) P^{opt}_1 \left( r(x_1,x_2) \right) 
+  \psi_2(x_1 - \mathsf{x}_1)P^{opt}_2 \left(  r(x_1,x_2)  \right), \ \forall x \in \mathcal{G}_{\lambda_*} \, ,
\end{split} 
\end{equation}
where $P^{opt}_1,P^{opt}_2$ are the polynomia defined for $s \in [0,1]$ by
\begin{equation*}
P^{opt}_1(s) = 3s^2 - 2s^3 \quad \text{and} \quad P^{opt}_2(s) = s^2(s-1),
\end{equation*}
$\psi_1,\psi_2$ are {chosen so that $\nabla \psi = \nabla \psi_*$ on $\partial B$}, namely: 
\begin{equation}
\left\{ 
\begin{array}{l}
\psi_1(\tau)  = \psi_*(\mathsf x_1 + \tau,  h + \gamma[\theta](\tau))  - c_*,\\ \medskip
\psi_2(\tau)  = \partial_2 \psi_*(\mathsf x_1 + \tau,  h + \gamma[\theta](\tau)) (\mathfrak d + \gamma[\theta](\tau) - \mathsf x_2)\,.
\end{array}\right.
\end{equation}
{With respect to computations restricted to simpler motions and fully symmetric geometries, as for instance in \cite{Hillairet2007}}, we emphasize the appearance of a ``new'' constant $c_* \in \mathbb{R}$. Optimizing \eqref{eq:min2} with respect to this constant $c_*$ fixes its value through the optimality condition
\begin{equation} \label{eq_d222psinul}
\int_{\mathsf x_1-\lambda_{*}}^{\mathsf x_1+ \lambda_{*}} \partial_{222} \psi_{opt}(x_1,-L) \, d x_1 =0 \, .
\end{equation}

Now, we propose the following construction of an approximate solution $\tilde{v}$ to problem \eqref{eq_stokesprofile}.  We fix $\zeta \in \mathcal{C}^{\infty}(\mathbb{R})$ such that {$\mathds{1}_{[-1,1]} \leq \zeta \leq \mathds{1}_{[-2,2]}$},
and we set, for all $x\in A$, 
\begin{equation}\label{eq:tilde_psi}
\tilde{\psi}(x)
= 
\left\{
\begin{aligned}
& \zeta\left( \dfrac{(x_1 - \mathsf x_1)}{\lambda_*} \right)  \psi_{opt}(x) \\
& + \left(1-\zeta\left( \dfrac{(x_1 -\mathsf x_1) }{\lambda_*} \right)\right) \zeta\left( \dfrac{{\rm dist}({x},B)}{d_{*}}\right) \left( \psi_*(x) - c_* \right) \; && \forall x \in \mathcal G_{2\lambda_*}\,, \\[6pt]
&  \zeta\left( \dfrac{{\rm dist}({x},B)}{d_{*}}\right) \left( \psi_*(x) - c_* \right) 
\; &&  \forall x \in \Omega \setminus \mathcal G_{2\lambda_*}\,, \\[6pt]
&  \psi_*(x) - c_* && \forall x \in B \, .
\end{aligned}
\right.
 \end{equation}
Standard computations show that $\tilde{v} \doteq \nabla^{\bot} \tilde{\psi}$ is such that
\[
  \tilde{{v}} \in H^1_0(A) \, , \qquad \nabla \cdot \tilde{v} = 0 \quad \text{in} \ A \, , \qquad  \tilde{v} = \nabla^{\bot} \psi_* \quad \text{on} \ B \,.
\]
We postpone a sharper description of this construction to Appendix \ref{app_asymptotics}.
For the three different boundary conditions given in \eqref{threebc}, we deduce from \eqref{eq:stream} and \eqref{eq:optimizer} the following expressions, for $x=(x_1,x_2)\in\mathcal G_{2\lambda_*}$, of the stream functions $\psi^\perp_{opt}$, $\psi^\parallel_{opt}$,$\psi^\circlearrowleft_{opt}$ of the respective optimizer in \eqref{eq:min2}:
\begin{equation} \label{psising}
\begin{aligned}
\psi^\perp_{opt}(x_1,x_2) & \doteq \left(x_1 - {\sf x}_1\black- c_*^{\bot} \right)\, \black P_1^{opt} \left(r(x_1,x_2) \right)\,, \\[6pt]
\psi^\parallel_{opt}(x_1,x_2) & \doteq  -(\gamma(x_1 - \mathsf{x}_1) -{\sf x}_2+ c_{*}^{\parallel})\, P_1^{opt}  \left(r(x_1,x_2) \right)  \\
& \hspace{5mm} - \left( \mathfrak{d} + \gamma(x_1 - \mathsf{x}_1) - \mathsf x_2\right)\, P_{2}^{opt}   \left(r(x_1,x_2) \right) \, , \\[6pt]
\psi^\circlearrowleft_{opt}(x_1,x_2) & \doteq  \left(\frac{1}{2} (x_1 - \mathsf{x}_1)^2 + \frac{1}{2}(\gamma(x_1-\mathsf x_1) - \mathsf x_2)^2 - c_{*}^{\circlearrowleft}\right)\, P_1^{opt}  \left(r(x_1,x_2) \right) \\
& \hspace{5mm} - (\gamma(x_1-\mathsf x_1) - \mathsf x_2) (\mathfrak{d} + \gamma(x_1 - \mathsf x_1) - \mathsf x_2)\, P_2^{opt}  \left(r(x_1,x_2) \right)\,.
\end{aligned}
\end{equation}
We refer to Lemma \ref{lem_c*} in Appendix \ref{app_asymptotics} for an asymptotic expansion of the values $c_*^{\bot}$, $c_*^{||}$, $c_{*}^{\circlearrowleft}$ when $\mathfrak d << 1$. 
The streamfunctions $\tilde\psi^\perp$, $\tilde\psi^\parallel$,$\tilde\psi^\circlearrowleft$ are then computed using \eqref{eq:tilde_psi} with the respective expression of $\psi_*$ as in \eqref{eq:stream}. The associated approximate Stokes solutions will be denoted by $\tilde{v}^{\bot},\tilde{v}^{\parallel}$ and $\tilde{v}^{\circlearrowleft}.$ 

For further references, we compute, using the change of variables in \eqref{change_variables} and the above construction of $\psi_{opt}^{\bot},$ 
\begin{align}
\int_{\mathcal G_{\lambda_*}} |\partial_{22} \psi_{opt}^{\bot} |^2 \, dx
&= \int_{-\lambda_*}^{\lambda_*} \dfrac{(\tau - c_{*}^{\bot})^2}{(\mathfrak d (t)+ \gamma[\theta](\tau) - \mathsf x_2)^3} \, {d}\tau\int_0^1 |\partial_{zz}P_{1}^{opt}(z)|^2{d}z\,  \\[6pt]&=  12  \int_{-\lambda_*}^{\lambda_*} \dfrac{(\tau- c_*^{\bot})^2}{(\mathfrak d(t) + \gamma[\theta(t)](\tau) - \mathsf x_2[\theta(t)])^3} \, {d}\tau, \label{eq:intpsiopt}
\end{align}
{and we observe that the approximations are singular in the limit $\mathfrak d << 1$ only at the $H^1$ level and that the singularity is concentrated in the gap. Namely, 
there exists a constant $C_{geo}\,>0$ (therefore depending only on the geometry of $B$) independent of $\mathfrak{d}$ and $\theta$ such that 
\begin{equation}\label{stime_test}
\left\{
\begin{aligned}
\|\tilde{v}^{\bot}\|_{L^2(A)}+ \|\tilde{v}^{||}\|_{L^2(A)} + \|\tilde{v}^{\circlearrowleft} \|_{L^2(A)} & \le C_{geo}\, \\
	\|\tilde{v}^{\bot} \|_{\mathcal{C}^2_b(\Omega\setminus \mathcal{G}_{\lambda_*})}+\|\tilde{v}^{||} \|_{\mathcal{C}^2_b(\Omega\setminus \mathcal{G}_{\lambda_*})}+\|\tilde{v}^{\circlearrowleft} \|_{\mathcal{C}^2_b(\Omega\setminus \mathcal{G}_{\lambda_*})}& \le C_{geo}\,	\,,
	\end{aligned}
\right.
\end{equation}  
where $\mathcal{C}^2_b(\Omega\setminus\mathcal{G}_{\lambda_*})$ denotes the space of all bounded continuous functions on $\Omega\setminus\mathcal{G}_{\lambda_*}$.}
\subsection{A refined trace inequality when $B$ is close to $\partial A$}
In our computations, we shall use a suitable test function to prove a minimal distance estimate.  A first {possible way to tackle this issue could be} to relate the velocities of the ellipse to the {fluid dissipation -- {\em i.e.} the $H^1$-norm of the velocity of the fluid --  and conclude by a Gr\"onwall argument.}  We recall here standard {trace results for this purpose}. They will however be insufficient for our purpose but we will need them to estimate remainder terms.
\par
To state these results,  we fix $h < 0$ and $\theta \in \mathbb R$ such that $(h,\theta) \in A_{1,e}.$ Given $v \in \mathcal V(B)$, we recall that $v$ matches a rigid motion in $B$. 
In the following lemma, we compute this rigid motion with respect to $\mathsf x+ h\widehat{e}_2$, namely the point realizing the minimum distance between $B$ and $\partial A$ (and thus betwen $B$ and $x_2 = -L$ since $h <0$). 
We set
\begin{equation} \label{rigid}
v(x_1,x_2) = \ell_1 \widehat{e}_1 + \ell_2 \widehat{e}_2 + \alpha (x - (\mathsf x + h\widehat{e}_2))^{\bot}\,, \qquad \forall (x_{1},x_{2}) \in \overline{B} \, .
\end{equation} 
{We emphasize that $\ell_1$ might not vanish since this is the velocity of the point  $\mathsf x + h\widehat{e}_2$ and not the center of the ellipse.}
With this convention, we have the following statement
\begin{lemma}\label{lemma_smart}
There exists a constant $C_{geo}>0$, depending only on $e$ and $L$ such that,
for every $v \in \mathcal V(B)$ satisfying \eqref{rigid} in $\overline{B}$, 
\begin{itemize}
\item[(i)] $|\ell_2| \leq C_{geo} \, \mathfrak{d}(h,\theta)^{3/4} \|\nabla v \|_{L^2(A)}$,
\item[(ii)]  $| \ell_1 | \leq C_{geo} \,{\mathfrak d(h,\theta)^{1/4}} \|\nabla v \|_{L^2(A)}$,
\item[(iii)]  $| \alpha | \leq C_{geo} \, \mathfrak{d}(h,\theta)^{1/4} \|\nabla v \|_{L^2(A)}$.
\end{itemize}
\end{lemma}
\begin{proof}
{We prove our result in case $\mathfrak d(h,\theta) \leq \lambda_*^2$ ($\lambda_*$ being a minimum length on which we have a local parametrization of $\partial B$) otherwise this is a consequence to Poincaré inequality  \eqref{eq_dissipationsolide}.}

We start with the second inequality. We simply integrate $\partial_2 v_1$ in the gap $\mathcal G_{\lambda}$ to get
\[
\begin{split}
\left|\int_{\sf x_1-\lambda}^{\sf x_1+\lambda} v_{1}(x_1,h + \gamma(x_1-\mathsf x_1)) \, dx_1 \right| & = \left|\int_{\mathcal G_{\lambda}} \partial_2 v_1\, dx_1\,dx_2\right|\\
& \leq \sqrt{\lambda(\mathfrak d(h,\theta)+c_2^{(2)}\lambda^2)}\|\nabla v\|_{L^2(A)}\,,   
\end{split}
\]
for all small $\lambda > 0$.
However, for $\lambda > 0$ small, we have 
\[
\left|\int_{\sf x_1-\lambda}^{\sf x_1+\lambda} v_{1}(x_1,h+\gamma(x_1-\mathsf x_1))\,{ d}x_1 - 2\lambda \ell_1 \right| \le C_{geo}\, \lambda^3 |\alpha|\,.
\]
We note at this point that, by \eqref{eq_dissipationsolide}, we have
$|\alpha| \leq C_{geo} \|\nabla v\|_{L^2(A)}.$
Eventually, we obtain 
$$|\ell_1|\leq C_{geo} \left(\sqrt{\frac{\mathfrak d(h,\theta)}{\lambda}+c_2^{(2)}\lambda}  + {\lambda^2}\right)\|\nabla v\|_{L^2(A)}\,.
$$
To optimize the asymptotic estimate for $\mathfrak d(h,\theta) << 1$, we now take $\lambda=\mathfrak d(h,\theta)^{\frac12}$ (that is a possible choice {since $\mathfrak{d}(h,\theta) \leq \lambda_*^2$}) to conclude that
$$
|\ell_1|\leq C_{geo}\, \left(\mathfrak d(h,\theta)^{1/4}+{\mathfrak d(h,\theta)}\right)\|\nabla v\|_{L^2(A)} {\leq C_{geo}\, \mathfrak d(h,\theta)^{1/4} \|\nabla v\|_{L^2(A)}} \,.
$$

We now turn back to assertion (i). Going inside the proof of {\cite[Theorem 3.1 - eq. (12)]{Star2}}, one realizes that 
$$\left|v({\sf x}_1,{\sf x}_2)\cdot \widehat{n}({\sf x}_1,{\sf x}_2)\right|\leq C_{geo} \mathfrak{d}(h,\theta)^{3/4} \|\nabla v \|_{L^2(A)}$$
for some $C_{geo}>0$ that can be controlled once we know the curvature of $\partial B$ is bounded from above and from below by positive constants. This actually leads to (i). However, to have a better view on the dependance of $C_{geo}$, we {propose a rewriting of the } argument which will be used anyway to prove assertion (iii). We integrate $\partial_2 v_2=-\partial_1 v_1$ on $\mathcal G_{\lambda}$ to obtain
\[
2\lambda\ell_2 = \int_{\sf x_1-\lambda}^{\sf x_1+\lambda} v_{2}(x_1,h + \gamma(x_1-\mathsf x_1)) \, dx_1 = \int_{{\mathcal G_{\lambda}}} \partial_1 v_1\, dx_1\,dx_2\,.
\]
Estimating the right-hand side like in the proof of (ii) is not enough. We compute instead 
$$\int_{{\mathcal G_{\lambda}}} \partial_1 v_1\, dx_1\,dx_2 = \int_{-L}^{h+\gamma(-\lambda)}v_1({\sf x}_1-\lambda,x_2)\, dx_2 - \int_{-L}^{h+\gamma(\lambda)}v_1({\sf x}_1+\lambda,x_2)\, dx_2$$
and we deduce that 
$$\left|\int_{{\mathcal G_{\lambda}}} \partial_1 v_1\, dx_1\,dx_2 \right|^2
\le 2(\mathfrak d(h,\theta)+c_2^{(2)}\lambda^2)\left(\|v_1({\sf x}_1-\lambda,\cdot)\|_{L^2(dx_2)}^2+\|v_1({\sf x}_1+\lambda,\cdot)\|_{L^2(dx_2)}^2\right),$$
for all small $\lambda > 0$.
Using Hardy inequality, we then infer  
\[
\begin{split}\medskip
\left|\int_{{\mathcal G_{\lambda}}} \partial_1 v_1\, dx_1\,dx_2 \right|^2
& 
\le  C\left(\mathfrak d(h,\theta)+c_2^{(2)}\lambda^2\right)^3\left(\|\partial_2v_1({\sf x}_1-\lambda,\cdot)\|_{L^2(dx_2)}^2\right.\\ \medskip
& \quad + \left.\|\partial_2v_1({\sf x}_1+\lambda,\cdot)\|_{L^2(dx_2)}^2\right).\\
\end{split}
\]
Integrating over $(0,\lambda)$ gives
\[
\begin{split}
\frac{4\lambda^3}{3}|\ell_2|^2
& \le C\left(\mathfrak d(h,\theta)+c_2^{(2)}\lambda^2\right)^3 \left(\int_{{\sf x}_1-\lambda}^{{\sf x}_1+\lambda}\int_{-L}^{h+\gamma(\xi-{\sf x}_1)}|\partial_2v_1(\xi,x_2)|^2\,  dx_2 d\xi\right)\\
\end{split}
\]
and therefore 
\[
\begin{split}
 |\ell_2|& \le  C\left(\frac{\mathfrak d(h,\theta)}{\lambda}+c_2^{(2)}\lambda\right)^{3/2}\|\nabla v\|_{L^{2}(A)},
\end{split}
\]
so that (i) is proved by choosing $\lambda=\mathfrak d(h,\theta)^{\frac12}$.

We finally focus on assertion (iii). We integrate again $\partial_2 v_2=-\partial_1 v_1$ but this time we select on a non-symmetric subregion, for instance $\overline{\mathcal G_{\lambda}}={\{\mathcal G_{\lambda}\mid \lambda>0\}\setminus\{\mathcal G_{\lambda/2}\mid \lambda>0\}}$,  of the gap $\mathcal G_{\lambda}$ to obtain
\[
\left|\frac{\lambda \ell_2}{2}  +\frac{3\lambda^2\alpha}{8}\right|^2 = \left|\int_{\sf x_1+\lambda/2}^{\sf x_1+\lambda} v_{2}(x_1,h + \gamma(x_1-\mathsf x_1)) \, dx_1 \right|^2= \left|\int_{\overline{\mathcal G_{\lambda}}} \partial_1 v_1\, dx_1\,dx_2\right|^2\,.
\]
As above, we compute
$$\int_{\overline{\mathcal G_{\lambda}}} \partial_1 v_1\, dx_1\,dx_2 = \int_{-L}^{h+\gamma(\lambda)}v_1({\sf x}_1+\lambda,x_2)\, dx_2 - \int_{-L}^{h+\gamma(\lambda/2)}v_1({\sf x}_1+\lambda/2,x_2)\, dx_2$$
and we deduce that 
\[
\begin{split}
\Bigg|\int_{\overline{\mathcal G_{\lambda}}}  & \partial_1 v_1\, dx_1\,dx_2 \Bigg|^2 \\
&\le 2\left(\mathfrak d(h,\theta)+c_2^{(2)}\lambda^2\right)\!\left(\|v_1({\sf x}_1+\lambda,\cdot)\|_{L^2(dx_2)}^2+
\|v_1({\sf x}_1+\lambda/2,\cdot)\|_{L^2(dx_2)}^2\right), 
\end{split}
\]
for all small $\lambda > 0$.
Using Hardy inequality,
we deduce that
\[
\begin{split}\medskip
\Bigg|\int_{\overline{\mathcal G_{\lambda}}}  & \partial_1 v_1\, dx_1\,dx_2 \Bigg|^2 \\
& \le 
  \left(\mathfrak d(h,\theta)+c_2^{(2)}\lambda^2\right)^3\left(\|\partial_2v_1({\sf x}_1+\lambda,\cdot)\|_{L^2(dx_2)}^2+\|\partial_2v_1({\sf x}_1+\lambda/2,\cdot)\|_{L^2(dx_2)}^2\right).
\end{split}
\]
It then follows 
\[
\begin{split}
\int_{0}^\lambda \left|\frac{\xi \ell_2}2  +\frac{3\xi^2\alpha}{8}\right|^2\, d\xi
& \le 2 \left(\mathfrak d(h,\theta)+c_2^{(2)}\lambda^2\right)^3 \left(\int_{{\sf x}_1}^{{\sf x}_1+\lambda}\int_{-L}^{h+\gamma(\xi-{\sf x}_1)}|\partial_2v_1(\xi,x_2)|^2\,  dx_2 d\xi\right. \\
& \quad \left.+2 \int_{{\sf x}_1}^{{\sf x}_1+\lambda/2}\int_{-L}^{h+\gamma(\xi-{\sf x}_1)}|\partial_2v_1(\xi,x_2)|^2\,  dx_2 d\xi\right)\\
\end{split}
\]
and therefore, using  (i), we get the estimate 
\[
\begin{split}
 |\alpha|^2& \le  C_{geo}\, \lambda\left(\left(\frac{\mathfrak d(h,\theta)}{\lambda^{2}}+c_2^{(2)}\right)^3 + \frac{\mathfrak{d}(h,\theta)^{3/2} }{\lambda^3}\right)\|\nabla v\|_{L^{2}(A)}^2,
\end{split}
\]
so that (iii) is proved by choosing $\lambda=\mathfrak d(h,\theta)^{\frac12}$. 
\end{proof}
Applied to solutions of \eqref{eq:hom_pb1}-\eqref{eq:ODE_newref1}-\eqref{ichat}, Lemma \ref{lemma_smart} gives a finer control on the velocities $(\theta',h')$ close to the walls. The distance estimate \eqref{eq:dist-estim} is already contained in \cite[Theorem 3.1]{Star2}. The velocity bound \eqref{eq:refined-trace} will be crucial in our analysis in Section \ref{sec:distance2}.
\begin{corollary}\label{cor:finer-trace}
Let $\lambda_0 \leq \lambda_0^{(0)}$ and $(u_0,{ h'_0,\theta'_0)\in\mathcal{H}(B_0)}$. Let $(u,{h},{\theta})$ be a weak solution to problem \eqref{eq:hom_pb1}-\eqref{eq:ODE_newref1}-\eqref{ichat} whose existence follows from 
Theorem \ref{th:existence_weak_local} on its lifespan $(0,T)$. 
We have 
\begin{equation}\label{eq:dist-estim}
|\mathfrak d'(t)|^2 \leq C_{geo}\,  \mathfrak d(t)^{3/2} \|\nabla u\|^2_{L^2(A)}  
\end{equation}
and
\begin{equation}\label{eq:refined-trace}
|h'|^2 + |\theta'|^2 \leq C_{geo}\,  \sqrt{\mathfrak d(t)} \|\nabla u\|^2_{L^2(A)}\,,
\end{equation}
where $C_{geo}\,$ depends on $e$.
\end{corollary}
\begin{proof}
Since, expressing the rotation with respect to the point of contact $\sf{x}=(\sf{x}_1,\sf{x}_2)$, we have 
$$u={h'}\,\widehat{e}_2 + \theta'(x-h \widehat{e}_{2})^{\perp} =  
- \theta'{\sf x}_2 \widehat{e}_{1} + (\theta'{\sf x}_1+h')\widehat{e}_{2} + \theta' (x - (\mathsf x + h\widehat{e}_2))^{\bot}  \text{ on } \partial B(t)\,,
$$
we deduce from Lemma \ref{lemma_smart} that 
$$|\theta'|^2\le C_{geo}\, \sqrt{\mathfrak d(t)} \|\nabla u\|^2_{L^2(A)}, \quad |\theta'{\sf x}_1+h'|^2\le C_{geo}\,\mathfrak d(t)^{3/2} \|\nabla u\|^2_{L^2(A)}\,,$$
{Recalling that $\mathfrak{d}(t) = h(t) + \mathsf{x}_2[\theta(t)] + L$,  the conclusion follows since $\partial_{\theta}{\sf x}_2[\theta]={\sf x}_1[\theta]$.  This latter identity can be observed by $\theta$-differentiating  \eqref{eq:x1x2} for instance.}
\end{proof}


\section{Distance estimate -- Identifying a potential energy of contact}\label{sec:distance2}
In this section we fix an initial configuration 
and we assume that $\lambda_0 \leq \lambda_{0}^{(0)}$ so that Theorem \ref{th:existence_weak_local} holds true. We focus on a weak solution $(u,h,\theta)$ 
to \eqref{eq:hom_pb1}-\eqref{eq:ODE_newref1}-\eqref{ichat} on the time interval $(0,T_M)$,  where we recall that the definition of $T_M$ (given in \eqref{T_M}) implies that no contact arises and $|\theta|< \theta_M$ on $(0,T_M)$. To fix the ideas,  we focus on a distance estimate with the lower boundary $x_2 = -L$ of $A. $ By symmetry, similar arguments would enable to control the distance to  the top boundary  of $A.$ We assume that, on some time interval $ (T_-,T_+) \subset [0,T_M]$ we have $h(t) < 0$ and $\mathfrak d(t) \doteq \mathfrak d(h(t),\theta(t))< \eta_0$ with $\eta_0 >0$ small to be fixed later on.   Up to taking $\eta_0$ smaller in case $T_-=0$, we may assume $\mathfrak d(T_-) = \eta_0$ without loss of generality.

For further reference, we recall that \eqref{hp_F} implies that {given $\eta_0\in(0,L-1)$, there exists a constant $k_0 > 0$ such that
\begin{equation} \label{eq_lbF}
	H_h(h(t),\theta(t)) \leq {-k_0} = -\frac{{\bar{\varpi}}\bar{\varrho}(L-1-\eta_0)}2< 0 \quad \text{as long as } \mathfrak d(t) < \eta_0.
\end{equation}

The section is dedicated to the proof of the following assertions that obviously imply Theorem \ref{th:uniform_distance-intro} since $p_0=2\lambda_0/L^2$. 
\begin{theorem}\label{th:uniform_distance}
	There exist a decreasing function $\omega_0^-(E_0)$ of the initial energy, that also depends on the parameters $m,J,e,L,\overline{\varpi},\overline{\varrho}$, a constant $\sigma\in(0,1)$ that depends on $\overline{\varpi}$, $C_{geo}>0$,
	and a threshold $\lambda_0^{(1)}>0$  depending on $m,J,e,L,\overline{\varpi}$,  such that, if 
	\begin{equation}\label{condition_lambda0}
\lambda_0 \leq \lambda_{0}^{(0)},\ 	\lambda_0\le \lambda_0^{(1)}k_0\ \text{ and } \			
C_{geo}\lambda^2_0\le \sigma\omega_0^-\bar \varrho,
	\end{equation}
	then
	\begin{itemize}
		\item[(i)] there exists a global weak solution $(u,h,\theta)$ to \eqref{eq:hom_pb1}-\eqref{eq:ODE_newref1}-\eqref{ichat};
		\item[(ii)] there exists ${\mathfrak d}_{min}^0 >0$ for which
		\begin{equation} \label{eq_minimaldistance}
			{\rm dist}(B(t), \partial A) \geq {\mathfrak d}_{min}^0 \qquad \forall t \geq 0 \, ;
		\end{equation}
		\item[(iii)]
		there exists $\beta_0>0$, that depends decreasingly on $E_0$, such that 
		\begin{equation}\label{uniform_estimate_en-section5}
			E_{tot}(t)\le  3E_0\,e^{-\beta_0 t}+  \frac{C_{geo}\lambda_0^2}{\beta_0}.
		\end{equation}
	\end{itemize}
\end{theorem}

  The proof, which basically requires to show that $T_M=+\infty$, will be done in several steps. First, we deduce a local-in-time version of Assertion (ii), namely we prove the existence of a lower bound for the distance on the time interval $(0,T_M)$, see Proposition \ref{prop:uniform_distance}. This implies that either $T_M=+\infty$ or $\theta(T_M)=\theta_M$. In the second case, in Theorem \ref{cor_estimation1}, by exploiting a continuation argument, built on a local-in-time version of Assertion (iii), we show an absurd, so that actually $T_M=+\infty$. Assertion (i) is then a consequence of the alternative in Item (iii) of Thereom \ref{th:existence_weak_local} and the statements (ii)-(iii) eventually hold for all $t\ge 0$.

\par

With the study of the long time behavior in mind, the crucial novelty of this theorem is that the bound {\eqref{eq_minimaldistance}} is {global in time}.  This generality requires the smallness condition imposed on the Poiseuille flow and motivates a novel approach. The interested reader may note that integrating \eqref{eq:eq_estimate_dist} {below} without imposing that the coefficient multiplying $(t-T_-)$ is {negative} (namely, for an arbitrary Poiseuille flow) yields only a distance estimate in finite-time. In fact, our approach is flexible enough to extend the no-collision result in finite time known for a disk to the case of the ellipse, even in the absence of a restoring force. As already mentioned, when the solid is a disk, the symmetries play a fundamental role to rule out collisions. {Since our analysis includes the case of an ellipse, we will have to consider a non-symmetric gap and take into account that the curvature varies in time.}   

\begin{remark}
	The explicit forms of ${\mathfrak d}_{min}^0$, $\omega_0^-$ and $\beta_0$ will be given respectively in Proposition \ref{prop:uniform_distance} and 
	Theorem \ref{cor_estimation1}. We also comment on the physical meaning of the smallness assumption \eqref{condition_lambda0} of Theorem \ref{th:uniform_distance} in   Remark \ref{remark_viscosity2}.
\end{remark}

\subsection{Further remarks on the energy estimate and weak formulation}
Our reasoning will be based on two main ingredients: the energy estimates and the weak formulation of \eqref{eq:evolution_pb}.

\medskip

Concerning the energy estimate, we can use a finer version of \eqref{eq_Ekin} on $(T_-,T_+)$. {Indeed, taking advantage of the assumption that $\mathfrak{d}(t) < \eta_0$ on $(T_-,T_+),$  we extract that the potential energy of the ellipse is almost constant so that the dissipation essentially decreases the kinetic energy.} More precisely,  under this assumption, we may replace \eqref{eq_dissipationsolide} 
by its refined version from Corollary \ref{cor:finer-trace}, that is 
\[
|h'|^2 + |\theta'|^2 \leq {C_{geo}}  \sqrt{\eta_0} \|\nabla u\|^2_{L^2(A)}
\] 
where we have bounded $\mathfrak d(t)$ by $\eta_0$ and {$C_{geo}$} is a constant that depends only on {$e$ and $L$.} We then obtain the refined bound
\[
 \left| h'  H_h(h(t),\theta(t)) + \theta'  H_\theta(h(t),\theta(t)) \right|  \leq  \dfrac{1}{8} \int_{A} |\nabla u|^2+ C_{geo}{\sqrt{\eta_0}}
 \max_{s\in(0,t)}|\nabla H(h(s),\theta(s))|^2\,,
 \] 
and arguing as in the proof of Proposition \ref{prop_Ekin}, we infer that
 \[
 \begin{split}
E_{kin}(t_2) - E_{kin}(t_1) + & \dfrac{1}{8} \int_{t_1}^{t_2} \|\nabla u(\tau)\|^2_{L^2(A)}d\tau 
 \leq  {C_{geo}} \left( {\lambda^2_0} + \sqrt{\eta_0} S_M  \right)(t_2-t_1),
\end{split}
 \]
for all $T_- \leq t_1 \leq t_2 \leq T_+.$ 
Writing this estimate between $t_1 = T_-$ and $t=t_2 \leq T_+$, we infer 
\begin{equation} \label{eq_estimate_0}
\int_{T_-}^{t} \|\nabla u(\tau)\|^2_{L^2(A)}d\tau  \leq 8 E_{kin}^{M} + {C_{geo}} \left({\lambda_0^2}+{\sqrt{\eta_0}} S_M \right)(t-T_-) \, ,
\end{equation}
where we recall, from \eqref{Ekin_max} and the definition of $S_M$, that 
\begin{equation}\label{eq:E_kinM}
E_{kin}^{M}= E_{kin}(0) + {C_{dyn}}\left({\lambda^2_0} + S_M   \right)\le E_0+ {C_{dyn}}\left({\lambda_0^2} + S_M   \right)\,.
\end{equation}
In particular, we infer that 
\begin{equation} \label{eq_estimate_01}
 \int_{T_-}^{t} \|\nabla u(\tau)\|^2_{L^2(A)}d\tau  \leq 8E_0+C_{dyn}(\lambda_0^2+S_M) + {C_{geo}} \left({\lambda_0^2}+{\sqrt{\eta_0}} S_M\right)(t-T_-) \, .
\end{equation}

\black
\medskip

As for the weak formulation,  we recall that  $u \in     \mathcal{C}([0,T]; L^2(A))$ satisfies  \eqref{eq:weak_gen} for arbitrary compatible $(\phi,\ell,\alpha) \in H^1((T_-,T_+) \times A) \times [H^1((T_-,T_+))]^2.$  We point out that,  for $t \in [T_-,T_+]$, we have $h < 1/2$ since we may assume without loss of generality that $\eta_0 < 1/2.$ Hence,  the term $\widehat{f}[h]$ in \eqref{eq:weak_hom} vanishes on $[T_-,T_+]$ and the weak formulation in the form of \eqref{eq:weak_gen} reduces to:
\begin{equation} \label{eq:weak_gen2}
\begin{aligned}
&\int_{\Omega(t_2)} u(t_2) \cdot w(t_2)dx -  \int_{\Omega(t_1)} u(t_1) \cdot w(t_1)dx \\[4pt]
&
+ m \left(h'(t_2)\ell(t_2)-h'(t_1)\ell(t_1)\right) + {J} \left( \theta'(t_2) \alpha(t_2)-\theta'(t_1) \alpha(t_1)\right)\\
& -\int_{t_1}^{t_2} \left( \int_{\Omega(\tau)}{ \left( u \cdot \partial_{t} w + (u\cdot \nabla)\,{w} \cdot u\right)} \, dx + m\,{h'}\,{\ell'} -H_{h}(h,\theta) \, \ell + {J}\,{\theta'}\,{\alpha'}-H_{\theta}(h,\theta) \, \alpha \right) \, d\tau \\[2pt]
& + 2  \int_{t_1}^{t_2} \int_{\Omega(\tau)} D(u) : D(w) \, dxd\tau + \int_{t_1}^{t_2} \int_{\Omega(\tau)} \left( (u\cdot \nabla)\,s  +(s\cdot \nabla)\,u  \right) \cdot w \, dxd\tau \\ 
& =  \int_{t_1}^{t_2} \int_{\Omega(\tau)} \widehat{g} \cdot w \, dx d\tau   
 \end{aligned}
\end{equation}
for all $T_- \leq t_1 \leq t_2 \leq T_+.$ 
Below, we will use this weak formulation with compatible triplets $(w,\ell,\alpha)$ that satisfy further 
\[
(w,\ell,\alpha)  \in W^{1,\infty}((T_-,T_+)\times A){\times\left( W^{1,\infty}((T_-,T_+))\right)^2}.
\]
The above weak formulation implies then in particular that 
\begin{equation} \label{eq_Tw}
I_{w} \doteq t \mapsto \int_{\Omega(t)} u \cdot w \, dx + m h' \ell + J \theta' \alpha \in H^1([T_-,T_+]) \, ,
\end{equation}
with
\begin{equation} \label{eq_Tw'}
\begin{aligned}
I_{w}' (t)  
  = &    \int_{\Omega(t)} \left( u \cdot \partial_{t} w + (u\cdot \nabla)\,{w} \cdot u\right) dx + m\,{h'}\,{\ell'} -H_{h}(h,\theta) \, \ell + {J}\,{\theta'}\,{\alpha'}-H_{\theta}(h,\theta) \, \alpha  \\[6pt]
& - 2   \int_{A} D(u) : D(w) \, dx  -  \int_{\Omega(t)} \left( (u\cdot \nabla)\,s +(s\cdot \nabla)\,u  \right) \cdot w\, dx +  \int_{\Omega(t)} \widehat{g} \cdot w \, dx\,,
\end{aligned}
\end{equation}
 for a.e.  $t\in [T_-,T_+]$.\black

\subsection{Definition of a potential energy of contact}

As emphasized in \cite{GeVaHill} among others, controlling the distance between $B(t)$ and $\partial A$ should be based on exploiting the term
\[
2 \int_{A} D(u) : D(w)\, dx
\]
that enables to  approximate the repulsive force $F_{rep}$ exherted on $B(t)$ by the one induced by the Stokes problem. In case of a simple vertical translation, equivalent computations to the ones in \cite[Section 3.1]{GeVaHill1} (see in particular formula (3.7)) in the 2D case show that, this repulsive force is approximated by:
\[
F_{rep} \sim - {\mathfrak d}' \int_{\mathcal G_{\lambda_*}} |\partial_{22} \psi_{opt}^{\bot}|^2\, dx = - 12{\mathfrak d}'  \int_{-\lambda_*}^{\lambda_*} \dfrac{(\tau- c_*^{\bot})^2}{(\mathfrak d + \gamma[\theta](\tau) - \mathsf x_2[\theta])^3}\,{d}\tau.
\]
If the rotation is frozen, this quantity is related to the time-derivative of
\[
{P_c^0}(t) \doteq 6 \int_{-\lambda_*}^{\lambda_*} \dfrac{(\tau- c_*^{\bot})^2}{(\mathfrak d(t) + \gamma[\theta(t)](\tau) - \mathsf x_2[\theta(t)])^2}\, {d}\tau
\]
and we may expect to control the distance $\mathfrak d$ through ${{P_c^0}}$ (remember that it scales like $\mathfrak d^{-1/2}$ when $\mathfrak d << 1$). This will stand for our potential energy of contact. However, when rotation is allowed we get:
\begin{equation} \label{eq_H_0'}
\begin{aligned}
({P_c^0})'(t) = & - 12 {\mathfrak d}'(t)  \int_{-\lambda_*}^{\lambda_*} \dfrac{(\tau- c_*^{\bot})^2}{(\mathfrak d(t) + \gamma[\theta(t)](\tau) - \mathsf x_2[\theta(t)])^3}\, {d}\tau \\
&  +  12 \int_{-\lambda_*}^{\lambda_*} \dfrac{{({\mathfrak d}' (t) \partial_{\mathfrak d}{c}_*^{\bot} + \theta'(t) \partial_{\theta} {c}_*^{\bot})}  (\tau- c_*^{\bot})}{(\mathfrak d(t)  + \gamma[\theta(t)](\tau) - \mathsf x_2[\theta(t)])^2}{d}\tau   \\
&  - 12\theta'(t)  \int_{-\lambda_*}^{\lambda_*} \dfrac{(\tau- c_*^{\bot})^2 (\partial_\theta \gamma[\theta(t)](t) - \partial_{\theta}{\mathsf x}_2[\theta])}{(\mathfrak d(t) + \gamma[\theta(t)](\tau) - \mathsf x_2[\theta(t)])^3} \, {d}\tau\,,
\end{aligned}
\end{equation}
where the terms on the second and third lines are comparable to ${P_c^0}$ when $\mathfrak d << 1.$ In order to avoid the use of a Gr\"onwall argument that would induce a loss of control on the distance for large times, we propose here to use a modulation trick. We must also take into account the effect of fluid inertia. This motivates the following construction.
\black

\medskip

{Let $a(t)\in W^{1,\infty}((T_-,T_+))$ be a strictly positive amplitude function such that $a(T_-)=1$, and suppose that there are two constants $0 < \underline{a} \leq \overline{a}$ such that $\underline{a}\le a(t)\le \overline{a}$. We then define
\[
\begin{split}
P_c(t) \doteq \   6 a(t) \int_{-\lambda_*}^{\lambda_*} & \dfrac{(\tau -c_*^{\bot}) ^2}{(\mathfrak d(t) + \gamma[\theta(t)](\tau) - \mathsf x_2[\theta(t)])^2 } \, d\tau \\
& \quad  - \left( \int_{\Omega(t)} u(t,x) \cdot  a(t) \tilde{v}^{\bot}(t,x) \, dx +  m h'(t)  a(t)   \right),
\end{split}
\]
for $t \in [T_{-},T_{+}]$. The precise definition of the amplitude function $a(t)$ will be given later on. We first justify the construction of $P_c(t)$ and show that it enables to control the distance $\mathfrak d(t)$.}

\begin{lemma} \label{lem_H}
 We have $P_c \in W^{1,\infty}((T_-,T_+))$ and there is $C_{geo}^{(min)}>0$ and $C_{geo}^{(max)}>0$, depending only on $e$, such that 
 \[
\dfrac{C_{geo}^{(min)}}{\sqrt{\mathfrak d(t)}} - \sqrt{C_0^+}  
\leq \dfrac{P_c(t)}{a(t)}  \leq  \dfrac{C_{geo}^{(max)}}{\sqrt{\mathfrak d(t)}} 
+\sqrt{C_0^+},
\]
for all $t\in [T_-,T_+]$, where $C_0^+=C_{dyn}(E_0+\lambda_0^2+S_M)$ is an increasing function of $E_0$.
\end{lemma}  
\begin{proof}
Time regularity of $P_c$ is a straightforward consequence ot the time regularity of $h,\theta$ and the smoothness of $\partial B.$
To obtain the lower and upper estimates, we bound independently the two terms in $P_c(t)$. We start observing that 
		$$
		\begin{aligned}
		\int_{-\lambda_*}^{\lambda_*} \dfrac{ (\tau -c_*^{\bot}) ^2}{(\mathfrak d(t) + \gamma[\theta(t)](\tau) - \mathsf x_2[\theta(t)])^2 } \, d\tau\  = \ &\int_{-\lambda_*}^{\lambda_*} \dfrac{ \tau^2}{(\mathfrak d(t) + \gamma[\theta(t)](\tau) - \mathsf x_2[\theta(t)])^2 } \, d\tau\\ & -	\int_{-\lambda_*}^{\lambda_*} \dfrac{ 2\tau c_*^{\bot}}{(\mathfrak d(t) + \gamma[\theta(t)](\tau) - \mathsf x_2[\theta(t)])^2 } \, d\tau\\&+\int_{-\lambda_*}^{\lambda_*} \dfrac{(c_*^{\bot})^2}{(\mathfrak d(t) + \gamma[\theta(t)](\tau) - \mathsf x_2[\theta(t)])^2 } \, d\tau\,.
		\end{aligned}
		$$
A combination of Lemmas \ref{lemma_Nico}-\ref{lem_c*} in Appendix \ref{app_asymptotics} thus yields that, since $\mathfrak{d}(t) < \eta_0$ there exists a constant $C_{geo}$ depending only on $e$ for which
\[
\left|	\int_{-\lambda_*}^{\lambda_*} \dfrac{(\tau -c_*^{\bot}) ^2}{(\mathfrak d(t) + \gamma[\theta(t)](\tau) - \mathsf x_2[\theta(t)])^2 } \, d\tau - \dfrac{I_{2,2}(\kappa_2[\theta])}{\sqrt{\mathfrak d}} \right| \leq  C_{geo}\,,
\]
	where  the symbol  $I_{2,2}(\kappa_2)$ stands for the integral 
\[
I_{2,2}(\kappa_2) \dot{=}   \int_{\mathbb R} \dfrac{\tau^2}{(1 + \kappa_2 \tau^2)^2}\, {d}\tau.
\]
We remark that $I_{2,2}$ is a smooth and bounded function of $ \kappa_2\in [\kappa_2^{min},\kappa_2^{max}]$.  	
Concerning the term
$$\int_{\Omega(t)} u(t,x) \cdot  a(t) \tilde{v}^{\bot}(t,x) \, dx +  m h'(t)  a(t),$$
we use Cauchy–Schwarz inequality together with the energy bound \eqref{eq_Ekin}, \eqref{stime_test} and \eqref{eq:E_kinM}  for $t\in~[T_{-}, T_{+}]$. 
\end{proof}

We proceed now with the computation of the time-derivative of $P_c$. To this purpose we recall some definitions of the previous section.
We set $\tilde{v}^{\bot}(t,x) = \nabla^{\bot} \tilde{\psi}^\bot$ the approximation, at time $t \in [T_-,T_+]$, of the Stokes solution $v^{\bot}$ given in \eqref{eq_stokesprofile} with $v_{*}=v_{*}^{\bot}(x) = \widehat{e}_2$. We briefly recall the construction: the geometry of the gap and the distance $\mathfrak d(t)$ are determined at every time $t \in [T_-,T_+]$ by the values of $(h(t),\theta(t))$, then 
$\tilde\psi^\perp$ is defined by \eqref{eq:tilde_psi}  with 
$\psi_{*}^{\bot}(x)=x_1-{\sf x}_1$ and $\psi^\perp_{opt}$ defined in $\eqref{psising}_1$. 
We have similar constructions for the other solid motions yielding the vector-fields $\tilde{v}^{||}$ and $\tilde{v}^{\circlearrowleft}$.
In particular, we set, for all $t\in [T_{-}, T_{+}]$ and for all $x\in A$: 
\begin{equation}\label{vtilde_new}
\tilde{v}(x,t) \doteq  {\mathfrak d}'(t) \tilde{v}^{\bot}(x,t) - \theta'(t) \mathsf{x}_2(t) \tilde{v}^{||}(t) + \theta'(t) \tilde{v}^{\circlearrowleft}(x,t)\,.
\end{equation}
We observe then that $\tilde{v}=u$ on $B(t)$ for all $t \in [T_-,T_+]$. With these notations at hand, we have the following lemma which states the proper way to compute the time-derivative of $P_c$.
\begin{lemma}\label{lem:H'} 
 The derivative ${P_c}'$ can be decomposed as
\begin{equation} \label{eq_Hdot}
{P_c}'(t) =   H_{h}(h(t),\theta(t))a(t)+    {\rm Mod}(t) + a(t)  {\rm Rem}_1(t) + {\rm Rem}_2(t)\,,
\end{equation}
where  ${\rm Rem}_1(t) = {\rm Rem}_{1}^{(a)}(t) + {\rm Rem}_{1}^{(b)}(t)$ and there is a constant $C_{geo}$ (depending only on $e$) for which 
\begin{align} 
\hspace{-8mm}\left|{\rm Rem}^{(a)}_1(t) \right. - &\left.\int_{\Omega(t)} \nabla \tilde{v}^{\bot} : \nabla (u - \tilde{v})\,dx  \right| \leq  C_{geo} \left(\dfrac{|{\mathfrak{d}}'(t)|} {\sqrt{\mathfrak d}(t)} + {|\theta'(t)|} \right)   \,, \label{eq_Rem1a} \\[6pt]  
\hspace{-8mm}{\rm Rem}_{1}^{(b)} (t)= & - {\theta'(t)}\int_{\Omega(t)} 
\bigg[2 \partial_{12} \tilde{\psi}^{\bot}  \left(
\partial_{12} \tilde{\psi}^{\circlearrowleft}  -  \mathsf x_2 \partial_{12} \tilde{\psi}^{||}  \right)
\\[6pt]
&\hspace{8mm}+  
  \partial_{11} \tilde{\psi}^{\bot}\! \left( \partial_{11} \tilde{\psi}^{\circlearrowleft}  - 
 \mathsf x_2 \partial_{11} \tilde{\psi}^{||} \right)
  \bigg]\,dx, \label{eq_Rem1b}
 \end{align}
\begin{align}\label{eq:Mod}   
{\rm Mod}(t)  = &\  6a'(t)  \int_{-\lambda_*}^{\lambda_*} \dfrac{ (\tau -c_*^{\bot}) ^2}{(\mathfrak d(t) + \gamma[\theta(t)](\tau))^2 } d \tau \\[6pt]& + a(t)\theta'(t)  \bigg(  
 \int_{\Omega(t)}  \partial_{22} \tilde{\psi}^{\bot}  \left(\partial_{22}\tilde{\psi}^{\circlearrowleft}   - 
\mathsf x_2 [\theta(t)] \partial_{22} \tilde{\psi}^{||}  
 \right) dx \\[6pt]&
 \hspace{20mm}- 12 \int_{-\lambda_*}^{\lambda_*} \dfrac{ \left( \partial_{\theta}\gamma[\theta](\tau)  - \partial_{\theta} \mathsf x_2[\theta(t)]\right) (\tau-c_*^{\bot})^2}{(\mathfrak d(t) + \gamma[\theta(t)](\tau) - \mathsf x_2)^3} {d}\tau  
\bigg)\,, \\ 
{\rm Rem}_2(t) = 
&  -   \int_{\Omega(t)} [\partial_{t} (a(t)\tilde{v}^{\bot}(t)) + (u(t) \cdot \nabla) (a(t)\tilde{v}^{\bot}(t))] \cdot u(t) \, dx  - m h'(t) {a}'(t)   \\[6pt]
& +   \int_{\Omega(t)}\!\! \left((u(t)\cdot \nabla)\,s  +(s\cdot \nabla)\,u(t) \right) \cdot a(t) \tilde{v}^{\bot} (t) \, dx \!   
- \!\!\int_{\Omega(t)}\!\! \widehat{g} \cdot a(t) \tilde{v}^{\bot}(t)  \, dx. \label{eq:Rem2}
\end{align}
\end{lemma}

\par Observe that we have organized the terms in \eqref{eq_Hdot} in such a way that ${\rm Rem}_2(t)$ includes all the other terms (not included so far) but the viscous ones in the weak formulation  \eqref{eq_Tw'} with $w= a(t) \tilde{v}^{\bot}$ (we point out that the rigid velocities associated with $w$ are then $\ell(t) = a(t)$ while $\alpha(t) =0$), whereas ${\rm Mod}(t)$ incorporates the terms that will be balanced with the variations of $a(t)$. As it will be clear from the proof of Lemma \ref{lem_Mod}, indeed, ${\rm Mod}(t)$ behaves as a factor of $1/\sqrt{\mathfrak{d}}$, thus in order to obtain a uniform bound for $P_c$ we need to ``neutralize" its behavior with respect to the distance parameter $\mathfrak{d}$ by means of the amplitude function $a$. 

\begin{proof}[Proof of Lemma \ref{lem:H'}]
We first remind the time-derivative of ${P_c^0}:$
\begin{align*}
{\dfrac{d}{dt}} \Bigl[\int_{-\lambda_*}^{\lambda_*} \dfrac{ 6 (\tau- c_*^{\bot})^2}{(\mathfrak d + \gamma[\theta](\tau) - \mathsf x_2[\theta])^2 } \, d\tau \Bigr]
 = & -12  \int_{-\lambda_*}^{\lambda_*} \dfrac{{\mathfrak d}' (\tau- c_*^{\bot})^2}{(\mathfrak d + \gamma[\theta](\tau) - \mathsf x_2[\theta])^3} {d}\tau  \\[4pt]
&  +  12 \int_{-\lambda_*}^{\lambda_*} \dfrac{({\mathfrak d}'  \partial_{\mathfrak d}{c}_*^{\bot} + \theta' \partial_{\theta} {c}_*^{\bot})  (\tau- c_*^{\bot})}{(\mathfrak d  + \gamma[\theta](\tau) - \mathsf x_2[\theta])^2}{d}\tau   \\[4pt]
& - 12 \int_{-\lambda_*}^{\lambda_*} \dfrac{\theta' \left( \partial_{\theta}\gamma[\theta](\tau)  - \partial_{\theta} \mathsf x_2[\theta]\right)(\tau- c_*^{\bot})^2}{(\mathfrak d + \gamma[\theta](\tau) - \mathsf x_2[\theta])^3} { d}\tau \, .
\end{align*}
We denote the three integrals appearing on the right-hand side by ${\rm D}_1(t),{\rm D}_2(t),{\rm D}_3(t)$, respectively. The last ${\rm D}_3$ will be incorporated in the ${\rm Mod}$ term.  Thanks to the Lemmas \ref{lemma_Nico}- \ref{lem_c*} and estimate \eqref{eq_cdot} we can bound
\begin{equation} \label{eq_D2}
|{\rm D}_2(t)| \leq C_{geo}\left(\dfrac{|{\mathfrak{d}}'(t)|} {\sqrt{\mathfrak d(t)}} + {|\theta'(t)|}\right).
\end{equation}

 \medskip
 
We next consider ${\rm D}_1(t)$. We aim to relate it to the second term in the expression of $P_c.$ To rephrase it, we remark that by \eqref{eq:intpsiopt}, we have
\[
\begin{aligned}
\int_{\mathcal G_{\lambda_*}} |\partial_{22} \psi_{opt}^{\bot} |^2 \, dx
=  12  \int_{-\lambda_*}^{\lambda_*} \dfrac{(\tau- c_*^{\bot})^2}{(\mathfrak d + \gamma[\theta](\tau) - \mathsf x_2[\theta])^3} \, {d}\tau.
\end{aligned}
\]
Conversely,  {from the notational conventions above},  we infer 
\[
\begin{aligned}
{\rm D}_1(t)   = \ & - {\mathfrak{d}}'(t) \int_{\mathcal G_{\lambda_*}} |\partial_{22} \psi_{opt}^{\bot}  |^2\, dx\\[6pt]
 = \ &{- {\mathfrak{d}}'(t)\int_{\mathcal G_{\lambda_*}} \nabla \tilde{v}^{\bot} : \nabla \tilde{v}^{\bot}\, dx}{+{\mathfrak{d}}'(t)\int_{\mathcal G_{\lambda_*}}\bigg(2|\partial_{12}\psi_{opt}^{\bot} |^2+|\partial_{11}\psi_{opt}^{\bot}|^2\bigg)\, dx}  
\\[6pt] = \ & { - {\mathfrak{d}}'(t)\int_{\Omega(t)} \nabla \tilde{v}^{\bot} : \nabla \tilde{v}^{\bot}\, dx+ {\mathfrak{d}}'(t)\int_{\Omega(t)\setminus {\mathcal{G}_{\lambda_*}}} \nabla \tilde{v}^{\bot} : \nabla \tilde{v}^{\bot}\, dx} 
\\[6pt]&
+{\mathfrak{d}}'(t)\int_{\mathcal G_{\lambda_*}}\left(2|\partial_{12}\psi_{opt}^{\bot} |^2+|\partial_{11}\psi_{opt}^{\bot}|^2\right)\, dx
\\[6pt] 
 = \ & - \int_{\Omega(t)} \nabla \tilde{v}^{\bot} : \nabla u \, dx + \int_{\Omega(t)} \nabla \tilde{v}^{\bot} :  \nabla ( u - {\mathfrak{d}}'(t) \tilde{v}^{\bot} ) \, dx  + {\rm L}_1(t)
\end{aligned}
\]
where 
\[
{\rm  L}_1(t) =    {\mathfrak{d}}'(t)\int_{\Omega(t)\setminus {\mathcal{G}_{\lambda_*}}} \nabla \tilde{v}^{\bot} : \nabla \tilde{v}^{\bot} \, dx +{\mathfrak{d}}'(t)\int_{\mathcal G_{\lambda_*}}\left(2|\partial_{12}\psi_{opt}^{\bot} |^2+|\partial_{11}\psi_{opt}^{\bot}|^2\right)\, dx.
\]
Using that the singularity of $\tilde{v}^{\bot}$ is concentrated in $\mathcal G_{\lambda_*}$, the estimates in \eqref{stime_test}, where all but vertical second order derivatives are subcritical, and \eqref{stime_test2}, we infer that
\begin{equation} \label{eq_L1}
|{\rm L_1}(t)| \leq   C_{geo} |{\mathfrak d}'(t)|  \left(\dfrac{1}{\sqrt{\mathfrak d(t)}} \right)\,.
\end{equation}

To rewrite the second-term integral in the last expression for ${\rm D}_1(t)$, we {recall that $u= \tilde{v}$ on $\partial B(t).$}
Correspondingly, we use \eqref{vtilde_new} and we split
\[
u - {\mathfrak{d}'} \tilde{v}^{\bot}  = u - \tilde{v} + \theta'\, \mathsf x_2 [\theta] \tilde{v}^{||} - \theta' \tilde{v}^{\circlearrowleft}\,.
\]
We then decompose ${\rm D}_1(t)$ as
\[
\begin{aligned}
{\rm D}_1(t) =&    - \int_{\Omega(t)}  \nabla \tilde{v}^{\bot} : \nabla u  \, dx+ \int_{\Omega(t)} \nabla \tilde{v}^{\bot} :  \nabla ( u -  \tilde{v} )  \, dx  + {{\rm L}_{1}(t) } \\
&  - {\theta'}(t) \int_{\Omega(t)}  \partial_{22} \tilde{\psi}^{\bot}  \left(\partial_{22}\tilde{\psi}^{\circlearrowleft}   - 
 \mathsf x_2[\theta] \partial_{22} \psi_{opt}^{||}  
  \right)\, dx \\
  & - {\theta'}(t)\int_{\Omega(t)} 
\bigg[2 \partial_{12} \tilde{\psi}^{\bot}  \left(
\partial_{12}\tilde{\psi}^{\circlearrowleft}  -  \mathsf x_2[\theta] \partial_{12} \tilde{\psi}^{||}  \right) +  
  \partial_{11} \tilde{\psi}^{\bot}\! \left( \partial_{11} \tilde{\psi}^{\circlearrowleft}  - 
 \mathsf x_2[\theta] \partial_{11} \tilde{\psi}^{||} \right)
  \bigg]\,{ d}x.
  \end{aligned}
\]
Eventually, we introduce 
\[
\begin{aligned}
{\rm Rem}_{1}^{(a)}(t)&  =  \int_{\Omega(t)} \nabla \tilde{v}^{\bot} :  \nabla ( u -  \tilde{v})  \, dx  + {{\rm L}_{1}(t) } +{\rm D}_2(t)  \\
{\rm Rem_{1}^{(b)}} (t)& =\! - {\theta'(t)}\int_{\Omega(t)} \!\!
\bigg[2 \partial_{12} \tilde{\psi}^{\bot}  \left(
\partial_{12} \tilde{\psi}^{\circlearrowleft}  \!-\!  \mathsf x_2[\theta] \partial_{12} \tilde{\psi}^{||}  \right) \!+\!  
  \partial_{11} \tilde{\psi}^{\bot}\! \left( \partial_{11} \tilde{\psi}^{\circlearrowleft} \! -\! 
 \mathsf x_2[\theta] \partial_{11} \tilde{\psi}^{||} \right)
  \bigg]\,dx.
\end{aligned}
\]
Recalling \eqref{eq_D2} and \eqref{eq_L1},  we have the expected estimate for ${\rm Rem}_1^{(a)}.$ Applying a Korn equality to handle the first term of ${\rm D}_1(t)$ and
setting  ${\rm Rem}_1 = {\rm Rem}_1^{(a)} + {\rm Rem}_1^{(b)}$ we have then:
\begin{multline*}
{\rm D}_1(t)+{\rm D}_2(t) + {\rm D}_3(t)   =   - 2 \int_{A}  D(\tilde{v}^{\bot}) : D(u) \, dx + {\rm Rem}_{1}(t)  \\
+ \theta'(t)  \bigg(  
 \int_{\Omega(t)}  \partial_{22} \tilde{\psi}^{\bot}  \left(\partial_{22}\tilde{\psi}^{\circlearrowleft}   - 
\mathsf x_2 [\theta(t)] \partial_{22} \tilde{\psi}^{||}  
 \right) dx \\[6pt]
- 12 \int_{-\lambda_*}^{\lambda_*} \dfrac{ \left( \partial_{\theta}\gamma[\theta](\tau)  - \partial_{\theta} \mathsf x_2[\theta]\right) (\tau-c_*^{\bot})^2}{(\mathfrak d(t) + \gamma[\theta(t)](\tau) - \mathsf x_2[\theta(t)])^3} {d}\tau  
\bigg)\,.
\end{multline*}
Recalling {\rm Mod} as defined in \eqref{eq:Mod}, we have by a simple combination:
\[
\left(a {P_c^0}\right)'(t) = -2  a \int_{A}  D(\tilde{v}^{\bot}) : D(u) \, dx +  a {\rm Rem}_{1}(t)  +  {\rm Mod}(t)\,.
\]
We recognize in the first term on the right-hand side the dissipation involved in the weak-formulation \eqref{eq_Tw'} 
with $w=a \tilde{v}^{\bot}. $ Since $a \tilde{v}^{\bot} \in H^1((T_-,T_+) \times A),$ with $a \tilde{v}^{\bot} = a \widehat{e}_2$ on $B(t)$ we replace:
\begin{multline*}
-2  a \int_{A}  D(\tilde{v}^{\bot}) : D(u) \, dx=  \frac{{d}}{{d}t} \left[\int_{\Omega(\cdot)} u \cdot a \tilde{v}^{\bot} \, dx+ m h' a \right]  + H_{h}(h,\theta) \, a   \\
\begin{aligned}
 & -  \left( \int_{\Omega(t)} \left( u \cdot \partial_{t} (a \tilde{v}^{\bot}) + (u\cdot \nabla)\,(a \tilde{v}^{\bot}) \cdot u\right) \, dx + m\,{h'}\,{a'}  \right) \\
& +  \int_{\Omega(t)} \left( (u\cdot \nabla)\,s \cdot w +(s\cdot \nabla)\,u \cdot (a \tilde{v}^{\bot}) \right) \, dx   - \int_{\Omega(t)} \widehat{g} \cdot (a \tilde{v}^{\bot})\, dx\,.
\end{aligned}
\end{multline*}
Gathering the terms on the second and third lines into the remainder ${\rm Rem}_2$ and remarking that
\[
P_c =  a {P_c^0} - \left( \int_{\Omega(t)} u \cdot a \tilde{v}^{\bot}\, dx + m h' a\right),
\] 
we obtain finally
\[
{P_c}'(t) =   H_{h}(h(t),\theta(t))a(t)+    {\rm Mod}(t) + a(t)  {\rm Rem}_1(t) + {\rm Rem}_2(t)\,,
\]
which is the expected identity.
\end{proof}

\subsection{Lower bound on the distance}
Combining Lemma \ref{lem:H'}  with Lemma \ref{lem_H}, a lower bound on $\sqrt{\mathfrak d(t)}$ then derives from two more lemmas providing estimates on each of the pieces in the decomposition of ${P_c}'$. 
\begin{lemma} \label{lem_Mod}
We can find two constants $0 < \underline{a} \leq \overline{a},$ and another constant ${C_{geo}}$ all three depending only on $e$
such that
\begin{itemize}
\item[(a)]  $ \underline{a} \leq a(t) \leq \overline{a}$ for all $ t\in [T_-,T_+],$\\[-8pt]
\item[(b)]    $|{a}'(t)| \leq  C_{geo} \overline{a}|\theta'(t)|$ for all $t \in (T_-,T_+)\,.$ \\[-10pt]
\end{itemize}
 Furthermore, there holds
\[
 \int_{T_-}^{t} |{\rm Mod}(\tau)|d\tau  \leq C_{geo}  \sqrt{t-T_-} \left( \int_{T_-}^t \|\nabla u(\tau)\|^2_{L^2(A)}d\tau \right)^{\frac 12}.
\]

\end{lemma}
\begin{lemma} \label{lem_Rem}
{
There exists a constant $C_{dyn}$ that depends on $m,J,L,e$ such that, for all $t < T_+$,   we have:
\begin{multline*}
\int_{T_-}^{t} \big(  a(\tau) |{\rm Rem}_1(\tau)| + |{\rm Rem}_2(\tau)| \big) d\tau \\
\leq C_{dyn} \left(   \sqrt{t-T_-} \left( \int_{T_-}^t\|\nabla u(\tau)\|^2_{L^2(A)}d\tau    \right)^{\frac 12} +   \int_{T_-}^t\|\nabla u(\tau)\|^2_{L^2(A)}d\tau  + \lambda_0^2 (t-T_-)\right)  . 
\end{multline*}
}
\end{lemma}
The proofs of Lemma \ref{lem_Mod} and Lemma \ref{lem_Rem} are rather technical and postponed to the very end of the section in order to keep the focus on the estimate of $\sqrt{\mathfrak d(t)}$.
As a first step towards the proof of the lower bound stated in Assertion (ii) of Theorem \ref{th:uniform_distance}, we establish a lower bound restricted to the (possibly bounded) time interval $(0,T_M)$.
\begin{proposition}\label{prop:uniform_distance}
Assume $\lambda_0 \leq \lambda_{0}^{(0)}$. There exist 
 $\lambda_{0}^{(1)}>0$ depending on $m, {J}, e,L$ such that
 \begin{equation}\label{eq_minimaldistance-local}
  {\lambda_0} \le \lambda_0^{(1)} {k_0}\Rightarrow  {\rm dist}(B(t), \partial A) \geq {\mathfrak d}_{min}(E_0,k_0) \qquad \forall t \in(0,T_M)\,,
 \end{equation}
 where $ {\mathfrak d}_{min}$ is a monotone increasing function of $k_0$ and decreasing in $E_0$.
\end{proposition}
\begin{proof}
Injecting Lemma \ref{lem_Mod} and Lemma \ref{lem_Rem} into the integration of \eqref{eq_Hdot} and recalling \eqref{eq_lbF} and \eqref{eq_estimate_01},  we infer that
\[
\begin{aligned}
P_c(t) &\leq P_c(T_-)\! +\!  C_{dyn}\left(E_0+\lambda_0^2+S_M\right)^{1/2}\!\!\sqrt{t \!-\! T_-}\!+\!C_{dyn}\left(E_0+\lambda_0^2+S_M\right) +C_{dyn}^-(t\!-\!T_-),
\end{aligned}
\]
where: 
\[
C^{-}_{dyn}=C_{dyn}\left(\sqrt{\eta_0}S_M+\lambda_0^2\right)^{\tfrac{1}{2}}+{C_{dyn}}\left(\sqrt{\eta_0}S_M+2\lambda_0^2\right)- \underline{a} k_0\,.
\]
We now choose $\eta_0,\lambda_0$ such that $\sqrt{\eta_0}S_M$ and $\lambda_0^2$ are small enough compared to $k_0^2$ so that 
$$C_{dyn}^-\le -\frac{\underline{a} k_0}{2}<0.$$ 
The previous estimate then yields that
\begin{equation}\label{eq:encoreune}
 P_c(t) \leq P_c(T_-) + C_{dyn}\left(E_0+\lambda_0^2+S_M\right)\left(1+\dfrac{1}{k_0}\right) \quad \forall \, t \in  [T_-,T_+]. 
\end{equation}
Moreover,  Lemma \ref{lem_H} implies 
\begin{equation}\label{eq:encoreuneautre}
P_c(T_-)  \leq  \dfrac{C_{geo}^{(max)} }{\sqrt{\eta_0}} +  C_{dyn}\left(E_0+\lambda_0^2+S_M\right)^{1/2}\,.
\end{equation}
Using Lemma \ref{lem_H} again and the choice of $\eta_0$, we eventually conclude that 
\begin{equation} \label{eq:eq_estimate_dist}
\begin{aligned}
\underline{a}\dfrac{ C_{geo}^{(min)}}{\sqrt{\mathfrak d(t)}}  
\leq  C_{dyn} +  C_0^+\left( 1+\dfrac{1}{k_0}+\dfrac1{k_0^2}\right)
\end{aligned}
\end{equation}
as long as  $t \in [T_-,T_+].$ Hence, the conclusion follows. 
\end{proof}
In the sequel, given $E_0$, we simply denote ${\mathfrak d}_{min}(E_0,k_0)$ by ${\mathfrak d}_{min}^0$.
\begin{remark}
It is important to emphasize that the lower bound found in \eqref{eq:eq_estimate_dist} does not depend on the length of the time interval $(T_-,T_+)$ so that it does not depend on the exact value of the time $T_M$. In particular, it becomes uniform in time when $T_M=+\infty$. Observe also that the elastic force $H_h$ brings the critical contribution in $C_{dyn}^-$ that provides the global control in time of the distance to the boundary. In absence of a vertical restoring force, we still have a bound which degenerates when $|T_+-T_-|\to\infty$ (and is henceforth valid only in finite time). 
\end{remark}

\begin{remark}\label{remark_viscosity}
The threshold for $\lambda_0$ in Proposition \ref{prop:uniform_distance} does basically depend on $k_0$ only as seen from the expression of $C^-_{dyn}$. In particular, $\lambda_0^{(1)}$ depends only on the solid description through $C_{dyn}$ and is independent of $\rho$ and $\mu^2$. Going back to the physical variables, this means the threshold is a smallness condition of the pressure drop $\mathcal P_0$ compared to the stiffness of the elastic force $F_h$.
Observe that the choice $\eta_0$ is also independent of $\rho$ and $\mu^2$ but it does depend decreasingly on $E_0$. It then follows from \eqref{eq:encoreune}-\eqref{eq:encoreuneautre} that the estimate of the minimal distance does not improve for large viscosity.  
\end{remark}

\black
\medskip

In order to complete the proof of Theorem \ref{th:uniform_distance}, and now that we have bounded $\mathfrak d(t)$ as long as $|\theta|<\theta_M$, it is enough to show that indeed $\theta$ never approaches $\theta_M$. This is the objective of the next subsection.

\subsection{Uniform bound on $|\theta|$ and conclusion of the proof of Theorem \ref{th:uniform_distance}} 
 We assume that $\lambda_0 \le  \lambda_{0}^{(1)}k_0$, so that \eqref{E_kin^M}-\eqref{eq_minimaldistance-local} apply to any solution $(u,h,\theta)$ to \eqref{eq:hom_pb1}-\eqref{eq:ODE_newref1} on $(0,T_M)$. Given an initial energy $E_0$, in view of \eqref{coercivity-bis},  $\theta_M$ is fixed by $\alpha_0$ and  so are the constants $S_M$ and $\frak{d}^0_{min}$. 
As already said, we aim to prove that $T_M=+\infty$, which then proves Assertion (ii) of Theorem \ref{th:uniform_distance}. This will be a consequence of the next theorem, which in fact contains Assertion (iii). 
\begin{theorem} \label{cor_estimation1}
Assume  $\lambda_0 \leq \min(\lambda_{0}^{(0)},\lambda_0^{(1)}{k_0})$. 
Let $\omega=\omega_0^-(E_0)$, $\sigma\in(0,1)$ and $C_{geo}>0$ be given by Proposition \ref{thm_dissipation 1}.
If 
\begin{equation}\label{jesperequecestladerniere}
2C_{geo}\lambda^2_0\le \sigma\omega\bar \varrho, 
\end{equation}
then	
$$
|\theta(t)|\le \alpha_0 \qquad\forall t\ge 0\,,
$$
and there holds
	\begin{equation}\label{uniform_estimate_en}
		E_{tot}(t)\le   3E_0\,e^{-\sigma\omega t}+  \frac{2C_{geo}\lambda_0^2}{\sigma\omega}\qquad\forall t\ge 0.
	\end{equation}
\end{theorem}
\begin{remark}\label{remark_viscosity2}
Theorem \ref{cor_estimation1} requires three smallness condition on $\lambda_0$. 
The first one is used to get the basic energy estimate of Theorem \ref{th:existence_weak_local}, see Remark \ref{rem-small}.	As emphasized in Remark \ref{remark_viscosity}, the second one boils down to a comparison between the pressure drop and the stiffness of the elastic force $F_h$. The third one, namely \eqref{jesperequecestladerniere}, is in fact of the same type as the first two. Indeed, it will be seen that $\omega$ has to satisfy two inequalities, see \eqref{etmerdecetaitpasladerniere}.  If we fix all the physical parameters but $\rho$ and $\mu$, the smallness condition \eqref{jesperequecestladerniere} means $\lambda_0^2$ should be small in front of $\bar \varrho^2$. In terms of physical variables, this is again a comparison between the pressure drop and the stiffness of the elastic force $F_h$.
Overall, in the initial physical problem, as soon as the stiffness of the elastic force $F_h$, the density and viscosity of the fluid are fixed, we can always define a threshold on the pressure drop so that all smallness conditions are satisfied. If the pressure drop is given and is compatible with the density and viscosity of the fluid to get Theorem \ref{th:existence_weak_local}, the result then applies if the stiffness of the elastic force $F_h$ is big enough. 
\end{remark}

Our main ingredient to prove Theorem \ref{cor_estimation1} is to adapt a strategy introduced by Haraux for the wave equation, see \cite{haraux1985two} (see also \cite{denistheboss,gazzpatapat} in the context of fluid-solid interaction models). For $\omega\black\in (0,1)$ to be fixed later on, we add the term $\omega\black hh' + \omega\black \theta\theta'$ to the total energy of the solid.  However, in a fluid-solid context, this requires introducing also corresponding terms in the energy of the fluid. To this aim, we define a  solenoidal vector field $w: {A \times (0,T_M)} \to \mathbb{R}^{2}$ by 
\begin{equation}\label{w}
	w(x,t)=\left(-\frac{\partial }{\partial x_2}\left( \zeta(x)b(x,t)\right), \frac{\partial }{\partial x_1}\left(\zeta(x)b(x,t)\right)\right)\qquad \forall (x,t)\in { A  \times (0,T_M)}\,,
\end{equation}
where 
$$
b(x,t)\doteq h(t)x_1+\frac{x_1^2}{2}\theta(t)+\frac{\theta(t)}{2}(h(t)-x_2)^2 \qquad \forall (x,t)\in { A \times { (0,T_M)}}\,,
$$
and $\zeta \in \mathcal{C}^{\infty}(A)$ is a cut-off function equal to one in $[-2,2]\times[-L+ \mathfrak{d}_{min}^0,L- \mathfrak{d}_{min}^0]$, and equal to 0  for  $|x_2| <  L - \mathfrak{d}^{0}_{min}/2,$   or $|x_1| > 3$. By construction and the regularity of $t \mapsto {(h(t),\theta(t))}$, we therefore have { $(w, h,\theta) \in W^{1,\infty}({ (0,T_M)}\times A) \times W^{1,\infty}({ (0,T_M)})^2$,}  with 
\[
{\rm div}\, w(t,x)=0 \text{ in $\Omega(t)$}\,, \,\, w(t,x) = h(t)\widehat{e}_2 +  \theta(t)(x-h(t)\widehat{e}_2)^{\bot} \text{ on $B(t)$}\,,
\quad 
w (t,x)= 0 \text{ on $\partial A$}\,.
\]
It is easily seen that the following estimates hold for every $t \in{(0,T_M)}$:
\begin{equation}\label{stime_w}
	\left\{\ \begin{aligned}
		&\displaystyle \| w(t) \|_{L^2(A)} +\| w(t) \|_{L^{\infty}(A)}\le { {C_{geo}}\left(\mathfrak{d}_{min}^0\right)^{-1}}\,(|h(t)|+|\theta(t)|)\,, \\[3pt] 
		& \|\nabla w(t)\|_{L^2(A)} 
		+ \|\nabla w(t)\|_{L^{\infty}(A)}\le {{ {C_{geo}}\left(\mathfrak{d}_{min}^0\right)^{-2}}}\, (|h(t)|+|\theta(t)|)\,, \\[3pt] & \| \partial_t w(t)\|_{L^2(A)}\le{ {C_{geo}}\left(\mathfrak{d}_{min}^0\right)^{-2}} \, (|h'(t)|+|\theta'(t)|)\,, \end{aligned}\right.
\end{equation}
\black
for some constant  $C_{geo}>0$ depending on $e,L$.
We note also that $(w,h,\theta)$ is an admissible multiplier in the weak formulation \eqref{eq_Tw}-\eqref{eq_Tw'}. 
We then define a modified energy
\begin{equation} \label{eq_modifenergy}
	E_\omega(t)\doteq E_{tot}(t)+\omega\left(m\, h(t)h'(t)+ {J}\,\theta(t)\theta'(t)\black +\int_{\Omega(t)}u(t)\cdot w(t)\,dx\right), \quad t \in  [0,T_M),
\end{equation}
where $\omega\black$ will be fixed later on. {
	Arguing as in Theorem \ref{th:existence_weak_local}, i.e. using again \cite[Theorem 2.1]{bravin}, we infer that $E_\omega \in \mathcal{C}([{ 0,T_M}])$.}
To make precise the  choice of $\omega$, we first emphasize that $E_\omega$ is controlled by the total energy and in turn controls the total energy provided $\omega$ is small enough.
\begin{lemma} \label{prop_boundsE}
Let $C_{geo}>0$ be as in \eqref{stime_w}, $\lambda_0^{(1)},{\mathfrak d}_{min}^0$ be given by Proposition \ref{prop:uniform_distance} and assume that $\lambda_0\le\lambda_0^{(1)}{k_0}$. For any initial energy $E_0$,
	\begin{equation}\label{bounds_E}
0< \omega \leq \omega_0(E_0) =  C_{dyn}{\mathfrak d}_{min}^0\min(1, \overline{\varrho}/2)\ \Rightarrow		\   \dfrac{E_{tot}(t)}{2} \le  E_\omega (t)\le \dfrac{3}{2}E_{tot}(t)\,,
	\end{equation} 
	for all $t \in  [0 ,T_M)$.
\end{lemma}
\begin{proof}
Assume $\omega \leq \omega_0\le \frac12$ and 
$$C_{dyn}\omega_0\le \frac{{\mathfrak d}_{min}^0}4\min(1, \overline{\varrho}/2).$$
		Using the distance estimate of Proposition \ref{prop:uniform_distance}, Young inequality combined with the estimates in \eqref{stime_w} yield, for all $t\in [0 ,T_M)$,
		$$
		\begin{aligned}
			\omega\Big(m\left| h\,h'\right|+ {J}\left|\theta\theta'\right|  + & \Big|\int_{\Omega(t)}u\cdot w\,dx\Big|\Big) \\[6pt]
			& \le  \frac{1}{4}\left(\|u\|^2_{L^2(\Omega(t))}+ m|h'|^2 +{J}|{\theta'}|^2+ \frac{\overline{\varrho}}{2}\left( |h|^2+|\theta|^2\right)\right).
		\end{aligned}
		$$
	Recalling the assumption \eqref{hp_H}, the statement \eqref{bounds_E} immediately follows.
\end{proof}

The following proposition is a first step towards the proof of Theorem \ref{cor_estimation1}.
\begin{proposition} \label{thm_dissipation 1}
	Let $\lambda_0^{(1)},{\mathfrak d}_{min}^0$ be given by Proposition \ref{prop:uniform_distance}, $\omega_0(E_0)$ be given by Lemma \ref{prop_boundsE}, and assume that $\lambda_0\le\lambda_0^{(1)}{k_0}$. 
	  There exist $0<\sigma<1$, $C_{geo}>0$ and a decreasing function $\omega_0^-$ such that for any initial energy $E_0$,
\begin{equation}\label{eq_dissipation1_local}
0< \omega \leq \omega_0^-(E_0) \le \omega_0(E_0)\ \ \Rightarrow \ \	E_{\omega}(t_2) - E_{\omega}(t_1) + \sigma \omega \int_{t_1}^{t_2}E_{\omega}(\tau)d\tau 
		\leq  {C_{geo} \lambda_0^2}(t_2 - t_1 ), 	
\end{equation} 
for all $0\leq t_1 \leq t_2<T_M$.
\end{proposition}
\begin{proof}
Assume that $\omega\le \omega_0(E_0)$ as given in \eqref{bounds_E} by Lemma \ref{prop_boundsE}.
We first observe that \eqref{eq_dissipationsolide} implies there exists $C_{dyn}>0$ such that for $\omega\in(0,C_{dyn})$, we have
\begin{equation} \label{eq:estimationnumero1000}
\begin{aligned}
& \frac{7}{32} \int_{t_1}^{t_2}\|\nabla u(\tau)\|^2_{L^2(A)}d\tau  \\
& \quad \le  \dfrac{1}{4} \int_{t_1}^{t_2} \left(\|\nabla u(\tau)\|^2_{L^2(A)} - \omega \left( m{|h'(\tau)|}^2 + {J}{|\theta'(\tau)|}^2+\|u(\tau)\|^2_{L^2(A)}\right)\right) \, d\tau,
\end{aligned}
		\end{equation}
whereas we recall that \eqref{hp_F} entails
$$\overline{\varpi} H(h,\theta) \le H_h(h,\theta)h +   H_\theta(h,\theta)\theta.$$
Using $\omega\black (w,h,\theta)$ as test function in the weak formulation in the form \eqref{eq:weak_gen} and the energy estimate \eqref{eq:energy_estimate_weak}, we then deduce
\begin{equation} \label{eq:weak_gen-w}
			\begin{aligned}
			E_{\omega}(t_2) - E_{\omega}(t_1) & + \omega\overline{\varpi} \int_{t_1}^{t_2} H(h(\tau),\theta(\tau))  d\tau+\frac{7}{32} \int_{t_1}^{t_2}\|\nabla u(\tau)\|^2_{L^2(A)}d\tau   \\[6pt] 
			  \leq & \  C_{geo} {\, \lambda_0^2}{} (t_2 -t_1) - \omega\int_{t_1}^{t_2} \int_{A}\nabla u:\nabla w\,dxd\tau   + \omega\int_{t_1}^{t_2} \int_{\Omega(t)} \widehat{g} \cdot w \, dxd\tau   \\[6pt] & 
				- \omega\int_{t_1}^{t_2} \int_{\Omega(\tau)} \left( (u\cdot \nabla)\,s  +(s\cdot \nabla)\,u { + \widehat{f}[h] }
				\right)\cdot w \, dx d\tau 
				\\[6pt] & 
				+\omega\int_{t_1}^{t_2} \int_{\Omega(\tau)}{ \left( u \cdot \partial_{t} w + (u\cdot \nabla)\,{w} \cdot u\right)} \, dxd\tau  \\[6pt]
			\end{aligned}
		\end{equation}
		for arbitrary $0 \leq t_1 \leq t_2<T_M$.
\begin{equation}\label{E_omegaaa}
\begin{aligned}
	&		 E_{\omega}(t_2) - E_{\omega}(t_1) + \omega\overline{\varpi} \int_{t_1}^{t_2} H(h(\tau),\theta(\tau))  d\tau+\frac{3}{16} \int_{t_1}^{t_2}\|\nabla u(\tau)\|^2_{L^2(A)}d\tau   \\[6pt]  
&\quad \le   C_{geo}{\lambda_0^2}{} (t_2 - t_1 ) + C_{geo}\omega \int_{t_1}^{t_2} \left(\omega\|w\|^2_{H^1(A)} + \omega\|\partial_t w\|^2_{L^2(A)}+\|w\|_{L^\infty(A)} \|\nabla u\|^2_{L^2(A)}\right)d\tau. \\[6pt] 
		\end{aligned}
\end{equation}
Since we know by Proposition \ref{prop:uniform_distance} that the solution satisfies \eqref{eq_minimaldistance-local}, we can exploit the estimates in \eqref{stime_w} and \eqref{hp_H}-\eqref{eq_dissipationsolide} to get 

\begin{equation}\label{E_omegaa}
\begin{aligned}
			 & E_{\omega}(t_2)  - E_{\omega}(t_1) + \omega\overline{\varpi}\left(1-\frac{C_{geo}\omega}{\overline{\varrho}\,\overline\varpi(\mathfrak{d}_{min}^0)^4}\right)   \int_{t_1}^{t_2} H(h(\tau),\theta(\tau))  d\tau   \\[6pt] 
			  & \quad  +\left(\frac{3}{16}-C_{geo}\frac{\omega}{(\mathfrak{d}_{min}^0)^2}\left(L+\theta_M+\frac{\omega}{(\mathfrak{d}_{min}^0)^2}\right)\right) \int_{t_1}^{t_2}\|\nabla u(\tau)\|^2_{L^2(A)}d\tau   
			\le   C_{geo}{\lambda_0^2}{} (t_2 - t_1 )  \,.
		\end{aligned}
\end{equation}
Therefore, using \eqref{eq_dissipationsolide} one more time and \eqref{eq:estimationnumero1000}, it is easily seen, taking a smaller constant $\tilde C_{geo}$ if necessary, that if $\omega\in(0,C_{dyn})$ satisfies
\begin{equation}\label{etmerdecetaitpasladerniere}
\omega\le \tilde C_{geo}\min(\overline{\varrho}\, \overline\varpi,1)\,{(\mathfrak{d}_{min}^{0})}^{4} \text{ and } 
(L+\theta_M)\omega\le \tilde C_{geo}{(\mathfrak{d}_{min}^{0})}^{2}, 
\end{equation}
it holds
\begin{equation}\label{E_omega33}
	E_{\omega}(t_2) - E_{\omega}(t_1) + \frac{\overline{\varpi}}2\omega \int_{t_1}^{t_2} H(h(\tau),\theta(\tau))  d\tau+\omega \int_{t_1}^{t_2}E_{kin}(\tau)d\tau 
	\leq   C_{geo}{\lambda_0^2} (t_2 - t_1 ).
\end{equation}
Using \eqref{bounds_E}, this eventually leads to the estimate
$$
E_{\omega}(t_2) - E_{\omega}(t_1) + \sigma\omega \int_{t_1}^{t_2}E_{\omega}(\tau)d\tau 
\leq   C_{geo}{\lambda_0^2} (t_2 - t_1 )
$$
for some $0<\sigma<1$ whose choice is independent of $E_0$.
It is clear that since $\mathfrak{d}_{min}^{0}$ is a decreasing function of $E_0$ and the initial choice of $\theta_M$ is increasing with $E_0$, the threshold $\omega_0^-$ is a decreasing function of $E_0$.
\end{proof}

 We are now ready to give the proof of Theorem \ref{cor_estimation1}.

\begin{proof}[Proof of Theorem \ref{cor_estimation1}]
 By contradiction, assume that $T_M<\infty$. Then,  as $\theta \in \mathcal{C}([0,T_M])$, it must be that $\theta(T_M)=\theta_M$. We infer from \eqref{eq_dissipation1_local} that for some $0<\sigma<1$
\begin{equation}\label{E_omega3}
	\begin{aligned}
		E_{\omega}(t_2) - E_{\omega}(t_1) + \sigma\omega \int_{t_1}^{t_2}E_{\omega}(\tau)d\tau 
		\leq   C_{geo}{\lambda_0^2} (t_2 - t_1 ) \,\,,
	\end{aligned}
\end{equation}
for all $0\le t_1\le t_2<T_M$ and $\omega= \omega_0^-(E_0)$. Since ${E_\omega}$ is continuous with time, Gr\"onwall Lemma implies that
\begin{equation}
	{E_\omega}(t)\le {E_\omega}(0)\,e^{-\sigma \omega t} +  \frac{C_{geo}\lambda^2_0}{\sigma\omega} \quad\forall t\in (0,T_M)\,.
\end{equation}
In particular, we deduce from \eqref{bounds_E} that 
\begin{equation}\label{E_omega4_local}
	{E}_{tot}(t)\le 2 {E_\omega}(t) \le 3 E_0 +  \frac{2C_{geo}\lambda^2_0}{\sigma\omega} \quad\forall t\in (0,T_M)\,.
\end{equation}
We now impose a new threshold on $\lambda_0$, namely we assume 
$$
2C_{geo}\lambda^2_0\le \sigma\omega_0^-(E_0)\bar \varrho.
$$
This latter bound entails that 
\[
H(h,\theta) \leq  
3E_0+\overline{\varrho}\,.
\]
Exploiting \eqref{coercivity-bis}, we thus deduce the estimate 
$$
|\theta(t)|\le \alpha_0=\theta_M-1 \qquad \forall t\in (0,T_M]\,, 
$$
which gives a contradiction, thus $T_M=+\infty$. We can  now conclude that  
\eqref{E_omega4_local} holds for all $t\ge 0$. We finally observe that 
$$E_{tot}(t)\le 2{E_\omega}(t)\le 2{E_\omega}(0)\,e^{-\sigma\omega t} +  \frac{2C_{geo}\lambda_0^2}{\sigma\omega}\le3E_0\,e^{-\sigma\omega t} +\overline{\varrho} \,,$$
which concludes the proof. 
\end{proof}

\subsection{Proof of Lemma \ref{lem_Mod}}

We construct $a$ by solving a differential equation depending on $\theta$ and $\theta'.$ Below we drop time-dependencies for a better legibility. {In order to analyze one by one the terms involved in the definition of {\rm Mod}(t) in \eqref{eq:Mod}, we will repeatedly use the following tools: the local Taylor expansion \eqref{eq_expansiongamma}-\eqref{eq_expansiongamma_der}, the description of the gap $\mathcal{G}_{\lambda_*}$ in terms of the variables introduced in \eqref{change_variables} and the asymptotic expansions in Appendix \ref{app_asymptotics}.}
With these tools, we deduce that
for all $\mathfrak{d} < \eta_0$ and $\theta \in \mathbb R:$
\begin{equation}
\label{eq:assymp}
\begin{aligned}
\int_{-\lambda_*}^{\lambda_*} \dfrac{(\tau-c_*^{\bot})^2}{(\mathfrak d + \gamma[\theta](\tau) - \mathsf x_2[\theta])^2} {d}\tau &= \dfrac{I_{2,2}(\kappa_2[\theta])}{\sqrt{\mathfrak d}} + \mathcal{O}(1)\,, \\[6pt]
\int_{-\lambda_*}^{\lambda_*} \dfrac{\left(\partial_{\theta}\gamma[\theta](\tau){-\partial_{\theta} \mathsf x_2[\theta]}\right) (\tau-c_*^{\bot})^2}{(\mathfrak d + \gamma[\theta](\tau)  - \mathsf x_2[\theta])^3} {d}\tau   & = \dfrac{I_{4,3}(\kappa_2[\theta]) \partial_{\theta} \kappa_2[\theta]}{\sqrt{\mathfrak d}} + \mathcal{O}(1)\,,
\end{aligned} 
\end{equation}
where the symbols $I_{2,2}$ and $I_{4,3}$ stand for  the integrals
\[
I_{4,3}(\kappa_2)  \dot{=} \int_{\mathbb R} \dfrac{\tau^4}{(1 + \kappa_2 \tau^2)^3}  \, d\tau,
\qquad 
I_{2,2}(\kappa_2) \dot{=}   \int_{\mathbb R} \dfrac{\tau^2}{(1 + \kappa_2 \tau^2)^2}   \, d\tau,
\]
which are smooth and bounded function of {$\kappa_2\in [\kappa_2^{(min)},\kappa_2^{(max)}]$}.
Using the asymptotic estimates \eqref{eq:assymp}, the first and third terms can be written as 
$$\dfrac{ 6 {a}' I_{2,2}(\kappa_2)- 12 a\,\theta'  \partial_{\theta} \kappa_2 I_{4,3}(\kappa_2)}{\sqrt{\mathfrak d}} + \mathcal{O}(|a'|+|\theta'||a|).$$
To estimate the second term in \eqref{eq:Mod}, we first observe that since the three functions $\tilde{\psi}^{\bot}$ and $\tilde{\psi}^{||},\tilde{\psi}^{\circlearrowleft}$ are uniformly bounded outside $\mathcal G_{\lambda_*}$ (recall \eqref{stime_test}), we have
\[
 \int_{\Omega}  \partial_{22} \tilde{\psi}^{\bot}  \left(\partial_{22}\tilde{\psi}^{\circlearrowleft}   - 
\mathsf x_2[\theta] \partial_{22} \tilde{\psi}^{||}  
 \right)\, dx =
  \int_{\mathcal G_{\lambda_*}}  \partial_{22} \psi_{opt}^{\bot}  \left(\partial_{22}\psi_{opt }^{\circlearrowleft}   - 
\mathsf x_2 [\theta] \partial_{22} \psi_{opt}^{||}  
 \right) \, dx   + \mathcal{O}(1).
\]
Inside $\mathcal G_{\lambda_*},$ we can use the explicit values recalling that
\[
\int_{0}^1 |\partial_{ss} P_{1}^{opt}(r)|^2\,dr = 12 \, , \qquad
\int_0^{1} \partial_{ss} P_{1}^{opt}(r) \partial_{ss} P_{2}^{opt}(r)\,dr = -6 \, .
\]
Recalling the explicit expressions \eqref{psising}, we infer that
(implicitly, boundary functions are evaluated in $({\mathsf x}_1+\tau,h+\gamma[\theta](\tau))$, we drop dependencies  for legibility):
\begin{equation}
\begin{aligned}
 & M_{opt}\doteq \int_{ \mathcal G_{\lambda_*}}   \partial_{22} \psi_{opt}^{\bot}  \left(\partial_{22}\psi_{opt }^{\circlearrowleft}   - 
\mathsf x_2 [\theta] \partial_{22} \psi_{opt}^{||}  
 \right) \, dx\\[6pt] 
&= \int_{-\lambda_*}^{\lambda_*} \int_{-L}^{\gamma[\theta](\tau) + h}  
 \dfrac{(\psi_*^{\bot} -c_*^{\bot})((\psi_*^{\circlearrowleft} -c_{*}^{\circlearrowleft})- {\mathsf x}_2 (\psi_*^{||} - c_*^{||})) }{(\mathfrak d + \gamma[\theta](\tau) - \mathsf x_2[\theta])^4} |\partial_{ss} P_1^{opt}  (r({\mathsf x}_1+ \tau,x_2))|^2
  { d}x_2 { d}\tau  \\
& + \!\! \int_{-\lambda_*}^{\lambda_*}\!\! \int_{-L}^{\gamma[\theta](\tau) + h}  \!\!
 \dfrac{(\psi_*^{\bot} - c_*^{\bot})(\partial_2 \psi_*^{\circlearrowleft} - {\mathsf x}_2 \partial_2 \psi_*^{||} )}{(\mathfrak d + \gamma[\theta](\tau) - \mathsf x_2[\theta])^3} \partial_{ss} P_1 (r({\mathsf x}_1\! +\! \tau,x_2))\partial_{ss} P_2^{opt}  (r({\mathsf x}_1\! +\! \tau,x_2)) { d}x_2 { d}\tau  \\
 & = \int_{-\lambda_*}^{\lambda^*} \dfrac{(\psi_*^{\bot} - c_*^{\bot})}{(\mathfrak d + \gamma[\theta](\tau) - \mathsf x_2[\theta])^3} 
 \left( 12 ((\psi_*^{\circlearrowleft} -c_{*}^{\circlearrowleft})- {\mathsf x}_2 (\psi_*^{||} - c_*^{||}) ) - 6 (\partial_2 \psi_*^{\circlearrowleft} - {\mathsf x}_2 \partial_2 \psi_*^{||} ) \right){ d}\tau.
 \\ 
 \end{aligned}
 \end{equation}
 Replacing boundary functions with their explicit expressions and recalling expansion of singular integrals computed in Appendix \ref{app_asymptotics},  we
 conclude that $M_{opt} = M_{opt}^{\infty} + \mathcal O(1)$, where
\[
\begin{aligned}
 M_{opt}^{\infty} \!=&\!\! \int_{-\lambda_*}^{\lambda_*} 
\dfrac{\tau - c_*^{\bot}}{(\mathfrak d + \gamma[\theta](\tau) - \mathsf x_2[\theta])^3}
\left(6 ( \tau^2 - c_*^{\circlearrowleft}) 
+ 12\mathsf{x}_2[\theta] \left(\gamma[\theta](\tau)- \mathsf x_2[\theta] + c_*^{||} \right)\right)\,d\tau
\\&+ 6\int_{-\lambda_*}^{\lambda_*} 
\dfrac{\tau - c_*^{\bot}}{(\mathfrak d + \gamma[\theta](\tau) - \mathsf x_2[\theta])^3}
\left(\mathfrak d + \gamma[\theta](\tau)-{\sf{x}}_2[\theta]\right)d\tau\,.
\end{aligned}
\]
Applying again expansion of singular integrals, we prove infer
\[
\left|
M_{opt}^{\infty}  - \dfrac{\kappa_3 I(\kappa_2)}{\sqrt{\mathfrak d}}  \right| \leq C_{geo},
\]
where $I(\kappa_2)$ is a combination of functions  depending smoothly on $\kappa_2 \in (0,\infty).$ We refer to the Appendix \ref{I_k2} for more details on the expansion  of $M_{opt}^{\infty}$.  We now conclude
\begin{equation}\label{Mod_new}
\left|{\rm Mod}(t)-  \dfrac{ 6 {a}' I_{2,2}(\kappa_2) {+} a\theta' \left( \kappa_3  I(\kappa_2) - 12 \partial_{\theta} \kappa_2 I_{4,3}(\kappa_2)\right)  }{\sqrt{\mathfrak d}} \right| \leq C_{geo}\left(|{a}'| + |\theta'||a| \right) \, .
\end{equation}
Hence, we fix $a$ by setting the term in factor of $1/\sqrt{\mathfrak d}$ to $0$, i.e. defining $a$ by 
\[
a'(t) =   {-}\dfrac{\theta' a}{6 I_{2,2} (\kappa_2)} \left( \kappa_3 I(\kappa_2) - 12 \partial_{\theta}\kappa_2 I_{4,3}(\kappa_2)
\right). 
\]
At this point, we recall that Lemma \ref{lem_geom} implies $\kappa_3[\theta] = \partial_{\theta} \mathcal K_3(\kappa_2[\theta])$ for some $\mathcal{C}^1$-function $\mathcal K_3 : [\kappa_2^{min},\kappa_2^{max}] \to \mathbb R$. 
We then deduce that 
\[
\begin{aligned}
a(t)  & = \exp\left( -\int_{T_-}^{t}  \dfrac{\theta'(\tau)\partial_{\theta}\kappa_2[\theta(\tau)] }{6 I_{2,2}(\kappa_2[\theta(\tau)])} 
\left(  \mathcal K_3'(\kappa_2[\theta(\tau)])] I(\kappa_2[\theta(\tau)]) - 12 I_{ {4},3}(\kappa_2[\theta(\tau)]) 
\right)\, d\tau  
\right) \\
& = \exp\left( \mathcal G(\kappa_2[\theta(t)]) - \mathcal G(\kappa_2[\theta(T_-)]) \right) \,,
\end{aligned}
\]
for a $\mathcal{C}^1$-function $\mathcal G$ {depending only on $B.$}
Since $\kappa_2$ is a periodic function of $\theta,$ we infer that $a$ is bounded from above and from below by strictly positive constants, showing (a) holds. The inequality (b) for ${a}'$ follows. Hence, the first term in \eqref{Mod_new} vanishes and by inequalities (a)-(b)-\eqref{eq_dissipationsolide} it follows that
\[
| {\rm Mod}(t)| \leq {C_{geo}}  |\theta'| \leq {C_{geo}}  \|\nabla u\|_{L^2(A)}\,.
\]
This completes the proof.
\qed

\subsection{Proof of Lemma \ref{lem_Rem}}
We start with ${\rm Rem}_1.$ We recall that we splitted ${\rm Rem}_1 = {\rm Rem}_{1}^{(a)}+ {\rm Rem}_{1}^{(b)}$ as defined by \eqref{eq_Rem1a}-\eqref{eq_Rem1b}. 
We estimate both terms independently. 
Concerning ${\rm Rem}_1^{(a)},$ we apply first Corollary \ref{cor:finer-trace} that entails
\[
\frac{|{\mathfrak d}'|}{\sqrt{\mathfrak d}} \leq  C_{geo} \|\nabla u\|_{L^2(A)}\,. 
\]
To estimate the other term in ${\rm Rem}_{1}^{(a)}$, we remark that $u - \tilde{v}$ is divergence-free and $H^1_0(\Omega).$ Consequently, we can apply estimate \eqref{eq_stokes_optim} to yield a  constant $C_{geo}$ depending only on the construction of $\tilde{v}^{\bot}$ for which:
\[
\left|\int_{\Omega} \nabla \tilde{v}^{\bot} : \nabla (u - \tilde{v}) \, dx \right| \leq  C_{geo} \left(  \|\nabla u\|_{L^2(A)} + \|\nabla \tilde{v}\|_{L^2(A)}\right) \, ,
\]
where, by construction of $\tilde{v}$ in \eqref{vtilde_new} and by Corollary \ref{cor:finer-trace}, we deduce
\[
 \|\nabla \tilde{v}\|_{L^2(A)} \leq  C_{geo} \left(\dfrac{|\theta'|}{\mathfrak d^{\frac 14}} + \dfrac{|{\mathfrak d}'|}{\mathfrak d^{\frac 34}} \right) \leq  C_{geo}\|\nabla u\|_{L^2(A)}\,.
\]
We now focus on ${\rm Rem}_1^{(b)}$.  Once again, as in the proof of Lemma \ref{lem_Mod}, we use the local Taylor expansion \eqref{eq_expansiongamma}-\eqref{eq_expansiongamma_der}, the description of the gap $\mathcal{G}_{\lambda_*}$ in terms of the variables introduced in \eqref{change_variables} and the asymptotic expansions in Appendix \ref{app_asymptotics}.  Then, we infer from the uniform boundedness of $\tilde{\psi}$ outside $\mathcal G_{\lambda_*}$ (see \eqref{stime_test}) and the explicit formulas inside $\mathcal G_{\lambda_*}$ that
\[
\begin{aligned}
|{\rm Rem}_{1}^{(b)}|&  \leq  |{\theta'}| \bigg[ \int_{\mathcal G_{\lambda_*}} 
\Big( 2|\partial_{12} \psi_{opt}^{\bot}| \left(
|\partial_{12} \psi_{opt }^{\circlearrowleft} | +  |{\sf{x}_2}||\partial_{12} \psi_{opt}^{||} |  \right) 
 \\ & \quad\quad\quad +  
 | \partial_{11} \psi_{opt}^{\bot}|  \left( |\partial_{11} \psi_{opt} ^{\circlearrowleft} |  + |{\sf{x}_2}| 
 |\partial_{11} \psi_{opt}^{||} | \right)
  \Big)\, dx  + \mathcal{O}(1) \bigg]
 \\[6pt]
 & \leq 
  |\theta'| \left( \int_{-{\lambda_*}}^{\lambda_*}
\dfrac{\left(|\tau| + \mathcal{O}(\mathfrak d) \right) }{(\mathfrak d + \gamma[\theta](\tau)-{\sf x}_2)^{2}}  \left(  \tau^2 +  (\gamma[\theta](\tau)-{\sf x}_2)^2 + \mathcal{O}(\mathfrak d) 
\right){d}\tau   + \mathcal{O}(1) \right)  \\[6pt]
& \leq C_{geo} |\theta'| (| \ln(\mathfrak d)| +1)\,. \bigskip
  \end{aligned}
\]
Here we apply again Corollary \ref{cor:finer-trace} in order to conclude that
\[
|{\rm Rem}_1^{(b)}| \leq C_{geo} \|\nabla u\|_{L^2(A)}\,.
\]
We have then a similar bound for ${\rm Rem}_1^{(b)}$ as for ${\rm Rem}_{1}^{(a)}$ and we end up with 
\begin{equation} \label{eq_Rem1ctrl}
 \int_{T_-}^{t} a(\tau)|{\rm Rem}_1 (\tau)|\, d\tau \leq   
{C_{geo}  \sqrt{t-T_-} \left( \int_{T-}^t \|\nabla u(\tau)\|^2_{L^2(A)}\, d\tau\right)^{\frac 12} }
\end{equation}
with $C_{geo}$  depending again only on $e,L$.

We finally proceed with estimating ${\rm Rem}_2.$ Given $t > T_-,$ we have
\[
\int_{T_-}^{t} {\rm Rem}_2(\tau)\, d\tau = -J_1 -  J_2 + J_3 - J_4 \, ,
\]
where
\begin{equation}\label{eq_J1} 
\left\{\begin{aligned}
J_1 &\doteq \int_{T_-}^{t}  \int_{\Omega(\tau)} [\partial_{t} (a(\tau)\tilde{v}^{\bot}) + (u \cdot \nabla) (a(\tau)\tilde{v}^{\bot})] \cdot u \, dxd\tau  \\[6pt]
J_2 &\doteq \int_{T_-}^{t} m h'(\tau)a'(\tau) \, d\tau \\[6pt]
J_3 & \doteq \int_{T_-}^{t}  \int_{\Omega(\tau)} \left((u\cdot \nabla)\,s  +(s\cdot \nabla)\,u \right) \cdot a(\tau) \tilde{v}^{\bot}  \, dx  d\tau\\[6pt]
J_4 & \doteq \int_{T_-}^{t} \int_{\Omega(\tau)} \widehat{g} \cdot a(\tau) \tilde{v}^{\bot}  \, dx \, . 
\end{aligned}\right.
\end{equation}\\[2mm]
The following estimate
\begin{equation}\label{eq_J1estim} 
|J_1| \leq C_{dyn} \int_{T_-}^{t} \|\nabla u(\tau)\|^2_{L^2(A)}\, d\tau
\end{equation}
is established in Section \ref{app_J1} of Appendix \ref{app_asymptotics}.
{The Poincaré inequality in \eqref{eq_dissipationsolide} combined with assertion (b) of Lemma \ref{lem_Mod}  yield a similar estimate for $J_2.$

For $J_3$ and $J_4$, we use Lemma \ref{lemma:extension}, assertion (a) of Lemma \ref{lem_Mod} and \eqref{stime_test}-\eqref{eq_estimate_0} {with $\lambda_0 \leq \lambda_0^{(0)}$ one more time}, we infer
\[
\begin{split}
|J_3|   & \leq {C_{dyn}  \lambda_0\int_{T_-}^{t} \|\nabla u(\tau)\|^2_{L^2(A)}\, d\tau}  \\
|J_4|  & \leq  {C_{dyn} \lambda_0 ^2(t-T_-).}
\end{split}
\]
The result yields by combining the latter computations for $J_1,J_2,J_3,J_4$ with \eqref{eq_Rem1ctrl}.

\section{Convergence to equilibrium}\label{sec:equil}

In this last section, we show that, {whatever the initial data}, if $\lambda_0$ is sufficiently small, any solution $(U,P,h,\theta)$ to \eqref{eq:evolution_pb-scaled}-\eqref{eq:ic-scaled} converges to the equilibrium configuration.  The equilibrium state of the system is given by the stationary solution to \eqref{eq:evolution_pb-scaled}.  
{As explained in the introduction, the natural strategy is to compare the solution with its expected limit, i.e. the unique  steady state associated to the stationary position $h=\theta=0$. The difficulty is that these equations are set on different domains.  To overcome this difficulty, we adapt the method of weak-strong uniqueness analysis. Namely, we transfer the stationary (smooth) solution into the geometry associated with the time-dependent (weak) solution. We show a Gr\"onwall-type  inequality on suitable norms of the difference between both quantities that control in particular the solid motion, implying that it returns to the rest state for large time.

\medskip

More precisely,  below,  we denote the (unique) solution to the following steady boundary-value problem by $(U_{eq},P_{eq},h_{eq},\theta_{eq})$:
\begin{equation} \label{eq:st_pb}
\left\{
\begin{aligned}
& -\,\Delta U_{eq}+(U_{eq}\cdot \nabla)U_{eq}+\nabla P_{eq}=0 \, ,\quad \nabla \cdot\, U_{eq}=0 && \text{in} \ \ \Omega_{eq}=A\setminus B_{eq} \, , \\[3pt]
& U_{eq}=0, && \text{on} \ \ \partial B_{eq} \cup \Gamma \, , \\[3pt]
& \lim_{|x_1|\to\infty} U_{eq}(x_1,x_2) = v_p(x_2),  && \forall x_{2} \in [-L,L] \, .
\end{aligned}
\right.
\end{equation}
The following result follows directly from \cite{BoGaGa} and symmetry considerations.  {Indeed, since $h_{eq} = \theta_{eq} = 0$ and the domain $B_{eq}$ is symmetric with respect to $x_2 = 0$, uniqueness induces that the vertical force and the torque applied on $B_{eq}$ vanish, similar to $H_h(0,0)$ and $H_{\theta}(0,0)$. }
\begin{theorem} \label{th:stationary}
There exists $\lambda_{0}^{(2)} > 0$ depending on $ L, e$ such that, if $\lambda_{0} \leq \lambda_{0}^{(2)}$,  the steady problem \eqref{eq:st_pb} admits a unique  solution {$(U_{eq},P_{eq}) \in H^{1}_{\text{loc}}(\Omega_{eq}) \times L_{\text{loc}}^{2}(\Omega_{eq})$}. Moreover,  this solution also satisfies:
\begin{equation} \label{eq:st_pbc}
\left\{
\begin{aligned}
& H_{h}(h_{eq},\theta_{eq})=-\widehat{e}_2\cdot\int_{\partial B_{eq}} \Sigma(U_{eq},P_{eq})\widehat{n} \, d\sigma = 0\, , \\[3pt]
& H_{\theta}(h_{eq},\theta_{eq})=-\int_{\partial B_{eq}} (x-h_{eq}\widehat{e}_2)^{\perp} \cdot \Sigma(U_{eq},P_{eq})\widehat{n} \, d\sigma =0 .
\end{aligned}
\right.
\end{equation}
\end{theorem}
With such notations, the main result of this section reads. 
\begin{theorem} \label{leray}
{Let $(h_0,\theta_0,h'_0,\theta'_0)$ and $u_{0}=U_0(x)-s(0,x)$ be a compatible initial condition in the sense of \eqref{eq:ic2} and 
$(u=U-s,h,\theta)$ be a weak solution to \eqref{eq:hom_pb1}-\eqref{eq:ODE_newref1}.}
{There exist $\omega_0^-,\sigma>0$ with the same dependencies as in Theorem \ref{th:uniform_distance}, and
 $C_{dyn}>0$ such that if \eqref{condition_lambda0} holds and  
$$
2C_{dyn}\lambda_0\le \sigma \omega_0^-\,,
$$ then
$$
\lim_{t\to\infty}\|U(t)-U_{\rm eq}\|^2_{L^2(A)}=\lim_{t \to +\infty} (|h(t)|^2+|\theta(t)|^2) =\lim_{t \to +\infty} (|h'(t)|^2+|\theta'(t)|^2)= 0\, .
$$}
\end{theorem}
Theorem \ref{leray} implies Theorem \ref{lerayintro} {with $p_0=2\lambda_0/L^2\le {p}_*^{(2)}$, where the threshold ${p}_*^{(2)}$ is deduced from the restriction imposed on $\lambda_0$}.
In this statement, we implicitly extended $U_{\rm eq}$ by $0$ on $B_{eq}.$ We point out that $U(t),U_{\rm eq} \in L^2_{loc}(A) \cap (v_P + L^2(|x_1| \geq {2}))$ for all times and consequently,  we have indeed $U(t)-U_{\rm eq} \in L^2(A).$
\begin{remark}\label{remark_viscosity3}
An analogous comment to Remark \ref{remark_viscosity2} holds also for Theorem \ref{leray} as the thresholds  are of the same type.
\end{remark}

The remainder of this section is devoted to the proof of Theorem \ref{leray}.  We split our approach into three parts. Firstly, we recall properties of the solutions $(U_{eq},P_{eq}).$ We also introduce a change of variables that enables to transfer $u_{eq} = U_{eq} - s^{-}$ into \textit{any} moving geometry. We can then apply this construction to the geometry associated to
a weak solution $(u,h,\theta)$ to \eqref{eq:hom_pb1}-\eqref{eq:ODE_newref1} -- once ensured that the ellipse $B(t)$ remains close to the stationary position $B_{eq}$ for large times -- by  building a vector-field ${v}_{eq}$ that we are allowed to compare with the velocity field $u$ of the weak solution.   Secondly, we provide a dissipation estimate on the difference $\tilde{u} =  {v}_{eq}-u$ in the spirit of Proposition \ref{thm_dissipation 1}. Such estimate includes a control on $(h,\theta)$ that shows a return to the equilibrium position of the ellipse making $v_{eq}$ converge to $u_{eq}$ asymptotically and thus the exponential decay of the difference $\tilde{u}$ provides the expected result.

\subsection{Moving the stationary solution into the time-depending geometry.}

Assume $\lambda_0 \leq \lambda_0^{(2)}$ and let $(U_{eq},P_{eq})$ be the unique $H^1_{loc}(\Omega_{eq}) \times L^2_{loc}(\Omega_{eq})$ solution to \eqref{eq:st_pb}. Let us introduce $u_{eq} = U_{eq} - s^{-}$,  $p_{eq}= P_{eq} - \pi_{p}.$ With computations similar to those of Section \ref{sec:preliminaries}, we observe that $(u_{eq},p_{eq})$ is a solution to 
\begin{equation} \label{eq:st_pb2}
\begin{aligned}
& 
\left\{
\begin{aligned}
& -\,\Delta u_{eq}+(u_{eq}\cdot \nabla)u_{eq}+ (u_{eq} \cdot \nabla) s^{-} + (s^{-} \cdot \nabla) u_{eq} + \nabla p_{eq}= \hat{g}^{-} \, && \text{in} \ \ \Omega_{eq}  \, , \\
& \nabla \cdot\, u_{eq}=0 && \text{in} \ \ \Omega_{eq}  \, , \\[3pt]
& u_{eq}=0, && \text{on} \ \ \partial B_{eq} \cup \Gamma \, , \\[3pt]
& \lim_{|x_1|\to\infty} u_{eq}(x_1,x_2) = 0,  && \forall x_{2} \in [-L,L] \, , 
\end{aligned}
\right.
\\
& \left\{
\begin{aligned}
& 0 =-\widehat{e}_2\cdot\int_{\partial B_{eq}} \Sigma(u_{eq},p_{eq})\widehat{n} \, d\sigma \, , \\[3pt]
& 0 =-\int_{\partial B_{eq}} x^{\perp} \cdot \Sigma(u_{eq},p_{eq})\widehat{n} \, d\sigma \,,
\end{aligned}
\right.
\end{aligned}
\end{equation}
where $\hat{g}^-$ is similarly constructed from $s^{-}$ as $\hat{g}$ is derived from $s$ in Lemma \ref{lemma:extension}. 
In particular, we recall that ${\rm supp}(\hat{g}^{-}) \subset A_0.$ 
By classical elliptic regularity \cite{galdi}, we remark that $u_{eq}$ and $p_{eq}$ are smooth on $\overline{\Omega}_{eq}$ with all norms controlled by $C_{geo}\,\lambda_0$.

\medskip

Our next aim is to define a change of variables that brings the steady solution $u_{eq}$ in the time-dependent geometry. In the construction, we first define a mapping that brings the equilibrium position $B_{eq}$ to a given position $B(h,\theta)$, assuming 
\begin{equation}\label{eq:h-theta-petits} 
\max(|h|,|\theta|)\le \eta_0, 
\end{equation}
with $\eta_0$ small. We will then apply this construction with $(h,\theta)=(h(t),\theta(t))$ associated to a given weak solution. 

For the construction, we need to set up some notations. First, we define two neighbourhoods $\mathcal O,\mathcal O'\subset A_0$ such that 
\begin{itemize}
\item  $\mathcal O\subset\mathcal O'$;
\item $\overline{B_{eq}}\subset\mathcal O$;
\item $\mathcal O'\cap{\rm supp}(s^-)=\emptyset$. 
\end{itemize}
This is possible since $B_{eq}\subset A_0$ and  $B_{eq}\cap{\rm supp}(s^-)=\emptyset$.

\begin{lemma}\label{lem:change-var}
There exist $\eta_0>0$ and a $\mathcal{C}^{1}$-mapping  $X : [-\eta_0,\eta_0]^2 \to \mathcal{C}^{3}(A)$ satisfying
\begin{itemize}
\item[(i)] $X[0,0]$ is the identity,
\end{itemize}
and, for all $(h,\theta)$ satisfying \eqref{eq:h-theta-petits},  
\begin{itemize}
\item[(ii)] $X[h,\theta]$ is a $\mathcal{C}^3$-diffeomorphism of $\Omega$,
\item[(iii)] $X[h,\theta](x) = h\widehat{e}_2 + Q(\theta) x$ in $\mathcal O$, $X[h,\theta]$ is the identity outside $\mathcal O'$,
\item[(iv)] $X[h,\theta](\mathcal O)\subset \mathcal O'$ and $X[h,\theta]^{-1}(\mathcal O)\supset B_{eq}$.
\end{itemize}
\end{lemma}
\begin{proof}
The proof is standard and can be done by defining $X[h,\theta]$ by 
\begin{equation}
X[h,\theta](y) = \xi(y)h\widehat{e}_2 + Q(\xi(y)\theta) y,
\end{equation}
where $\xi$ is a cut-off function such that $\xi=1$ inside $\mathcal O$ and $\xi=0$ outside $\mathcal O'$. Fixing $\eta_0$ small enough allow to show that $X[h,\theta]$ is basically a perturbation of the identity so that assertions (ii)-(iv) hold. 
\end{proof}
In the sequel, we set $Y[h,\theta] = X[h,\theta]^{-1}$. In particular, 
$$Y[h,\theta](x) = Q(\theta)^T(x-h\widehat{e}_2),$$
for $x\in X[h,\theta](\mathcal O)$ and 
\[
X[h,\theta](B_{eq}) = B(h,\theta).
\]

Given $(h,\theta) \in [-\eta_0,\eta_0]^2,$ we now transform $u_{eq}$ into a divergence-free vector-field defined on the fluid domain
 $A\setminus B(h,\theta)$ by setting 
 $${u}^{[h,\theta]}_{eq}(x) \doteq \nabla^\perp(\psi_{eq}\circ Y[h,\theta])(x),$$ 
where $\psi_{eq}$ is the stream function associated to ${u}_{eq}$ (that we denote by $(u_1,u_2)$ in the next identity), i.e. 
\begin{equation} \label{eq_formulauhs}
{u}^{[h,\theta]}_{eq}(x) = 
\begin{pmatrix}
 \partial_2 Y_2[h,\theta](x)u_1 (Y[h,\theta](x))  - \partial_2 Y_1[h,\theta](x) u_2(Y[h,\theta](x)) \\
 - \partial_1 Y_2[h,\theta](x)u_1(Y[h,\theta](x)) + \partial_1 Y_2[h,\theta](x) u_2(Y[h,\theta](x))
\end{pmatrix}.
\end{equation}
Correspondingly, we define the transformed pressure 
$$p^{[h,\theta]}_{eq}(x) = p_{eq}(Y[h,\theta](x)).$$

With such notations, we have the following lemma 
\begin{lemma}\label{lemma:u_eq}
Let $\eta_0 > 0$ be as in Lemma \ref{lem:change-var}. There exists $C_{geo}>0$ (depending only on $e$, $L$ and $\eta_0$) such that
given $(h,\theta) \in (-\eta_0,\eta_0)^2$ we have: 
\[
\begin{aligned}
& 
\left\{
\begin{aligned}
& -\,\Delta u^{[h,\theta]}_{eq}+ \nabla p_{eq}^{[h,\theta]} +(u_{eq}^{[h,\theta]}\cdot \nabla)u^{[h,\theta]}_{eq}+ (u_{eq}^{[h,\theta]} \cdot \nabla) s^{-} + (s^{-} \cdot \nabla) u^{[h,\theta]}_{eq}  \\&= \hat{g}^{[h,\theta]} \,\quad\text{in}\,\quad A \setminus \overline{B}(h,\theta)    , \\
& \nabla \cdot\, u^{[h,\theta]}_{eq}=0 \qquad \text{in} \ \ A\setminus \overline{B}(h,\theta)   , \\[3pt]
& u^{[h,\theta]}_{eq}=0, \qquad \text{on} \ \ \partial B(h,\theta) \cup \Gamma  , \\[3pt]
& \lim_{|x_1|\to\infty} u^{[h,\theta]}_{eq}(x_1,x_2) = 0,  \qquad \forall x_{2} \in [-L,L] \, , 
\end{aligned}
\right.
\\
& \left\{
\begin{aligned}
& F_{eq}^{[h,\theta]}  = \widehat{e}_2\cdot\int_{\partial B(h,\theta)} \Sigma(u^{[h,\theta]}_{eq},p_{eq}^{[h,\theta]})\widehat{n} \, d\sigma \, , \\[3pt]
& 0 = \int_{\partial B(h,\theta)} (x- h \widehat{e}_2 )^{\perp} \cdot \Sigma(u^{[h,\theta]}_{eq},p_{eq}^{[h,\theta]})\widehat{n} \, d\sigma  .
\end{aligned}
\right.
\end{aligned}
\]
where:
\begin{itemize}
\item[(i)] $|F_{eq}^{[h,\theta]}| \leq C_{geo}\, \lambda_0|\theta|$
\item[(ii)] ${\rm supp}(\hat g^{[h,\theta]}) \subset A_0$ and 
\[
\|\hat g^{[h,\theta]} - \hat g^{-}\|_{L^{\infty}(A_0 \setminus B(h,\theta))} \leq C_{geo}\lambda_0 (|h| + |\theta|).
\] 
\end{itemize}
\end{lemma}

Before giving a proof of our result, we recall that $\hat{g}^{-}$ as is defined in item (v) of Lemma \ref{lemma:extension} is defined all over $A.$ Assertion (ii) of the previous statement is then not ambiguous.

\begin{proof}
We will work differently inside and outside $\mathcal O$. Let 
$$\varepsilon_0={\rm dist}(B_{eq},A_0\setminus\mathcal O)>0$$
and set 
$$ \mathcal O_0 = \left\{x\in \mathcal O : {\rm dist}(x,B_{eq})< \frac{\varepsilon_0}{4}\right\}.$$
Observe we can assume 
$$|Y[h,\theta](x)-y|< \frac{\varepsilon_0}{4},$$
for all $x\in A$ by taking $\eta_0$ small enough.

Recall that $\hat{g}^{[0,0]}= \hat g^{-}$. Hence, for $x\in A\setminus A_0$, we remark that $Y[h,\theta](x)= x = Y[0,0](x)$ which 
implies that $\hat g^{[h,\theta]}(x) - \hat g^{-}(x) = \hat g^{[0,0]}(x) - \hat g^{-}(x) =0$.

If $x\in A_0\setminus\mathcal O_0$, then $Y[h,\theta](x)\in B(x,{\varepsilon_0}/{4})\subset A\setminus B_{eq}$.
Since $u_{eq}$ is smooth in $A_0\setminus B_{eq}$ and $(h,\theta) \to Y[h,\theta]$ is $\mathcal{C}^1 : [-\eta_0,\eta_0]^2 \to \mathcal{C}^3(\overline A_0)$, we can estimate, with standard chain rule arguments, all quantities involved in $\hat{g}^{[h,\theta]}- \hat g^{-}$ to get 
\[
\|\hat g^{[h,\theta]} - \hat g^{-}\|_{L^{\infty}(A_0 \setminus \mathcal O_0)} \leq C_{geo}\, \lambda_0(|h| + |\theta|).
\] 
It remains to treat the case where $x\in \mathcal O_0\setminus \overline{B_{eq}}$ and to compute the forces. 
For $x\in \mathcal O_0\setminus \overline{B_{eq}}$, $y=Y[h,\theta](x)\in \mathcal O$ and therefore both $X[h,\theta](y)$ and $Y[h,\theta](x)$ are isometries. Since $s^-(x)=0$ and $\hat g^-(x)= \lambda_0\widehat{e}_1$, we infer from the Galilean invariance of the Navier-Stokes equations that $\hat g^{[h,\theta]}(x) = \lambda_0 Q(-\theta)\widehat{e}_1$ closing the proof of assertion (ii).

Arguing similarly, still by Galilean invariance, we observe that 
$$\int_{\partial B(h,\theta)} (x- h \widehat{e}_2 )^{\perp} \cdot \Sigma(u^{[h,\theta]}_{eq},p_{eq}^{[h,\theta]})\widehat{n} \, d\sigma = \int_{\partial B_{eq}} x^{\perp} \cdot \Sigma(u_{eq},p_{eq})\widehat{n} \, d\sigma = 0,$$
and 
$$F_{eq}^{[h,\theta]}
= Q(-\theta)\widehat{e}_2 \cdot\int_{\partial B_{eq}} \Sigma(u_{eq},p_{eq})\widehat{n} \, d\sigma = 
\left(Q(-\theta)\widehat{e}_2 - \widehat{e}_2\right) \cdot\int_{\partial B_{eq}} \Sigma(u_{eq},p_{eq})\widehat{n} \, d\sigma.$$
We then deduce assertion (i) using the regularity of $(u_{eq},p_{eq})$.
\end{proof}

\subsection{Proof of Theorem \ref{leray}:  stability estimate}
Let $(u,h,\theta)$ be a weak solution to \eqref{eq:hom_pb1}-\eqref{eq:ODE_newref1}. Under the assumptions of Theorem \ref{th:uniform_distance}, we infer from \eqref{uniform_estimate_en-section5} and \eqref{hp_H} that 
$$|h(t)|^2+|\theta(t)|^2\le \frac{2}{\bar\varrho} H(h(t),\theta(t))\le E_{tot}(t)\le  3E_0\,e^{-\beta_0 t}+  \frac{C_{geo}\lambda_0^2}{\beta_0},$$
for all $t\ge 0$, yielding the existence of  $t_0>0$ and $\lambda_0(\eta_0)>0$ (depending decreasingly on $E_0$ and on $L,e$), such that 
\begin{equation}\label{Remark}
\max(|h(t)|,|\theta(t)|)\le \eta_0\,
\end{equation}
for $t\ge t_0$ whenever $\lambda\in[0,\lambda_0(\eta_0)]$. This does not yet prove the convergence to equilibria since the solid could still oscillate around the rest position. Since the problem is autonomous, we will assume in the sequel that $t_0=0$. We can also suppose  that $\lambda_0(\eta_0)$ is chosen in such a way that \eqref{uniform_estimate_en-section5} implies 
\begin{equation}\label{bounded_L2}
\sup_{t\in(0,\infty)}\|u(t)\|_{L^2(\Omega(t))}\le 4 E_0\,,
\end{equation}
whenever $\lambda\in[0,\lambda_0(\eta_0)]$.

To prepare the proof of {Theorem \ref{leray}} and to shorten notations, we denote $u^{[h(t),\theta(t)]}_{eq}(x)$ by $v_{eq}(t,x)$ from now on, and we define 
$$\tilde u(t,x) = v_{eq}(t,x)-u(t,x),$$
for $t\ge 0$ and $x\in \Omega(t)$. 
Observe that $v_{eq}$ is smooth in the space variables on $(0,\infty).$ We denote below: 
\begin{equation} \label{eq_ueta2}
\left\{
\begin{aligned}
 &- \Delta v_{eq}(t) + \nabla q_{eq}(t)  + (v_{eq}(t) \cdot \nabla) v_{eq}(t) + (s^{-} \cdot \nabla) v_{eq}(t)  + (v_{eq}(t) \cdot \nabla) s^{-}\\[2pt]& =  {g}_{eq}(t)\,,\quad
 \text{ in $\Omega(t)$}\,,\\[2pt]&\nabla\cdot v_{eq}(t)=0\,, \quad \text{in $\Omega(t)$}\\[2pt]&
 \widehat{e}_2\cdot\int_{\partial B(t)} \Sigma(v_{eq}(t),q_{eq}(t))\widehat{n} \, d\sigma = F_{eq} (t)\, 
 \, , \\[3pt]&
  \int_{\partial B(t)} (x- h(t) \widehat{e}_2 )^{\perp} \cdot \Sigma(v_{eq}(t),q_{eq}(t))\widehat{n} \, d\sigma  = 0\,,
\end{aligned}
\right.
\end{equation}
where $q_{eq}(t) = p^{[h(t),\theta(t)]}_{eq}$ and ${g}_{eq}(t) = \hat{g}^{[h(t),\theta(t)]},$ $F_{eq}(t) = F_{eq}^{[h(t),\theta(t)]}$ for all $t \ge 0.$ With obvious notations, we also have that
\[
\partial_t v_{eq} = {h}' \partial_h u^{[h,\theta]} + {\theta}' \partial_{\theta} u^{[h,\theta]}\,.
\]
The function $v_{eq}$ gains the $W^{1,\infty}$ regularity of $h$ and $\theta$ in time. In particular, observe that $v_{eq}\in W^{1,\infty}((0,\infty) \times A)$ and $\partial_t v_{eq}=0$ outside $\mathcal O'$ and, thanks to the regularity of $Y[h,\theta]$ and $u_{eq}$ (see \cite[Section VI.2]{galdi} for the decay properties of $u_{eq}$ outside $A_0$), there exists a constant $C_{geo}$
for which:
\begin{equation} \label{eq_veq}
\|\partial_t v_{eq}(t) \|_{L^{\infty}(A_0)} \leq C_{geo} \lambda_0 (|{h}'(t)| + |{\theta}'(t)|) \,, \qquad
\|\nabla v_{eq}(t)\|_{L^{\infty}(A)} \leq C_{geo} \lambda_0\,, 
\end{equation}
for all $t \ge 0.$
\medbreak

We can then use a weak-strong argument to compare $u(t)$ and $v_{eq}$ for large times.  To this aim, we introduce 
\[
\tilde{E}(t)\doteq\frac{1}2\|\tilde{u}(t)\|^2_{L^2(\Omega(t))}+ \frac{m}2|h'(t)|^2+\frac{J}2|\theta'(t)|^2+H(\theta(t),h(t))\qquad \forall \, t \geq 0.
\]
As previously,  we will also use a perturbed version of $\tilde{E}$, namely 
\[
  \tilde{E}_\omega(t)\doteq \tilde{E}(t)+\omega\left(m\, h(t)h'(t)+J\,\theta(t)\theta'(t)+\int_{\Omega(t)}\tilde{u}(t)\cdot w(t)\,dx\right)\,,
\]
where $\omega\in (0,1)$ will be chosen small enough in order to derive a dissipation estimate in the spirit of Proposition \ref{thm_dissipation 1}.
Similarly to Lemma \eqref{prop_boundsE}, with the further remark that in view of \eqref{Remark} the distance between $B(t)$ and $\partial A$ is now uniformly bounded from below by $1$,  we have 
\begin{equation} \label{eq_boundsEeta}
 \dfrac{\tilde{E}(t)}{2} \leq \tilde{E}_\omega(t) \leq \dfrac{3}{2} \tilde{E}(t)\,,
\end{equation}
under the condition $\omega \leq \omega_0$ with $\omega_0$ depending only on dynamical parameters. 
Theorem \ref{leray} will be a direct consequence of the next proposition where we show the exponential decay of $\tilde{E}_\omega(t)$. 

\begin{proposition} \label{thm_dissipation2}
There exist $\omega_0^-,\sigma>0$ with the same dependencies as in Theorem \ref{th:uniform_distance}, and
 $C_{dyn}>0$ such that, assuming \eqref{condition_lambda0} and further that 
$$
2C_{dyn}\lambda_0\le \sigma \omega_0^-\,,
$$
then
\begin{equation} \label{eq_dissipation2}
\limsup_{t\to \infty} \tilde{E}_{\omega}(t) =0.
\end{equation}
\end{proposition}

\begin{proof} 
Recalling the definition \eqref{eq_modifenergy} of $E_{\omega}$, we have:
\begin{equation} \label{eq_splitEomega}
\tilde{E}_\omega(t) = {E_\omega}(t) -  
 \int_{\Omega(t)} v_{eq}(t) \cdot u(t) + \frac12
\|v_{eq}(t)\|^2_{L^2(\Omega(t))}  - \omega \int_{\Omega(t)} v_{eq}(t) \cdot w(t)\, dx.
\end{equation}
We fix then $t_2 \geq t_1\ge 0$ and 
we use $\omega(w, h,\theta)$ as test in 
\eqref{eq:weak_gen} and \eqref{eq:energy_estimate_weak} to get, similarly to the previous section, that
\begin{equation}\label{E_omega2bis}
{
\begin{aligned}
&E_\omega(t_2) -  E_\omega(t_1) + \int_{t_1}^{t_2} \|\nabla u(\tau)\|^2_{L^2(A)}\, d\tau\\
& + \int_{t_1}^{t_2}\left(\omega \left[   H_h(h(\tau),\theta(\tau))h(\tau) +  H_\theta(h(\tau),\theta(\tau))\theta(\tau) - m{|h'(\tau)|}^2 - J{|\theta'(\tau)|}^2\right]\right)\, d\tau   \\[6pt]
= &
 \int_{t_1}^{t_2} \int_{\Omega(\tau)}
 [ \widehat{g}^-  - u  \cdot \nabla s^{-}] \cdot u \, dxd\tau +\omega\int_{t_1}^{t_2}\int_{\Omega(\tau)}u\cdot \partial_t w\,dxd\tau
  -\omega \int_{t_1}^{t_2} \int_{\Omega(\tau)}\nabla u:\nabla w\,dxd\tau \\
& + \omega\int_{t_1}^{t_2} \int_{\Omega(\tau)} \left(\widehat{g}^{-}-(u\cdot \nabla)s^- -(u\cdot \nabla)u - (s^-\cdot \nabla )u\right) \cdot w \, dxd\tau. \\ 
 \end{aligned}
  }
\end{equation}
We recall that since $|h(t)| \leq \eta_0 \leq 1/2$ on $(0,\infty)$, we have $\hat{f}[h(t)]=0$, $\hat{g}(t) = \hat{g}^{-},$ and $s(t)=s^-$ for $t \geq 0$.

We compute then the increment of the second term in the right-hand side of \eqref{eq_splitEomega} by remarking that $(v_{eq},0,0)$ has sufficient space/time-regularity to be admissible test-function in \eqref{eq:weak_gen}. We obtain
\begin{equation} \label{eq_weakueta}
\begin{aligned}
& \int_{\Omega(t_2)} u(t_2) \cdot v_{eq}(t_2)dx   -  \int_{\Omega(t_1)} u(t_1) \cdot v_{eq}(t_1)dx + \int_{t_1}^{t_2} \int_{\Omega(\tau)} \nabla u : \nabla v_{eq} \, dx \, d\tau  \\
= &  \int_{t_1}^{t_2} \int_{\Omega(\tau)} u \cdot \partial_t v_{eq}  \, dx \, d\tau + \int_{t_1}^{t_2}  \int_{\Omega(\tau)}{  (u\cdot \nabla)\,v_{eq} \cdot u} \, dx  \, d\tau   \\& + \int_{t_1}^{t_2} \int_{\Omega(\tau)} \left[ \widehat{g}^- - (u\cdot \nabla)\,s^- - s^-\cdot \nabla)\,u \right] \cdot v_{eq}  \, dx \, d\tau.
 \end{aligned}
\end{equation}
Multiplying now \eqref{eq_ueta2} by $u(t)$ between $t_1$ and $t_2$ and integrating in time and space, we get
\begin{equation} \label{eq_wfueta}
\begin{aligned}
& \int_{t_1}^{t_2} \int_{\Omega(\tau)} \nabla v_{eq} : \nabla  u \,dxd\tau = \int_{t_1}^{t_2}F_{eq}(\tau) {h}'(\tau)\, d \tau+ \int_{t_1}^{t_2} \int_{\Omega(\tau)}  {g}_{eq} \cdot u \,dxd\tau\, \\
& - \int_{t_1}^{t_2} \int_{\Omega(\tau)} \left[ (v_{eq} \cdot \nabla) v_{eq}  
+  ( s^{-} \cdot \nabla) v_{eq}  + (v_{eq} \cdot \nabla) s^{-}  \right] \cdot u\,  dxd\tau.
\end{aligned}
\end{equation}
We also observe that
\begin{equation} \label{eq_Dtw}
\int_{\Omega(t_2)} v_{eq} \cdot w(t_2) - \int_{\Omega(t_1)} v_{eq} \cdot w(t_1)  = 
\int_{t_1}^{t_2} \int_{\Omega(\tau)} \left( v_{eq} \cdot \partial_t w 
+ \partial_t v_{eq} \cdot w  \right)\, dxd\tau
\end{equation}
and 
\begin{equation} \label{eq_dtveq}
\frac12\|v_{eq}\|_{L^2(\Omega(t_2))}-\frac12\|v_{eq}\|_{L^2(\Omega(t_1))}
=\int_{t_1}^{t_2}\int_{\Omega(\tau)}\partial_t v_{eq}\cdot v_{eq}\, dx\, d\tau.
\end{equation}
We point out that we used the fact that $v_{eq} =0$ on $\partial B(t)$ for all $t\ge 0$ in this computation to reduce the material time-derivative to a standard time-derivative by integrating by parts.
Finally, multiplying the equation \eqref{eq_ueta2} with  $v_{eq}$ on $\Omega(\tau)$ for $\tau\in (t_1,t_2)$, we have
\begin{equation} \label{eq_energyueta2}
 \int_{t_1}^{t_2} \|\nabla v_{eq}\|_{L^2(\Omega(\tau))}^2\, d\tau  =   \int_{t_1}^{t_2}\int_{\Omega(\tau)} \left(g_{eq}- (v_{eq} \cdot \nabla)s^-\right) \cdot v_{eq}\, dxd\tau
\end{equation}

Substracting \eqref{eq_weakueta}-\eqref{eq_wfueta}-\eqref{eq_Dtw} from the sum of \eqref{E_omega2bis}
with \eqref{eq_energyueta2} and substituting with \eqref{eq_dtveq}, we infer that
\begin{equation}\label{utile2}
\begin{aligned}
& \tilde{E}_\omega(t_2) - \tilde{E}_\omega(t_1) +   \int_{t_1}^{t_2}\|\nabla \tilde{u}(\tau)\|^2_{L^2(A)} \, d\tau\\
& +  \int_{t_1}^{t_2}\left( \omega\left[   H_h(h(\tau),\theta(\tau))h(\tau) +  H_\theta(h(\tau),\theta(\tau))\theta(\tau) - m{|h'(\tau)|}^2 - J{|\theta'(\tau)|}^2\right]\right)d\tau  \\[6pt]
& 
=  - \int_{t_1}^{t_2} \int_{\Omega(\tau)}
 ((\tilde{u}  \cdot \nabla) s^{-} + \partial_tv_{eq} + (\tilde u\cdot \nabla)v_{eq})\cdot \tilde{u}\,  dx\, d\tau    - \int_{t_1}^{t_2} F_{eq}(\tau)h'(\tau)\, d\tau
 \\
 &
  +\omega\int_{t_1}^{t_2}\int_{\Omega}\tilde{u}\cdot \partial_t w\,dx\, d\tau -\omega \int_{t_1}^{t_2} \int_{\Omega(\tau)}\nabla u:\nabla w\,dxd\tau\\
  &+ \omega\int_{t_1}^{t_2} \int_{\Omega(\tau)} \left(\widehat{g}^{-} - \partial_t v_{eq} - (u\cdot \nabla)s^- - (u\cdot \nabla)u -(s^-\cdot \nabla )u \right)\cdot w \, dxd\tau.
  \end{aligned}
\end{equation}
We then multiply \eqref{eq_ueta2} by $w$ to deduce
\begin{equation}\label{utile}
\begin{aligned}
&  \int_{t_1}^{t_2} \int_{\Omega(\tau)} \nabla v_{eq} : \nabla  w \,dxd\tau 
\\
&= \int_{t_1}^{t_2} F_{eq}(\tau)h(\tau)  + \int_{t_1}^{t_2} \int_{\Omega(t)} \left(  {g}_{eq}- v_{eq} \cdot \nabla v_{eq}
-   s^{-} \cdot \nabla v_{eq}   - v_{eq} \cdot \nabla s^{-} \right) \cdot w\, dxd\tau  .
\end{aligned}
\end{equation}
Multiplying \eqref{utile} by $\omega$ and subtracting this identity from \eqref{utile2}, we infer that 
\begin{equation}\label{utile3}
\begin{aligned}
& \tilde{E}_\omega(t_2) - \tilde{E}_\omega(t_1) +   \int_{t_1}^{t_2}\|\nabla \tilde{u}(\tau)\|^2_{L^2(A)} \, d\tau\\
& +  \int_{t_1}^{t_2}\left( \omega\left[   H_h(h(\tau),\theta(\tau))h(\tau) +  H_\theta(h(\tau),\theta(\tau))\theta(\tau) - m{|h'(\tau)|}^2 - J{|\theta'(\tau)|}^2\right]\right)d\tau   \\[6pt]
& 
=  - \int_{t_1}^{t_2} \int_{\Omega(\tau)}
\left( (\tilde{u}  \cdot \nabla) s^{-} + \partial_tv_{eq} + (\tilde u\cdot \nabla)v_{eq}\right)\cdot \tilde{u}\,  dx\, d\tau   - \int_{t_1}^{t_2} F_{eq}(\tau)(h'(\tau)+\omega h(\tau))\, d\tau\\
 & +\omega\int_{t_1}^{t_2}\int_{\Omega(\tau)}\tilde{u}\cdot \partial_t w\,dx\, d\tau-\omega \int_{t_1}^{t_2} \int_{\Omega(\tau)}\nabla \tilde u:\nabla w\,dxd\tau
\\ 
& + \omega\int_{t_1}^{t_2} \int_{\Omega(\tau)} \left(\widehat{g}^{-}-g_{eq}-\partial_t v_{eq}-(\tilde u\cdot \nabla)s^- - (u\cdot \nabla)\tilde{u}-(\tilde{u}\cdot \nabla){v}_{eq} - (s^-\cdot \nabla )\tilde u\right) \cdot w \, dxd\tau 
  \end{aligned}
\end{equation}
By standard H\"older inequality and Poincaré inequality, we will now estimate the right-hand side $RHS$ of \eqref{utile3}. We start with the terms in the first line 
which can be estimated by 
\[
\begin{split}
& C_{geo}\int_{t_1}^{t_2}\left(\|s^{-}\|_{W^{1,\infty}(\Omega(\tau))}+\|\partial_tv_{eq}\|_{L^2(\Omega(\tau))}+\|\nabla v_{eq}\|_{L^\infty(\Omega(\tau))}\right)\|\nabla \tilde{u}(\tau)\|^2_{L^2({A})}\,d\tau \\  & +  \int_{t_1}^{t_2}|F_{eq}(\tau)|(|h'(\tau)|+\omega |h(\tau)|)\,d\tau. 
\end{split}
\]
while the sum of all the remaining terms multiplied by $\omega$ is smaller than 

$$
\begin{aligned}
& (C_{geo}+1)\int_{t_1}^{t_2}\left(\|\partial_t w\|_{L^2(\Omega(\tau))}
+\|\nabla w\|_{L^2(\Omega(\tau))}
\right)\|\nabla \tilde{u}(\tau)\|_{L^2(A)}\, d\tau
\\[2pt]
& +\int_{t_1}^{t_2}\left(\|\hat{g}^{-}-g_{eq}\|_{L^2(\Omega(\tau))}+\|\partial_t v_{eq}\|_{L^2(\Omega(\tau))}\right)\|w\|_{L^2(\Omega(\tau))}\,d\tau
\\[2pt]
& \!+\!(C_{geo}+1)\!\! \int_{t_1}^{t_2}\left(\|s^{-}\|_{W^{1,\infty}(\Omega(\tau))}+\|u\|_{L^2(\Omega(\tau))} + \|\nabla v_{eq}\|_{L^\infty(\Omega(\tau))}\right)\|w\|_{L^{\infty}(A_0)}\|\nabla \tilde{u}\|_{L^2(A)}\,d\tau.
\end{aligned}
$$
Next, we exploit Lemma \ref{lemma:u_eq} and estimates \eqref{stime_w}, where $\mathfrak{d}_{min}^0\ge 1$, and  \eqref{eq_dissipationsolide}-\eqref{bounded_L2}-\eqref{eq_veq} to get 
$$
\begin{aligned}
|RHS|\le& \left[(C_{geo}+1)(\lambda_0+\omega)+\frac{1}4\right]\int_{t_1}^{t_2}\|\nabla \tilde{u}(\tau)\|^2_{L^2(A)}\,d\tau\\[2pt]&+(C_{geo}+1)\,[(1+E_0)\omega^2+\lambda_0\omega+\lambda_0^2]\int_{t_1}^{t_2}(|h(\tau)|^2+|\theta(\tau)|^2)\,d\tau\\[2pt]\le& \left[(C_{geo}+1)(\lambda_0+\omega)+\frac{1}4\right]\int_{t_1}^{t_2}\|\nabla \tilde{u}(\tau)\|^2_{L^2(A)}\,d\tau\\[2pt]&+C_{dyn}\,[(1+E_0)\omega^2+\lambda_0\omega+\lambda_0^2]\int_{t_1}^{t_2}\tilde{E}_\omega(\tau)\,d\tau\,,
\end{aligned}
$$
where in the last inequality we used again \eqref{hp_H} and \eqref{eq_boundsEeta}. Similarly, from \eqref{hp_F} and \eqref{eq_dissipationsolide}, provided that $\omega\le C_{dyn}^{-1}$,  we infer that 
\begin{multline*}
\int_{t_1}^{t_2}\left( \omega\left[   H_h(h(\tau),\theta(\tau))h(\tau) +  H_\theta(h(\tau),\theta(\tau))\theta(\tau) - m{|h'(\tau)|}^2 - J{|\theta'(\tau)|}^2\right] \right)d\tau \\[2pt]
+\int_{t_1}^{t_2} \|\nabla \tilde{u}(\tau)\|^2_{L^2(A)}  d\tau \ge \int_{t_1}^{t_2}\left( \omega \tilde{E}_\omega(\tau)+  \frac{1}2\|\nabla \tilde{u}\|^2_{L^2(A)}  \right)d\tau\,.
\end{multline*}
With our assumptions on $\omega_0^{-}$ and $\lambda_0$, we deduce that 
$$
\tilde{E}_{\omega}(t_2)-\tilde{E}_{\omega}(t_1)+\frac{\omega_0^-}2\int_{t_1}^{t_2}\tilde{E}_{\omega}(\tau)\,d\tau \le 0\,.
$$
We can then conclude by a Grownwall-type inequality taking into account the continuity of $\tilde{E}_{\omega}$
that
\[
\lim_{t \to \infty} |h(t)| + |\theta(t)| = \lim_{t\to \infty} |h'(t)| + |\theta'(t)| = \lim_{t \to \infty} \|u(t) - v_{eq}(t)\|_{L^2(\Omega(t))} = 0,
\]
and thus, by extending $u(t)$ and $v_{eq}(t)$ with their solid counterparts:
\[
\lim_{t \to \infty} \|u(t) - v_{eq}(t)\|_{L^2(A)} = 0.
\]
However, since $v_{eq}(t)$ is obtained by transforming $u_{eq}$ into the geometry associated with $(h(t),\theta(t))$ that converges to $0$ the regularity of the transformation \eqref{eq_formulauhs} in $L^2(A)$ entails:
\[
\lim_{t \to \infty} \|v_{eq}(t) - u_{eq}\|_{L^2(A)} = 0.
\]
Eventually, we conclude that
\[
\lim_{t \to \infty} \|u(t) - u_{eq}\|_{L^2(A)} = 0.
\]

\end{proof}


\appendix

\section{Further properties of the ellipse} \label{app_difftheta}
This appendix contains explicit computations supporting the general description of Section \ref{sec_gapdescription} in the case of an ellipse. 
We also provide a proof of Lemma \ref{lem_geom} that applies in this specific case of an ellipse.\medskip

\subsection{Explicit description of $\gamma[\theta]$}
We start by providing an explicit expression for the parametrizing function $\gamma[\theta]$ and by proving \eqref{constants}. As $h$ has no influence here, we assume $h=0$ without loss of generality and, due to the symmetries of the ellipse, we can consider the case $\theta\in (-\pi/2,0]$ only. Then, for every $(x_1,x_2)\in \partial B$ there holds 
\begin{equation}\label{eq:ellipse1}
	(x_1,x_2)\in \partial B \iff Q(\theta(t))^{\top}(x_1,x_2)\in \partial B_{eq}. 
\end{equation}
Given the shape of the obstacle $B_{eq}$, condition \eqref{eq:ellipse1} is verified if and only if 
\begin{equation} \label{elipse0}
	g(x_1,x_2)\doteq{(x_1\cos{\theta}+x_2\sin{\theta})^2}+
	\frac{(-x_1\sin{\theta}+x_2\cos{\theta})^2}{e^2}-1=0.
\end{equation}
Noticing that the coordinates $(\sf{x}_1,\sf{x}_2)$ of the point of contact $X_v$ must satisfy the relation 
$$
\frac{\partial g}{\partial x_1}({\sf{x}_1},{\sf{x}_2})=0,
$$
we differentiate \eqref{elipse0} implicitly with respect to $x_{1}$, thereby obtaining the identity
\begin{equation}\label{identity_c}
{	\sf{x}_1}=\frac{(1-e^2)\cos{\theta}\sin{\theta}}{{e^2\cos^2{\theta}}+\sin^2{\theta}}\,\sf{x}_2.  
\end{equation}
Inserting \eqref{identity_c} in \eqref{elipse0}, we obtain the coordinates $X_v=(\sf{x}_1,\sf{x}_2)$ as functions of the angle $\theta$ 
\begin{equation} \label{eq:x1x2}
{	\sf{x}_1} =-\frac{(1-e^2){\cos{\theta}\sin{\theta}}}{{\sqrt{ \sin^2\theta + e^{2}\cos^2{\theta}}}}>0  \quad \text{and} \quad {\sf{x}_2} = - {\sqrt{ \sin^2\theta + e^{2}\cos^2{\theta}}} < 0 .
\end{equation}
We want to express $x_{2}$ as a function of $x_{1}$ in some interval of the type $[-\lambda(\theta), \lambda(\theta) ]$, with $\lambda(\theta) > 0$, by parametrizing the border of the obstacle close to the point of contact. For this reason, we rewrite the identity \eqref{elipse0} as a quadratic equation in the unknown variable $x_2$:
\begin{equation} \label{elipse3}
	\widetilde{A}(\theta)x^2_1+\widetilde{B}(\theta)x_1 x_2+\widetilde{C}(\theta)x_{2}^{2}-1=0,
\end{equation}
where
$$
\begin{aligned}
	& \widetilde{A}(\theta)=\left({\cos{\theta}}\right)^2+\left(\frac{\sin{\theta}}{e}\right)^2>0, \quad \widetilde{B}(\theta)=\sin2\theta \left({1}-\frac{1}{e^2} \right)>0, \\& \widetilde{C}(\theta)=\left({\sin{\theta}}\right)^2+\left(\frac{\cos{\theta}}{e}\right)^2>0.
\end{aligned}
$$
Identity \eqref{elipse3} shows that we can express $x_{2}$ as a function of $x_{1}$ provided that
\begin{equation} \label{disc1}
	|x_1|\le \lambda(\theta) \doteq \sqrt{e^2 \sin^2{\theta} +  \cos^2{\theta}},
\end{equation}
thus yielding
\begin{equation}\label{eq:Gamma}
	\Gamma_\theta(x_1) \doteq x_{2} (x_{1}) = - \dfrac{\widetilde{B}(\theta)}{2\widetilde{C}(\theta)}x_1 -\dfrac{1}{\widetilde{C}(\theta)} \sqrt{\widetilde{C}(\theta)-\frac{x^2_1}{e^2}}\qquad \forall x_{1} \in [-\lambda(\theta), \lambda(\theta) ] .
\end{equation}
Notice that, by construction, we have 
\begin{equation} \label{gammath}
	\Gamma_{\theta}({\sf x_1}) = {\sf x}_2 \qquad \text{and} \qquad \Gamma_{\theta}'({\sf x}_1) = 0.
\end{equation}
We notice that for all $\theta \in \left[ -\pi/2,0 \right]$
$$
\mathfrak{P}(\theta) \doteq | \lambda(\theta) - {\sf x}_1| = \sqrt{e^2 \sin^2{\theta} + \cos^2{\theta}} + \frac{(1-e^2){\cos{\theta}\sin{\theta}}}{{\sqrt{ \sin^2\theta + e^{2}\cos^2{\theta}}}} 
$$
is uniformly bounded from below: $
\mathfrak{P}(\theta) > e. $
We can then take $\lambda_0 \doteq  e$, so that
\begin{equation} \label{interval}
	[{\sf x}_1 - \lambda_0, {\sf x}_1 + \lambda_0] \subset (-\lambda(\theta), \lambda(\theta)) \qquad \forall\theta \in \left[ 0,  \pi/2\right],
\end{equation}
and define $\gamma : [-\lambda_0, \lambda_0] \longrightarrow \mathbb{R}$ by
$$
\begin{aligned}
	\gamma(x_1) & \doteq\Gamma_{\theta}(x_1+{\sf x}_1) \\[3pt]
	& =  - \dfrac{\widetilde{B}(\theta)}{2\widetilde{C}(\theta)}(x_1+{\sf x}_1) -\dfrac{1}{\widetilde{C}(\theta)} \sqrt{\widetilde{C}(\theta)-\frac{(x_1+{\sf x}_1)^2}{e^2}}  \qquad\forall \,x_1\in [-\lambda_0,\lambda_0],
\end{aligned}
$$
where $\Gamma_\theta : [-\lambda(\theta), \lambda(\theta) ] \to \mathbb{R}$ is as in \eqref{eq:Gamma}. In view of \eqref{interval}, we infer $\gamma\in\mathcal{C}^{\infty}([-\lambda_0,\lambda_0];\mathbb{R})$. From \eqref{gammath} we also deduce that
$\gamma(0)-{\sf x_2}=\gamma'(0)=0$ and $\gamma''(0) > 0$. Then, given $\tau \in [-\lambda_{0}, \lambda_{0}]$ we can find $\xi_{\tau} \in [-\lambda_{0}, \lambda_{0}]$ such that $|\xi_{\tau} - \tau | < |\tau |$ and
$$
\gamma[\theta](\tau)-{\sf x}_2 = \dfrac{\gamma''(0)}{2} \tau^{2} + \dfrac{\gamma'''(\xi_{\tau})}{6} \tau^{3} .
$$
This immediately yields the etimates
$$
\left( \dfrac{\gamma''(0)}{2} - \dfrac{\lambda_{0}}{6} \| \gamma''' \|_{L^{\infty}([-\lambda_{0}, \lambda_{0}])} \right) \tau^{2} \leq  \gamma[\theta](\tau) -{\sf x}_2 \leq \left( \dfrac{\gamma''(0)}{2} + \dfrac{\lambda_{0}}{6} \| \gamma''' \|_{L^{\infty}([-\lambda_{0}, \lambda_{0}])}\right) \tau^{2} .
$$
In fact, notice further that
$$
\gamma''(0) \geq \dfrac{1}{ e^{2}} \left( 1 + \dfrac{1}{ e^{2}} \right)^{-3/2} \,
$$
so that there exists $\lambda_{1} \in (0,\lambda_{0})$ independent of $\theta$ and $h$, sufficiently small so that 
$$
\dfrac{\gamma''(0)}{2} - \dfrac{\lambda_{1}}{6} \| \gamma''' \|_{L^{\infty}([-\lambda_{1}, \lambda_{1}])} > 0 .
$$
 The proof of \eqref{constants} is concluded by choosing $\lambda_{*} = \min \{ \lambda_{1}, \lambda_{2} \}$ and defining 
$$
\begin{aligned}
	& c_1^{(2)} \doteq \inf_{\theta\in(-\pi/2,0]}\left( \dfrac{\gamma''(0)}{2} - \dfrac{\lambda_{*}}{6} \| \gamma''' \|_{L^{\infty}([-\lambda_{*}, \lambda_{*}])} \right), \\[6pt]
	& c_2^{(2)}\doteq \sup_{\theta\in(-\pi/2,0]}\left( \dfrac{\gamma''(0)}{2} + \dfrac{\lambda_{*}}{6} \| \gamma''' \|_{L^{\infty}([-\lambda_{*}, \lambda_{*}])} \right), \\[6pt]
	& c_3^{(2)} \doteq \inf_{\theta\in(-\pi/2,0]}\left( \dfrac{1}{2}\frac{\partial}{\partial\theta}\gamma''(0) - \dfrac{\lambda_{*}}{6} \left\| \frac{\partial}{\partial\theta} \gamma''' \right\|_{L^{\infty}([-\lambda_{*}, \lambda_{*}])} \right), \\[6pt]
	& c_4^{(2)} \doteq \sup_{\theta\in(-\pi/2,0]}\left( \dfrac{1}{2}\frac{\partial}{\partial\theta}\gamma''(0) - \dfrac{\lambda_{*}}{6} \left\| \frac{\partial}{\partial\theta} \gamma''' \right\|_{L^{\infty}([-\lambda_{*}, \lambda_{*}])} \right) .
\end{aligned}
$$
\subsection{Proof of Lemma \ref{lem_geom}}\label{sec:Proof-appA}
Let $h < 0$ and $\theta \in [0,2\pi)$ be fixed and such that $(h,\theta) \in A_{1,e}$. 
When $\theta \in [0,\pi/2]$ we have, by a geometric argument, that the mappings 
\[
\begin{aligned}
& [0,\pi/2] \mapsto [-1,-e] :\theta \mapsto \mathsf x_2[\theta] \\
& [0,\pi/2]\mapsto [\kappa_2^{min},\kappa_2^{max}] : \theta \mapsto  \kappa_2[\theta]
\end{aligned}
\]
are diffeomorphisms. {This follows from explicit expressions : \eqref{eq:x1x2} for $\mathsf x_2[\theta]$ and \eqref{eq_kappa2_theta} below for $\kappa_2$.}  Hence, they induce, by composition, two diffeomorphisms $\mathcal K_2$ and $\mathsf X_2$ as claimed in assertion $(i).$  The property 
\eqref{eq_K_2etY} extends to every {$\theta \in \mathbb R$} by symmetry and periodicity. 

\medskip

To prove assertion $(ii)$ {and compute $\kappa_2$}, we provide some more technical computations. As before, $h$ has no influence here and we assume $h=0$ without loss of generality. 
We can rewrite identity \eqref{elipse0} as
\[
x_1^2 \left({\cos^2\theta} + \dfrac{\sin^2\theta}{e^2} \right) + x_2^2 \left({\sin^2\theta} + \dfrac{\cos^2\theta}{e^2} \right) + 2x_1x_2 \cos\theta\sin\theta \left( 1 - \dfrac{1}{e^2}\right) = 1.
\]
Around $\mathsf x= (\mathsf x_1,\mathsf x_2)$ we can then plug the ansatz
\[
x_1 = \mathsf x_{1} + \tau \qquad x_2 = \mathsf x_2+ \kappa_2 \tau ^2+ \kappa_3 \tau^3 + \text{l.o.t.} 
\]
Identifying powers of $\tau,$ we obtain the following sequence of equations. 
At order $\tau^0$, we have 
\[
\mathsf x_1^2 \left({\cos^2\theta} + \dfrac{\sin^2\theta}{e^2} \right) + \mathsf x_2 ^2 \left({\sin^2\theta} + \dfrac{\cos^2\theta}{e^2} \right) + 2\mathsf x_1 \mathsf x_2 \cos\theta\sin\theta \left( 1 - \dfrac{1}{e^2}\right) = 1.
\]
that we rewrite as
$$
\begin{aligned}
& \mathsf x_1 \left[ \mathsf x_1 \left({\cos^2\theta} + \dfrac{\sin^2\theta}{e^2} \right)  +  \mathsf x_2 \cos\theta\sin\theta \left( 1- \dfrac{1}{e^2}\right) \right] \\[6pt]
& \hspace{-4mm} + \mathsf x_2 \left[  \mathsf x_1  \cos\theta\sin\theta \left( 1 - \dfrac{1}{e^2}\right) +  \mathsf x_2  \left({\sin^2\theta}  + \dfrac{\cos^2\theta}{e^2} \right)\right] = 1 .
\end{aligned}
$$
At order $\tau$, we have
\[
\mathsf x_1 \left({\cos^2\theta} + \dfrac{\sin^2\theta}{e^2} \right)  +  \mathsf x_2 \cos\theta\sin\theta \left( 1 - \dfrac{1}{e^2}\right)  = 0 .
\]
The $\tau$ equation, taking the $\tau^0$ equation into account, implies 
\begin{equation} \label{eq_geometricfactor}
 \mathsf x_1  \cos\theta\sin\theta \left( 1 - \dfrac{1}{e^2}\right) +  \mathsf x_2  \left({\sin^2\theta}  + \dfrac{\cos^2\theta}{e^2} \right) =  \dfrac{1}{\mathsf x_2} .
\end{equation}
At order $\tau^2$, we get
\[
  \left({\cos^2\theta} + \dfrac{\sin^2\theta}{e^2} \right)  + 2 \kappa_2 \left(  \mathsf x_1  \cos\theta\sin\theta \left( 1 - \dfrac{1}{e^2}\right) +  \mathsf x_2  \left({\sin^2\theta}  + \dfrac{\cos^2\theta}{e^2} \right) \right) = 0,
\]
and therefore, using \eqref{eq_geometricfactor}, we infer that
\begin{equation} \label{eq_kappa2_theta}
\kappa_2 = - \dfrac{\mathsf x_2}{2}  \left({\cos^2\theta} + \dfrac{\sin^2\theta}{e^2} \right) .
\end{equation}
At order $\tau^3$, we see
\[
\begin{split}
\kappa_3 \left( \mathsf x_1  \cos\theta\sin\theta \left( 1 - \dfrac{1}{e^2}\right)  \right.  + & \left. \mathsf x_2  \left({\sin^2\theta}+ \dfrac{\cos^2\theta}{e^2} \right) \right) 
+  \kappa_2 \cos\theta \sin\theta \left( 1 - \dfrac{1}{e^2}\right)
= 0,
 \end{split}
\]
that is, due to \eqref{eq_geometricfactor},
\begin{equation} \label{eq_kappa3_theta}
\kappa_3 = - {\cos\theta \sin\theta}{\mathsf x_2} \kappa_2 \left( 1 - \dfrac{1}{e^2}\right) .
\end{equation}
Differentiating then \eqref{eq_kappa2_theta} with respect to $\theta$, we deduce that 
\[
\partial_{\theta} (\kappa_2 / \mathsf x_2) =   \cos\theta \sin\theta \left( 1- \dfrac{1}{e^2} \right) = - \dfrac{\kappa_3}{\kappa_2\mathsf x_2 } . 
\]
Multiplying this latter identity with $\kappa_2 \mathsf x_2$ yields \eqref{eq_kappa_2}.
\medbreak
Finally, concerning the third assertion, combining (i) and (ii), we see that 
$$\kappa_3 
=- \left(\kappa_2 - \dfrac{\kappa_2^2\mathsf X_2'(\kappa_2)}{\,\mathsf X_2(\kappa_2)}\right)\partial_{\theta}\kappa_2$$
and we therefore define
$$\mathcal {K}_3(\kappa_2) =  \int_{\kappa_2}^{\kappa_2^{max}} \left(\xi -  \dfrac{\xi^2\,\mathsf X_2'(\xi)}{\mathsf X_2(\xi)}\right)d\xi.$$ 
\hfill $\square$
\par For technical purpose, we need to compute the $\theta$-derivative of $\mathsf x_i[\theta],$ for $i=1,2$ and prove their boundedness.  
\begin{lemma}
	The derivatives $\partial_{\theta} \mathsf x_1[\theta],\partial_{\theta} \mathsf x_2[\theta]$ are uniformly bounded.
\end{lemma}
\begin{proof}
	We recall that $\mathsf x_1,\mathsf x_2$ is fixed by the conditions:
	\[
	\left\{
	\begin{aligned}
		&\mathsf x_1^2 \left({\cos^2\theta} + \dfrac{\sin^2\theta}{e^2} \right) + \mathsf x_2 ^2 \left(\dfrac{\sin^2\theta}{d^2} + \dfrac{\cos^2\theta}{e^2} \right) + 2\mathsf x_1 \mathsf x_2 \cos\theta\sin\theta \left( 1 - \dfrac{1}{e^2}\right) = 1\\[6pt]
		&\mathsf x_1 \left({\cos^2\theta} + \dfrac{\sin^2\theta}{e^2} \right)  +  \mathsf x_2 \cos\theta\sin\theta \left( 1 - \dfrac{1}{e^2}\right)  = 0,
	\end{aligned}
	\right.
	\]
	with the condition $\mathsf x_2 < 0.$ We drop here and below the $\theta$-dependency in the notations for a better legibility.  
	Differentiating with respect to $\theta$ we obtain that $(\partial_{\theta} \mathsf x_1,\partial_{\theta} \mathsf x_2)$ is solution to the system:
	\[
	\begin{pmatrix}
		\mathsf x_1 a_{\theta} + \mathsf x_2 d_{\theta}  & \mathsf x_2 b_{\theta} + \mathsf x_1 d_{\theta}
		\\
		a_{\theta} & d_{\theta} 
	\end{pmatrix}
	\begin{pmatrix}
		\partial_{\theta} \mathsf x_1 \\
		\partial_{\theta} \mathsf x_2
	\end{pmatrix}
	= S(\theta),
	\]
	where $S(\theta)$ is a periodic continuous source-term and
	\[
	a_{\theta} = \left({\cos^2\theta} + \dfrac{\sin^2\theta}{e^2} \right), \quad 
	b_{\theta} = \left({\sin^2\theta} + \dfrac{\cos^2\theta}{e^2} \right), \quad 
	d_{\theta} = \cos\theta\sin\theta \left( 1 - \dfrac{1}{e^2}\right) .
	\]
	In particular 
	\[
	\begin{pmatrix}
		\mathsf x_1 a_{\theta} + \mathsf x_2 d_{\theta}  & \mathsf x_2 b_{\theta} + \mathsf x_1 d_{\theta} 
		\\
		a_{\theta} & d_{\theta} 
	\end{pmatrix}^{-1}
	= 
	\dfrac{1}{\mathsf x_2 (d_{\theta}^2 - a_{\theta} b_{\theta}) }
	\begin{pmatrix}
		d_{\theta}  & -\mathsf x_2 b_{\theta} - \mathsf x_1 d_{\theta} 
		\\
		- a_{\theta} & \mathsf x_1 a_{\theta} + \mathsf x_2 d_{\theta}    
	\end{pmatrix}
	\]
	where:
	\[
	\mathsf x_2 (d_{\theta}^2 - a_{\theta} b_{\theta})  = 
	- \dfrac{\mathsf x_2}{e^2}\] 
	is continuous and periodic, see \eqref{eq:x1x2}. This immediately gives the statement.

\end{proof}

\section{Analysis of the Stokes asymptotics} \label{app_asymptotics}

In this section, we analyze the properties of the approximations 
to the solution to the Stokes system provided in Section \ref{sec:test-function}. We start with an elementary lemma that classifies the family of diverging integrals
\[
 \int_{-\lambda^*}^{\lambda^*}\frac{\tau^{p}}{(\mathfrak{d}+\gamma[\theta](\tau) -\mathsf x_2[\theta] )^{q}}\,d\tau \qquad
 (p,q) \in \mathbb N^2, 
\]
 that are ubiquitous in the forthcoming computations:
\begin{lemma}\label{lemma_Nico}
We have the following bounds when $\mathfrak d << 1:$
\begin{itemize}
\item if $p$ is even :
\[
\int_{-\lambda_*}^{\lambda_*} \dfrac{\tau^{p}}{(\mathfrak{d}+\gamma[\theta](\tau) -\mathsf x_2[\theta])^q}d\tau =
\left\{
\begin{aligned}
& \mathcal{O} \left(  \mathfrak{d}^{\frac{p+1}{2} - q} \right)
&& \text{ if $p \leq 2q -2$,} \\[4pt]
& \mathcal{O}(1) && \text{ if $p \geq 2q $,}
\end{aligned} 
\right.
\]
\item if  $p$ is odd:
\[
\int_{-\lambda_*}^{\lambda_*} \dfrac{\tau^{p}}{(\mathfrak{d}+\gamma[\theta](\tau) -\mathsf x_2[\theta])^q}d\tau =
\left\{
\begin{aligned}
& \mathcal{O} \left(  \mathfrak{d}^{\frac{p+2}{2} -q}\right)
&& \text{ if $p \leq 2q - 3$,}\\[4pt]
& \mathcal{O}(1) && \text{ if $p  \geq 2q -1 $.}
\end{aligned} 
\right.
\]

\end{itemize}
Furthermore, there exist constants $K_i[\theta]<\infty$ such that
\begin{equation}\label{formula}
 \int_{-\lambda^*}^{\lambda^*}\frac{\tau^{2i}}{(\mathfrak{d}+\gamma[\theta](\tau)-\mathsf x_2[\theta])^{i+1}}\,d\tau = \dfrac{K_{i}[\theta]}{\sqrt{\mathfrak{d}}} + \mathcal{O}(1)\qquad \forall i \geq 1.   
\end{equation}
\end{lemma}

{In the above lemma we have introduced Landau notations $\mathcal O$ for comparing functions. We recall that whenever real functions $f,g$ are defined in a neighborhood of $0$ (possibly only for strictly positive variables, say $\mathfrak d \in (0,\mathfrak d_0)$)  with $f \geq 0$ we say that $g = \mathcal O(f)$  if there exists a constant $K>0$ for which:
\[
|f(\mathfrak d)| \leq K g(\mathfrak d) \qquad \forall \mathfrak d \in (0,\mathfrak d_0).
\]
In our statement and also in those below, this constant will depend only on $d,e.$ Whenever $f$ depends on a second variable $\theta \in \mathbb R$ we used the same notations to denote:
 \[
|f(\mathfrak d,\theta)| \leq K g(\mathfrak d) \qquad \forall \mathfrak d \in (0,\mathfrak d_0) \quad \forall \theta \in \mathbb R.
\]
The constant $K$ is thus independent of $\theta.$}

\begin{proof}
Let us briefly sketch the proof of these bounds. 
The even case relies on \eqref{constants}.  We can then bound:
\[
\int_{-\lambda_*}^{\lambda_*} \dfrac{\tau^{p}}{[(\mathfrak{d}+\gamma[\theta](\tau)-\mathsf x_2[\theta])]^q} \, d\tau \leq \int_{-\lambda_*}^{\lambda_*} \dfrac{\tau^{p}}{(\mathfrak{d}+ {c_1^{(2)}} \tau^2 )^q} \, d\tau .
\]
In case $p \geq 2q$ the integrand is uniformly bounded in $\mathfrak d$ while if $p \leq 2q-2$ we obtain the expected result by performing the change of variable $\tau = \sqrt{\mathfrak d} s.$

In the odd case, we must first use the explicit expansion \eqref{eq_expansiongamma} of $\gamma[\theta]$ by remarking that
\[
\int_{-\lambda_*}^{\lambda_*} \dfrac{\tau^{p}}{(\mathfrak d + \kappa_2[\theta] \tau^2)^{q}}d\tau = 0,
\]
since $q$ is odd.  Remarking further that $|\gamma[\theta](\tau) -\mathsf x_2[\theta]- \kappa_2[\theta]\tau^2| \leq {c^{(3)} |\tau|^3}$ {with $c^{(3)}$} uniform in $\theta,$ we infer that
\begin{align*}
\left| 
\int_{-\lambda_*}^{\lambda_*} \dfrac{\tau^{p}}{[(\mathfrak{d}+\gamma[\theta](\tau)-\mathsf x_2[\theta])]^q} \, d\tau \right|
& = \left| \int_{-\lambda_*}^{\lambda_*} \tau^p \left(\dfrac{1}{(\mathfrak{d}+\gamma[\theta](\tau)-\mathsf x_2[\theta])^q} -  \dfrac{1}{(\mathfrak d + \kappa_2[\theta] \tau^2)^{q}} \right) d\tau \right|  \\
& \leq  \int_{-\lambda_*}^{\lambda_*} \dfrac{{c^{(3)}}|\tau|^{p+3} }{(\mathfrak{d} + {c_2^{(2)}} \tau^2)^q}d\tau .
 \end{align*}
We conclude then like in the even case.

Concerning the asymptotic expansions we apply the same symmetry trick in order to obtain 
\[
\int_{-\lambda_*}^{\lambda_*} \frac{\tau^{2i}}{(\mathfrak{d}+\gamma[\theta](\tau)-\mathsf x_2[\theta])^{i+1}}\,d\tau - 
\int_{-\lambda_*}^{\lambda_*} \frac{\tau^{2i}}{(\mathfrak{d}+ \kappa_2[\theta]\tau^2)^{i+1}}\,d\tau  = \mathcal{O}(1),
\]
and we compute then the second integral by performing again the change of variable $\tau = \sqrt{\mathfrak d} s$. This gives \eqref{formula} with 
$$
K_i(\theta)\doteq \int_{-\infty}^{\infty}\frac{s^{2i}}{(1+\kappa_2[\theta]s^2)^{i+1}}\,ds.
$$
\end{proof}

We consider now a boundary data $v_* = \nabla^{\bot} \psi_*$ and address the relevance of the approximation $\tilde{v}$
to the solution $v$ of the Stokes problem to \eqref{eq_stokesprofile} as constructed in Section \ref{sec:test-function}.  
We provide the arguments in the general case and mention the subsequent results with the specific boundary conditions concerning $\tilde{v}^{\bot},\tilde{v}^{||},\tilde{v}^{\circlearrowleft}. $
Following the conventions of Section \ref{sec:test-function} we assume that $\psi_*({\mathsf x_1},{\mathsf x}_2+h) = 0$ (actually $\psi_*$ is defined up to a constant so that this does not restrict the generality).  We point out that, with respect to  rough {\em a priori} computations, we might gain a factor $\mathfrak d^{1/2}$ in the following estimates thanks to symmetry arguments like the ones in the previous section. This might be used without specific notification.

\subsection{Size of $c_*$} We recall that $c_*$ is computed by matching the condition
\[
\int_{\mathsf x_1-\lambda_*}^{\mathsf x_1+\lambda_*} \partial_{222} \psi_{opt}(x_1,-L) dx_1 = 0 .
\]
Dropping $\theta$ dependencies again for legibility and replacing with explicit values, we obtain 
\[
-12 \int_{-\lambda_*}^{\lambda_*} \dfrac{\psi_1(\tau)}{(\mathfrak d + \gamma[\theta](\tau) -\mathsf x_2[\theta])^3}d\tau   + 6 \int_{-\lambda_*}^{\lambda_*} \dfrac{\psi_2(\tau)}{(\mathfrak d+ \gamma[\theta](\tau) -\mathsf x_2[\theta])^3}d\tau   = 0,
\]
 and therefore
 \[
 \begin{split}
c_* = - \dfrac{1}{\displaystyle \int_{-\lambda_*}^{\lambda_*} \dfrac{1}{(\mathfrak d+ \gamma[\theta](\tau) - \mathsf x_2[\theta])^3}\, d\tau } & \left( 6 \int_{-\lambda_*}^{\lambda_*} \dfrac{\partial_2 \psi_*(\mathsf x_1+\tau,h+\gamma[\theta](\tau))}{(\mathfrak d+   \gamma[\theta](\tau) - \mathsf x_2[\theta])^{2}}d\tau \right. \\
 & \qquad \left.  - 12  \int_{-\lambda_*}^{\lambda_*} \dfrac{\psi_*(\mathsf x_1 +\tau, h+\gamma[\theta](\tau))}{(\mathfrak d + \gamma[\theta](\tau)- \mathsf x_2[\theta])^3}d\tau \right) . 
\end{split}
 \]
 {In the general case $\psi_* \in \mathcal{C}^{\infty}(\mathbb R^2)$ with $\psi_*(\mathsf x) = 0$, we obtain readily that 
 \begin{equation} \label{eq_aprioric*}
 |c_*| \leq C \sqrt{\mathfrak d}\|\psi_*\|_{\mathcal{C}^1(A \cap |x_1| < 2d)}
 \end{equation}
 where $C$ depends only on $B$. However, in the various cases at-hand here, we obtain the finer expansion as stated in the next lemma.}
 \begin{lemma} \label{lem_c*}
 When $\mathfrak d << 1$, the expansions 
\[
\begin{aligned}
& c_*^{\bot} = \mathfrak d \kappa_3[\theta] c_{\infty}^{\bot}(\kappa_2[\theta]) + \mathcal{O}(\mathfrak d^{3/2}),\\
& c_*^{||} = \mathfrak dc_{\infty}^{||}(\kappa_2[\theta]) + \mathcal{O}(\mathfrak d^{3/2}),\\
& c_{*}^{\circlearrowleft} =   \mathfrak d c_{\infty}^{\circlearrowleft}(\kappa_2[\theta]) + \mathcal{O}(\mathfrak d^{3/2}),
\end{aligned}
\] 
holds true, uniformly in $\theta$. Moreover $c_{\infty}^{\bot}$, $c_{\infty}^{||}$ and $c_{\infty}^{\circlearrowleft}$ are $\mathcal{C}^{\infty}$ on $(0,\infty).$
 \end{lemma}

 \begin{proof}
 We provide a proof in the case of $c_{*}^{\bot}. $ The other cases follow from similar computations. 
Replacing $\psi_*^{\bot}$ with its explicit value, we have
\begin{equation} \label{eq_c*bot}
c_*^{\bot} = \dfrac{12}{\displaystyle \int_{-\lambda_*}^{\lambda_*} \dfrac{1}{(\mathfrak d+ \gamma[\theta](\tau) - \mathsf x_2[\theta])^3}\, d\tau} \int_{-\lambda_*}^{\lambda_*} \dfrac{\tau}{(\mathfrak d + \gamma[\theta](\tau)- \mathsf x_2[\theta])^3}d\tau.
\end{equation}
Let start with computing the denominator, i.e. 
\[
den = \int_{-\lambda_*}^{\lambda_*} \dfrac{1}{(\mathfrak d+ \gamma[\theta](\tau) - \mathsf x_2[\theta])^3}\, d\tau.
\]
Expanding $\gamma$ and using the control from below given by  \eqref{constants}, we infer
\begin{multline*}
den = \int_{-\lambda_*}^{\lambda_*}  \left( \dfrac{1}{(\mathfrak d+ \kappa_2[\theta] \tau^2)^3}  - 3 \dfrac{\kappa_3[\theta] \tau^3}{(\mathfrak d+ \kappa_2[\theta] \tau^2)^4}  \right)d\tau \\
+ \mathcal O \left(  \int_{-\lambda_*}^{\lambda_*} \left( \dfrac{\tau^4 }{(\mathfrak d+ c_2^{(2)} \tau^2)^4}  +  \dfrac{\tau^6 }{(\mathfrak d+ c_2^{(2)} \tau^2)^5}   \right) d\tau \right).
\end{multline*}
By symmetry, the integral of the second term on the first line vanishes while we set $\tau = \sqrt{\mathfrak d}s$ when computing the integral of the first term. By straightforward estimate of remainder terms at infinity, we deduce
\[
\int_{-\lambda_*}^{\lambda_*}  \left( \dfrac{1}{(\mathfrak d+ \kappa_2[\theta] \tau^2)^3}  - 3 \dfrac{\kappa_3[\theta] \tau^3}{(\mathfrak d+ \kappa_2[\theta] \tau^2)^4}  \right)d\tau  = \dfrac{1}{{\mathfrak d}^{5/2}}\int_{-\infty}^{\infty} \dfrac{1}{(1+\kappa_2[\theta] s^2)^3}\, ds  + \mathcal O(1).
\] 
Below, we denote the integral on the right-hand side by $den_{\infty}(\kappa_2[\theta])$. We remark that it is a smooth function of $\kappa_2[\theta]$ when it ranges $(0,\infty)$ with strictly positive values when $\kappa_2[\theta] \in [\kappa_2^{min},\kappa_2^{max}].$
With a similar change of variable in the integral, we observe that the remainder term on the second line of the expression of $den$ is $\mathcal O({\mathfrak d}^{-3/2}).$

Concerning the numerator, we proceed similarly, remarking that the leading term $\frac{\tau}{(\mathfrak d + \kappa_2[\theta] \tau^2)^3}$ leads to a vanishing integral. We then obtain 
\[
 \int_{-\lambda_*}^{\lambda_*} \dfrac{\tau}{(\mathfrak d + \gamma[\theta](\tau)- \mathsf x_2[\theta])^3}d\tau =  \dfrac{ (-3) \kappa_3[\theta]}{{\mathfrak d}^{\frac 32}} \int_{-\infty}^{\infty} \dfrac{\tau^2}{({\mathfrak d} + \kappa_2[\theta]\tau^2)^4} \, d\tau + \mathcal O\left( \dfrac{1}{\mathfrak d}\right).
\]
We point out that the remainder term could be made $\mathcal O(\mathfrak d^{-1/2})$ by playing again on symmetries but this will have no influence below. 
We denote  the integral appearing in the right-hand side of this latter identity by $num_{\infty}(\kappa_2[\theta])$. We remark again that it is a smooth function of $\kappa_2[\theta]$ with strictly positive values when $\kappa_2[\theta] \in [\kappa_2^{min},\kappa_2^{max}].$  We eventually obtain 
\[
c_*^{\bot}  = - \dfrac{36 \kappa_3[\theta] \mathfrak d^{-3/2} num_{\infty}(\kappa_2[\theta]) + O(\mathfrak d^{-1})}{ {\mathfrak d}^{-5/2} den_{\infty}(\kappa_2[\theta])+\mathcal O(\mathfrak d^{-3/2})} = \mathfrak d \kappa_3[\theta] c_{\infty}^{\bot}(\kappa_2[\theta]) + \mathcal O(\mathfrak d^{3/2}),
\]
where
\[
c_{\infty}^{\bot}(\kappa_2) =  - 36 \dfrac{num_{\infty}(\kappa_2)}{den_{\infty}(\kappa_2)} \in  \mathcal{C}^{\infty}((0,\infty)).
\]
 \end{proof}
  The formula \eqref{eq_c*bot}  entails that $c_*^{\bot}$ depends on the gap geometry through both parameters $\mathfrak d$ and $\theta.$
Considering the mapping $(\mathfrak d,\theta) \to c_*^{\bot}$ -- that is smooth on $(0,\infty) \times \mathbb R$ by standard parameter-integral arguments -- we prove the following bounds on $\nabla c_*^{\bot}$.
\begin{lemma}
We have the following inequalities for $\mathfrak d << 1$ and $\theta \in \mathbb R:$
\begin{equation} \label{eq_cdot}
|\partial_{\theta} {c}_*^{\bot}| \leq C_{geo}  {\mathfrak d},  \qquad |\partial_{\mathfrak d} {c}_*^{\bot}| \leq C_{geo},
\end{equation}
with a constant $C_{geo}$ depending only on $d,e.$ 
\end{lemma} 
\begin{proof}
The proof follows from the following computations : 
$$
\begin{aligned}
\partial_{\theta}{c}_*^{\bot} & = \dfrac{36\, \displaystyle\int_{-\lambda_*}^{\lambda_*} \dfrac{\tau}{(\mathfrak d + \gamma[\theta](\tau)- \mathsf x_2[\theta])^3}d\tau \, \int_{-\lambda_*}^{\lambda_*} \dfrac{ \left( \partial_{\theta}\gamma[\theta](\tau) - \partial_{\theta} \mathsf x_2[\theta] \right)}{(\mathfrak d + \gamma[\theta](\tau)- \mathsf x_2[\theta])^4}d\tau}{\displaystyle \left( \int_{-\lambda_*}^{\lambda_*} \dfrac{1}{(\mathfrak d+ \gamma[\theta](\tau) - \mathsf x_2[\theta])^3}\, d\tau\right)^2}   \\[6pt]
& \hspace{4mm} - \dfrac{36\, \displaystyle\int_{-\lambda_*}^{\lambda_*} \dfrac{ \left( \partial_{\theta}\gamma[\theta](\tau) - \partial_{\theta} \mathsf x_2[\theta] \right)\tau}{(\mathfrak d + \gamma[\theta](\tau)- \mathsf x_2[\theta])^4}d\tau}{\displaystyle \int_{-\lambda_*}^{\lambda_*} \dfrac{1}{(\mathfrak d+ \gamma[\theta](\tau) - \mathsf x_2[\theta])^3}\, d\tau} , \\
\partial_{\mathfrak d}{c}_*^{\bot}  &=  \dfrac{36\displaystyle\, \int_{-\lambda_*}^{\lambda_*} \dfrac{1}{(\mathfrak d + \gamma[\theta](\tau)- \mathsf x_2[\theta])^4}\, d\tau \, \int_{-\lambda_*}^{\lambda_*} \dfrac{\tau}{(\mathfrak d + \gamma[\theta](\tau)- \mathsf x_2[\theta])^3}d\tau}{\displaystyle \left( \int_{-\lambda_*}^{\lambda_*} \dfrac{1}{(\mathfrak d+ \gamma[\theta](\tau) - \mathsf x_2[\theta])^3}\, d\tau\right)^2} \\[6pt]
& \hspace{4mm} - \dfrac{36\, \displaystyle \int_{-\lambda_*}^{\lambda_*} \dfrac{\tau}{(\mathfrak d + \gamma[\theta](\tau)- \mathsf x_2[\theta])^4}d\tau}{\displaystyle \int_{-\lambda_*}^{\lambda_*} \dfrac{1}{(\mathfrak d+ \gamma[\theta](\tau) - \mathsf x_2[\theta])^3}\, d\tau}\\
\end{aligned}
$$
arguing then as in the proof of {Lemma \ref{lem_c*}}.
\end{proof}

\subsection{Quality of the approximation}
{In what follows we compare $\tilde{v}$ reconstructed from $\psi_{opt}$ to the exact solution $(v,q)$ to \eqref{eq_stokesprofile} with the corresponding boundary data $v_* = \nabla^{\bot} \psi_*$. Again, we explain the computations in the general case and write down the results in the three cases at hand here.} 
Given the variational characterization of $\psi_{opt}$ we have:
\begin{equation} \label{eq_boundbelow}
\int_{\mathcal G_{\lambda_*}} |\partial_{22} \psi_{opt}|^2\, dx \leq \int_{\Omega} |\nabla v|^2\, dx.
\end{equation}
Owing to the fact that $\tilde{v}$ matches the same boundary conditions as ${v}$ on $\partial \Omega,$ we infer  also that:
\[
\int_{A} |\nabla (v -\tilde v)|^2 \, dx= \int_{\Omega} |\nabla (v -\tilde v)|^2 \, dx\leq \int_{\Omega} |\nabla \tilde{v}|^2 \, dx- \int_{\Omega} |\nabla{v}|^2 \, dx
\]
and, by \eqref{eq_boundbelow}:
\begin{equation} \label{eq_error}
\int_{A} |\nabla (v -\tilde{v})|^2 \, dx\leq \int_{\Omega} |\nabla \tilde{v}|^2 \, dx- \int_{\mathcal G_{\lambda_*}} |\partial_{22} \psi_{opt}|^2 \, dx.
\end{equation}
We remark here that our candidate $\tilde{v}$ is uniformly bounded in terms of $(h,\theta)$ outside $\mathcal G_{\lambda_*}. $
This implies that
\[
\int_{A} |\nabla {v} - \nabla {\tilde{v}}|^2  \, dx\leq \int_{\mathcal G_{\lambda_*}}  \left( |\nabla^2 \psi_{opt}|^2 - |\partial_{22}\psi_{opt}|^2 \right)\, dx + \mathcal{O}(1).
\]
The relevance of our approximation $\tilde{v}$ is then related to the computation of the $L^2$-norms of $\partial_{11} \psi_{opt}$ and $\partial_{12} \psi_{opt}$  in terms of 
$\psi_*$ and  the geometrical descriptors $\mathfrak d,\theta.$ {For instance, in the case of $\psi^{\bot}_{opt}$, adopting the change of variables in \eqref{change_variables} and using Lemmas \ref{lemma_Nico}- \ref{lem_c*}, we get:
\begin{multline}\label{stime_test2}
\int_{\mathcal G_{\lambda_*}}  \left( |\nabla^2 \psi^{\bot}_{opt}|^2 - |\partial_{22}\psi^{\bot}_{opt}|^2 \right)\, dx=\int_{\mathcal{G}_{\lambda_*}}\left(2|\partial_{12}\psi_{opt}^{\bot} |^2+|\partial_{11}\psi_{opt}^{\bot}|^2\right)\, dx\\
\le C\int_{-\lambda_*}^{\lambda_*} \left(\frac{{(\tau-c_*^{\bot})^2 \tau^2}}{(\mathfrak{d}+\gamma[\theta](\tau)-{\sf x_2})^3} { + \dfrac{1}{(\mathfrak{d}+\gamma[\theta](\tau)-{\sf x_2})}}\right)\,d\tau=\mathcal{O}(\mathfrak d^{-1/2})+\mathcal{O}(1).
\end{multline}
}
By similar computations for $\psi_{opt}^{||},\psi_{opt}^{\circlearrowleft}$, we obtain the next lemma.
\begin{lemma} \label{lem_relevance}
 We have the expansions when $\mathfrak d << 1$ uniformly in $\theta$:
\[
\|\nabla (v^{\bot} - \tilde{v}^{\bot})\|_{L^2(A)} =\mathcal{O}(\mathfrak d^{-1/2}),
\qquad
\|\nabla (v^{||} - \tilde{v}^{||})\|_{L^2(A)} + \|\nabla (v^{\circlearrowleft} - \tilde{v}^{\circlearrowleft})\|_{L^2(A)}   = \mathcal{O}(1).
\] 
\end{lemma}

\subsection{Pressure.} 
We construct now a pressure $\tilde{q}$ so that 
\[
- \Delta \tilde{v}  +\nabla \tilde{q} = R \qquad {\rm div} \tilde{v} = 0,
\] 
where $R$ is bounded in a suitable sense.  For this, we remark that $\tilde{v}$ is uniformly bounded outside the gap $\mathcal G_{\lambda_*}$ whatever the value of $\mathfrak d.$ Hence, we focus on the construction of $\tilde{q}$ inside $\mathcal G_{\lambda_*}.$  By construction, we have 
\[
\tilde{v} = \nabla^{\bot} \psi_{opt}\quad  \text{ and } \quad 
- \Delta \tilde{v}
 = \begin{pmatrix}
  \partial_{211} \psi_{opt} + \partial_{222} \psi_{opt} \\
  - \partial_{111} \psi_{opt} - \partial_{122} \psi_{opt}
 \end{pmatrix}.
\] 
In this expression, the most diverging term should be $\partial_{222} \psi_{opt}(x_1,x_2) =~d_{222}(x_1)$ and  it is a function of $x_1$ only. 
Hence,  we set, for $(x_1,x_2) \in \mathcal G_{\lambda_*}:$
\[
\begin{aligned}
\tilde{q}(x_1,x_2) &= - \int_{\mathsf x_1-\lambda_{*}}^{x_1} 
d_{222}(\tau)d\tau + \partial_{12} \psi_{opt}(x_1,x_2)  \\[6pt]
&=  \int_{x_1}^{\mathsf x_1+\lambda_*} d_{222}(\tau)d\tau  + \partial_{12} \psi_{opt}(x_1,x_2).
\end{aligned}
\]
The second identity uses that the mean of $d_{222}$ vanishes. We can extend this pressure by $0$ without generating singularities.  We point out that astonishingly,  this mean-free property seems to appear incidentally while it is deeply related to the variational construction of $\psi_{opt}.$ In particular $-\Delta \tilde{v} + \nabla \tilde{q}$ is uniformly bounded outside $\mathcal G_{\lambda_*/2}$. Then, for arbitrary $(w_1,w_2) \in [\mathcal{C}^{\infty}(\overline{\mathcal G_{\lambda_*}})]^2$ that vanish on the top and bottom boundaries, we have then
$$
\begin{aligned}
\int_{\mathcal G_{\lambda_*/2}} (-\Delta \tilde{v} & + \nabla \tilde{q}) \cdot  (w_1,w_2) \, dx \\[6pt]
& = \int_{\mathcal G_{\lambda_*/2}} \left(2 \partial_{211} \psi_{opt} w_1 - \partial_{111} \psi_{opt} w_2\right)\, dx \\[6pt]
&  = - \int_{\mathsf x_1-\lambda_*/2}^{\mathsf x_1+\lambda_*/2} \int_{-L}^{h + \gamma[\theta](x_1-\mathsf x_1)} 
\left( 2 \partial_{11} \psi_{opt} \partial_2 w_1 +
 \partial_{111} \psi_{opt} w_2\right) dx_1 dx_2
\end{aligned}
$$
 where we integrated by parts the first term.  For the second term we use the \textit{primitive} operator:
 \[
 P_2 \partial_{111} \psi_{opt}(x_1,x_2) = - \int_{x_2}^{h + \gamma[\theta](x_1- \mathsf x_1)}  \partial_{111} \psi_{opt}(x_1,z)dz,
 \]
 so that:
 $$
 \begin{aligned}
  \int_{\mathcal G_{\lambda_*}} (-\Delta \tilde{v} + & \nabla \tilde{q}) \cdot  (w_1,w_2)\, dx\\[6pt]
& \hspace{-4mm} = - \int_{\mathsf x_1-\lambda_*/2}^{\mathsf x_1+\lambda_*/2} 
\int_{-L}^{h + \gamma[\theta](x_1-\mathsf x_1)} 
\left( 2 \partial_{11} \psi_{opt}\partial_2 w_1 -
 P_2 \partial_{111} \psi_{opt}\partial_2 w_2\right) dx_1 dx_2.
 \end{aligned}
 $$
We can then estimate this latter term with H\"older inequalities.  In the general case we obtain the following statement. 
{
\begin{lemma} 
 Under the assumption that $\psi_*({\mathsf x_1},{\mathsf x}_2+h) = 0,$ there exists a constant $C_*$ depending only on $B$ such that
there exists a pressure $\tilde{q} \in L^2(\Omega)$ for which
\[
\|-\Delta \tilde{v} + \nabla \tilde{q}\|_{H^{-1}(\Omega)} \leq C_* \|\psi_*\|_{\mathcal{C}^3(A \cap |x_1| < d)} .
\]
\end{lemma}
The proof is based on the above identity combined with \eqref{eq_aprioric*} and the analysis of the operator $P_2$ given below in Appendix \ref{app_P2}.
}

A straightforward consequence to the previous lemma is that, if we are given an approximation $\tilde{v}=\nabla^\perp \psi_{opt}$ and two divergence-free test-functions $w,\tilde{w}$ that match the same boundary condition on $\partial B$ and $\partial A,$ we have 
\begin{equation} \label{eq_stokes_optim}
\left|\int_{\Omega} \nabla w : \nabla \tilde{v} \, dx- \int_{\Omega} \nabla \tilde{w} : \nabla \tilde{v}\, dx \right| \leq C_* \|\nabla (w - \tilde{w})\|_{L^2(\Omega)} .
\end{equation}
\subsection{Computation of $M_{opt}^{\infty}$}\label{I_k2} 
We split the expression in three terms :$M_{opt}^{\infty} = 6 G_1 + 12 G_2 + 6 G_3$, where
\[
\begin{aligned}
G_1 &\doteq\int_{-\lambda_*}^{\lambda_*} 
\dfrac{(\tau - c_*^{\bot}) ( \tau^2 - c_*^{\circlearrowleft}) }{(\mathfrak d + \gamma[\theta](\tau) - \mathsf x_2[\theta])^3} \, d\tau, \\
G_2 & \doteq\int_{-\lambda_*}^{\lambda_*} 
\dfrac{	 \mathsf{x}_2 (\tau - c_*^{\bot})
 (\gamma[\theta](\tau)- \mathsf x_2[\theta] + c_*^{||} )  }{(\mathfrak d + \gamma[\theta](\tau) - \mathsf x_2[\theta])^3} \, d\tau,
\\
G_3& \doteq\int_{-\lambda_*}^{\lambda_*} 
\dfrac{ {\mathsf x}_2(\tau - c_*^{\bot})}{(\mathfrak d + \gamma[\theta](\tau) - \mathsf x_2[\theta])^2} \, d\tau.
\end{aligned}
\]
The true challenge of these computations is to show that all these integrals can be written as the multiplication of $\kappa_3$ with some integral depending on $\kappa_2$. To this end, we may either exploit the fact that $c_*^{\bot}$ has this form or exploit symmetries to cancel integrals with an odd power of $\tau$ on the denominator.  
For instance, we have
\[
G_1 = \int_{-\lambda_*}^{\lambda_*} 
\dfrac{\tau^3  - c_*^{\bot}\tau^2  - c_*^{\circlearrowleft} \tau  + c_*^{\bot} c_*^{\circlearrowleft} }{(\mathfrak d + \gamma[\theta](\tau) - \mathsf x_2[\theta])^3} \, d\tau.
\]
Exploiting symmetries, we infer 
\[
\begin{aligned}
 \int_{-\lambda_*}^{\lambda_*} 
\dfrac{\tau^3}{(\mathfrak d + \gamma[\theta](\tau) - \mathsf x_2[\theta])^3}\, d\tau  & = -3 \dfrac{\kappa_3 I_{6,4}(\kappa_2)}{\sqrt{\mathfrak d}} + \mathcal O(1) \\
  \int_{-\lambda_*}^{\lambda_*} 
\dfrac{\tau}{(\mathfrak d + \gamma[\theta](\tau) - \mathsf x_2[\theta])^3}\, d\tau  & = -3 \dfrac{\kappa_3 I_{4,4}(\kappa_2)}{\mathfrak d^{\frac 32}} + \mathcal O(1) 
\end{aligned}
\]
where here and below, we fix the notation $I_{p,q}(\kappa_2)$ by
\[
I_{p,q}(\kappa_2) = \int_{\mathbb R} \dfrac{\tau^{p}}{(1+\kappa_2 \tau^2)^{q}}d\tau  \qquad \forall\, 0 \leq p < 2q  -1.
\]
Recalling Lemma \ref{lem_c*} to rewrite $c_*^{\bot}$ and $c_*^{\circlearrowleft},$ we deduce that
\[
\begin{split}
G_1 = \dfrac{\kappa_3}{\sqrt{\mathfrak d}} \Big( -3I_{6,4}(\kappa_2) & - 3
c_{\infty}^{\bot}(\kappa_2)  I_{4,4}(\kappa_2)   \\
&  + 3 c_*^{\circlearrowleft} (\kappa_2) I_{4,4}(\kappa_2)  + c_{\infty}^{\bot}(\kappa_2)   c_{\infty}^{\circlearrowleft} (\kappa_2) I_{0,3}(\kappa_2) \Big) + \mathcal O(1). 
\end{split}
\]
As for $G_2,G_3,$ we first remind that, Assertion (i) in Lemma \ref{lem_geom} implies $x_2 = X_2(\kappa_2),$ and then, arguing as previously, we eventually conclude that
\[
\begin{aligned}
G_2 = \kappa_3 \dfrac{X_2(\kappa_2)}{\sqrt{\mathfrak d}} \Big( - 3\kappa_2
 I_{6,4}(\kappa_2) & - \kappa_2  c_{\infty}^{\bot}(\kappa_2) I_{2,4}(\kappa_2)\\ 
 &  - 3  c_{\infty}^{||}(\kappa_2) I_{4,4}(\kappa_2) -  c_{\infty}^{\bot}(\kappa_2) c_{\infty}^{||}(\kappa_2) I_{0,3}(\kappa_2)  \Big) + \mathcal O(1),
\end{aligned}
\]
and 
\[ 
 G_3 = \kappa_3 \dfrac{X_2(\kappa_2)}{\sqrt{\mathfrak d}} \left( -3 I_{4,3}(\kappa_2) - c_{\infty}^{\bot}(\kappa_2)I_{0,2}(\kappa_2) \right) + \mathcal O(1).
 \]
This ends the proof.

\subsection{Computation of $J_1$} \label{app_J1}
This section is devoted to the estimate \eqref{eq_J1estim} of $J_1$ defined by
\[
\int_{T_-}^{\sigma}  \int_{\Omega(t)} [\partial_{t} (a(t)\tilde{v}^{\bot}) + (u \cdot \nabla) (a(t)\tilde{v}^{\bot})] \cdot u\, dxdt = J_1^{(a)} + J_{1}^{(b)} + J_{1}^{(c)},
\]
where
\[
\begin{aligned}
J_1^{(a)} & =\int_{T_-}^{\sigma}  \int_{\Omega(t)} {a}'(t)\tilde{v}^{\bot} \cdot u\, dxdt, \\[6pt]
J_1^{(b)} &= \int_{T_-}^{\sigma}  \int_{\Omega(t)} {a}(t) \partial_t \tilde{v}^{\bot} \cdot u\, dxdt, \\[6pt]
J_1^{(c)} &=  \int_{T_-}^{\sigma} a(t)  \int_{\Omega(t)}  (u \cdot \nabla) \tilde{v}^{\bot} \cdot u\, dxdt.
\end{aligned}
\]
Regarding  $J_1^{(c)}$, we first remark that $\tilde{v}^{\bot}$ is, by construction, uniformly bounded outside $\mathcal G_{\lambda_*}$, see \eqref{stime_test}. 
Thus there holds 
\[
|J_1^{(c)}| \leq C_{geo}  \int_{T_-}^{\sigma}\left( \|u\|^2_{L^2(\Omega)} + \left| \int_{\mathcal G_{\lambda_*}} ( u \cdot \nabla)\tilde{v}^{\bot} \cdot u \, dx\right|\right)\, dt.
\]
To compute the last integral, we introduce
\[
P_2 \nabla \tilde{v}^{\bot}(x_1,x_2) = -\int_{x_2}^{h+ \gamma[\theta](x_1-\mathsf x_1[\theta])} \nabla \tilde{v}^{\bot} (x_1,z)\, dz 
\]
so that, after integration by parts (noting that $P_2 \nabla \tilde{v}^{\bot}$ vanishes on the top-boundary of $\mathcal G_{\lambda_*}$ while $u$ vanishes on the lower boundary):
$$
\begin{aligned}
&\left| \int_{\mathcal G_{\lambda_*}} ( u \cdot \nabla \tilde{v}^{\bot} )\cdot u\, dx\right|   = \left| \int_{\mathcal G_{\lambda_*}}
 \partial_2 u \otimes u : P_2 \nabla \tilde{v}^{\bot} \, dx \right| \\[6pt]
&  \leq \int_{-\lambda_*}^{\lambda_*} \sup_{x_2 \in (-L,h+\gamma[\theta](\tau))} |P_2 \nabla \tilde{v}^{\bot}|  \left( \int_{-L}^{h+\gamma[\theta](\tau)}  |\partial_2 u(\mathsf x_1+\tau,x_2)|^2dx_2 \right)^{\frac 12}\, d\tau \\[6pt]
& \hspace{4mm} \times\left( \int_{-L}^{h+\gamma[\theta](\tau)}  | u(\mathsf x_1 + \tau,x_2)|^2 dx_2\right)^{\frac 12} \\[6pt]
& \leq  C_{geo}  \left(\sup_{\tau \in (-\lambda_*,\lambda_*)}\sup_{x_2 \in (-L,h+\gamma[\theta](\tau))} (\mathfrak d + \gamma[\theta]( \tau) - \mathsf x_2[\theta])|P_2 \nabla \tilde{v}^{\bot}| \right)
\int_{\mathcal G_{\lambda_*}} |\partial_2 u|^2\, dx,
\end{aligned}
$$
where we have used Poincar\'e inequality with optimal constant on the line $x_1 = \mathsf x_1+\tau,$ $x_2\in (-L,h+\gamma[\theta](\mathsf x_1+\tau))$ to pass from the second to the last line. With the explicit formulas at-hand, we obtain
\[
\left[\sup_{x_2 \in (-L,h+\gamma[\theta](\mathsf x_1+\tau))} (\mathfrak d + \gamma[\theta](\mathsf x_1 + \tau) - \mathsf x_2[\theta])|P_2 \nabla \tilde{v}^{\bot}| \right] \leq C_{geo}
\] 
so that
\begin{equation} \label{eq_J1c}
|J_1^{(c)}| \leq C_{geo} \int_{T_-}^{\sigma} \int_{A} \|\nabla u(\tau)\|^2\, dxd\tau.
\end{equation}

\medskip 

To compute $J_1^{(a)},$ we simply use the control on $a'$ obtained in Lemma \ref{lem_Mod} combined with Poincaré inequality.  By \eqref{stime_test}, this yields the desired estimate, that is 
\begin{equation} \label{eq_J1a}
|J_1^{(a)}| \leq C_{geo} \int_{T_{-}}^{\sigma} |{\theta}'(\tau)| \|u(\tau)\|_{L^2(A)}\, d\tau \leq C_{geo} \int_{T_-}^{\sigma} \|\nabla u(\tau)\|^2\, d\tau.
\end{equation}

It remains to compute $\partial_{t} \tilde{v}^{(\bot)}.$ 
By construction $\tilde{v}^{\bot}(t,x) = \nabla \tilde{\psi}^{\bot}(t,x)$, 
where
\[
\tilde{\psi}^{\bot}(t,x)
= 
\left\{
\begin{aligned}
& \zeta\left( \dfrac{ (x_1 - \mathsf x_1)}{\lambda_*} \right)  \psi^{\bot}_{opt}(t,x)\\& + \left(1-\zeta\left( \dfrac{ (x_1 -\mathsf x_1) }{\lambda_*} \right)\right) \zeta\left( \dfrac{{\rm dist}({x},B)}{d_{*}}\right) \left( (x_1- \mathsf x_1) - c^{\bot}_*)  \right) \; && \text{ in $\mathcal G_{2\lambda_*}$,} \\[6pt]
&  \zeta\left( \dfrac{{\rm dist}({x},B)}{d_{*}}\right) \left( (x_1- \mathsf x_1) - c^{\bot}_* \right) 
\; &&  \text{ in $\Omega \setminus \mathcal G_{2\lambda_*}$,} \\[6pt]
& (x_1 - \mathsf x_1) - c^{\bot}_* && \text{ in $B$}
\end{aligned}
\right.
\]
and
\[
\psi_{opt}^{\bot}(t,x) =  \left((x_1-\mathsf x_1) - c_*^{\bot} \right) P_1^{opt} \left(\dfrac{x_2 + L}{\mathfrak{d} + \gamma(x_1 - \mathsf x_1) - \mathsf x_2[\theta]} \right).
\]
We remind that this formula depends on time through $\mathsf x_1,\mathsf x_2$, $\theta$ and therefore also $\mathfrak d$. In particular, we can split
\[
\partial_t \tilde{\psi}^{\bot} = {\mathsf x}'_1 \psi_t^{(1)} + {\mathsf x}'_2 \psi_t^{(2)} + {\mathfrak d}' \psi_t^{(3)}  + [{c}_{*}^{\bot}]' \psi_t^{(4)}
\quad \text{ in $\Omega\setminus \mathcal G_{\lambda_*}$},
\]  
where $\psi_t^{(i)}$ ($i=1,2,3,4$) are smooth functions, while, in $\mathcal G_{\lambda_*}$, we have 
\[
\begin{aligned}
\partial_t  \psi^{\bot}(t,x) & = \partial_t \psi_{opt}^{\bot}(t,x) = 
\psi_{opt}^{reg}(t,x) + \psi_{opt}^{div}(t,x) 
\end{aligned}
\]
where
\[
\begin{aligned}
 \psi_{opt}^{reg}(t,x) = & -({\theta}' \partial_{\theta}{\mathsf x}_1 + [{c}_*^{\bot}]' )P_1^{opt} \left(\dfrac{x_2 + L}{\mathfrak{d} + \gamma(x_1 - \mathsf x_1) - \mathsf x_2[\theta]} \right), \\
 \psi_{opt}^{div}(t,x)  = &\! -\!\left( {\mathfrak d}'+ {\theta}' \left( \partial_{\theta} \gamma(x_1 \!-\! \mathsf x_1)  -  \partial_{\theta} \mathsf x_1 \partial_x \gamma(x_1-\mathsf x_1)\! -\! \partial_{\theta} \mathsf x_2[\theta] \right) \right)\!\! \dfrac{(x_2+L)}{(\mathfrak{d} + \gamma(x_1 - \mathsf x_1) - \mathsf x_2[\theta])^2 }  \\
&\times\left((x_1-\mathsf x_1) - c_*^{\bot} \right) \partial_z P_1^{opt} \left(\dfrac{x_2 + L}{\mathfrak{d} + \gamma(x_1 - \mathsf x_1) - \mathsf x_2[\theta]} \right).
\end{aligned}
\]
Using  the differentials computed in Appendix \ref{app_difftheta}, the previous computations of integrals in this Appendix \ref{app_asymptotics} and the use of the operator $P_2$ (see Appendix \ref{app_P2}), 
we infer that 
$$
\|\nabla  \psi_{opt}^{reg}  (t,x)\|_{L^2(\mathcal G_{\lambda_*})}
+ \|\nabla P_2 \psi_{opt}^{div}(t,x)\|_{L^2(\mathcal G_{\lambda_*})} \leq  C_{geo} \left( |{\theta}'| + |{\mathfrak d}'| \right),
$$
where
\[
P_2 \psi_{opt}^{div}(t,x) = - \int_{x_2}^{h+\gamma(x_1-\mathsf x_1)} \psi_{opt}^{div}(t,x_1,z)dz.
\]
Eventually, we obtain 
\[
\begin{aligned}
\int_{\Omega(t)}\partial_{t} \tilde{v}^{\bot} \cdot u\, dx& = \int_{\Omega(t) \setminus \mathcal G_{\lambda_*}}  \partial_{t} \tilde{v}^{\bot} \cdot u \, dx+ 
\int_{\mathcal G_{\lambda_*}} \nabla^{\bot} \psi_{opt}^{reg} \cdot u \, dx
-   \int_{\mathcal G_{\lambda_*}} \nabla^{\bot} P_2 \psi_{opt}^{div} \cdot \partial_2 u \, dx
 \\
\end{aligned}
\]
so that combining the previous computations with \eqref{eq_dissipationsolide}, we get
\begin{equation} \label{eq_J1b}
|J_1^{(b)}| \leq  C_{geo} \int_{T_-}^{\sigma} \left( |{\theta}'| + |{\mathfrak d}'| \right) \|\nabla u\|_{L^2(A)}   \leq C_{geo}  \int_{T_-}^{\sigma}  \|\nabla u(\tau)\|_{L^2(A)}^2 \, d\tau.
\end{equation}
Combining \eqref{eq_J1c}-\eqref{eq_J1a}-\eqref{eq_J1b} yields \eqref{eq_J1}.

\section{Properties of the operator $P_2$} \label{app_P2}

In the previous computations we have used several times the operator 
\[
P_2 f(x_1,x_2) = - \int_{x_2}^{h+\gamma[\theta](x_1-\mathsf x_1)} f(x_1,z)dz \quad \forall x \in \mathcal G_{\lambda} .
\]
for a function $f$ defined in the gap $\mathcal G_{\lambda}$.  The main purpose of this construction is to gain regularity in the $x_2$ variable.  Precisely, we have the following
properties for $f \in \mathcal{C}^{\infty}(\mathcal G_{\lambda})$.
First, observe that since we have the boundary condition
\[
P_2f(x_1,h+\gamma[\theta](x_1-\mathsf x_1)) = 0 \qquad \forall x_1 \in (\mathsf x_1 - \lambda,\mathsf x_1 + \lambda),
\]
we can apply a Hardy inequality on $(-L,h+\gamma[\theta](x_1-\mathsf x_1))$ for any $x_1 \in (\mathsf x_1 - \lambda,\mathsf x_1 + \lambda)$. We therefore compute  the gradient of $P_2$, i.e. 
\[
\partial_1 P_2f(x_1,x_2) =  P_2 \partial_1 f(x_1,x_2) - \partial_x \gamma[\theta] (x_1-\mathsf x_1)f(x_1,h+\gamma[\theta](x_1-\mathsf x_1)) \quad \text{and} \quad 
\partial_2 P_2f = f,
\]
to get
\[
\int_{-L}^{h+\gamma[\theta](x_1-\mathsf x_1)} |P_2f(x_1,x_2)|^2 dx_2 \leq (\mathfrak d + \gamma[\theta](x_1-\mathsf x_1)- \mathsf x_2[\theta])^2 \int_{-L}^{h+\gamma[\theta](x_1-\mathsf x_1)} |f(x_1,x_2)|^2  dx_2,
\]
since $\mathfrak d + \gamma[\theta](x_1-\mathsf x_1)- \mathsf x_2[\theta] = h+\gamma[\theta](x_1-\mathsf x_1) + L.$
We thus infer 
\begin{equation} \label{eq_boundP2L2}
\|P_2f \|_{L^2(\mathcal G_{\lambda})}^2 \leq  \int_{\mathsf x_1 -\lambda}^{\mathsf x_1 + \lambda} (\mathfrak d + \gamma[\theta](x_1-\mathsf x_1)- \mathsf x_2[\theta])^2 \int_{-L}^{h+\gamma[\theta](x_1-\mathsf x_1)} |f(x_1,x_2)|^2dx_1 dx_2 \
\end{equation}
and
$$
\begin{aligned}
&\|\nabla P_2f\|_{L^2(\mathcal G_{\lambda})}^2
 \leq  \int_{\mathsf x_1 -\lambda}^{\mathsf x_1 + \lambda} (\mathfrak d + \gamma[\theta](x_1-\mathsf x_1)- \mathsf x_2[\theta])^2 \int_{-L}^{h+\gamma[\theta](x_1-\mathsf x_1)} |\partial_1 f(x_1,x_2)|^2 dx_2 dx_1\\[6pt]
& \hspace{4mm} + \int_{\mathsf x_1 -\lambda}^{\mathsf x_1 + \lambda}(\mathfrak d + \gamma[\theta](x_1-\mathsf x_1)- \mathsf x_2[\theta]) |\partial_x \gamma[\theta] (x_1-\mathsf x_1)|^2 | f(x_1,h+\gamma[\theta](x_1-\mathsf x_1))|^2  dx_1 .
\end{aligned}
$$

\section*{Declarations}
\begin{itemize}
\item The authors have no relevant financial or non-financial interests to disclose.
\item The authors have no competing interests to declare that are relevant to the content of this article.
\item Data sharing not applicable to this article as no datasets were generated or analysed during the current study.
\end{itemize}


\begin{thebibliography}{64}
\providecommand{\natexlab}[1]{#1}
\providecommand{\url}[1]{\texttt{#1}}
\expandafter\ifx\csname urlstyle\endcsname\relax
  \providecommand{\doi}[1]{doi: #1}\else
  \providecommand{\doi}{doi: \begingroup \urlstyle{rm}\Url}\fi

\bibitem[Amick(1977)]{amick}
C.~J. Amick.
\newblock Steady solutions of the {N}avier-{S}tokes equations in unbounded
  channels and pipes.
\newblock \emph{Annali della Scuola Normale Superiore di Pisa - Classe di
  Scienze}, 4\penalty0 (3):\penalty0 473--513, 1977.

\bibitem[Bae and Jin(2023)]{MR4643426}
H-O. Bae and B.~J. Jin.
\newblock Asymptotic profile for the interaction of a rigid ball and an
  incompressible viscous fluid.
\newblock \emph{J. Differential Equations}, 376:\penalty0 682--713, 2023.

\bibitem[Berchio et~al.(2023)Berchio, Bonheure, Galdi, Gazzola, and
  Perotto]{denisetal}
E.~Berchio, D.~Bonheure, G.P. Galdi, F.~Gazzola, and S.~Perotto.
\newblock Equilibrium configurations of a symmetric body immersed in a
  stationary {N}avier-{S}tokes flow in a planar channel.
\newblock 2023.
\newblock Preprint.

\bibitem[Blackburn and Henderson(1999)]{blackburn1999study}
H.~M. Blackburn and R.~D. Henderson.
\newblock A study of two-dimensional flow past an oscillating cylinder.
\newblock \emph{Journal of Fluid Mechanics}, 385:\penalty0 255--286, 1999.

\bibitem[Bonheure and Galdi(2023)]{denistheboss}
D.~Bonheure and G.~P. Galdi.
\newblock Global weak solutions to a time-periodic body-liquid interaction
  problem.
\newblock 2023.
\newblock Preprint.

\bibitem[Bonheure et~al.(2020)Bonheure, Galdi, and Gazzola]{BoGaGa}
D.~Bonheure, \relax{G.P.} Galdi, and F.~Gazzola.
\newblock Equilibrium configuration of a rectangular obstacle immersed in a
  channel flow.
\newblock \emph{Comptes Rendus. Math\'ematique}, 358\penalty0 (8):\penalty0
  887--896, 2020.
\newblock see also updated version arXiv:2004.10062v2.

\bibitem[Bravin(2019)]{bravin}
M.~Bravin.
\newblock Energy equality and uniqueness of weak solutions of a ``viscous
  incompressible fluid+ rigid body'' system with {N}avier slip-with-friction
  conditions in a 2{D} bounded domain.
\newblock \emph{Journal of Mathematical Fluid Mechanics}, 21\penalty0
  (2):\penalty0 23, 2019.

\bibitem[Conca et~al.(2000)Conca, San~Mart{\'\i}n, and Tucsnak]{conca}
C.~Conca, J.A. San~Mart{\'\i}n, and M.~Tucsnak.
\newblock Existence of solutions for the equations modelling the motion of a
  rigid body in a viscous fluid.
\newblock \emph{Communications in Partial Differential Equations}, 25\penalty0
  (5-6):\penalty0 99--110, 2000.

\bibitem[Cooley and O'Neill(1968)]{Cooley&ONeill68}
M.~D.~A. Cooley and M.~E. O'Neill.
\newblock On the slow rotation of a sphere about a diameter parallel to a
  nearby plane wall.
\newblock \emph{J. Inst. Math. Applics}, 4:\penalty0 163--173, 1968.

\bibitem[Cooley and O'Neill(1969)]{Cooley&Oneill69}
M.D.A. Cooley and M.E. O'Neill.
\newblock On the slow motion generated in a viscous fluid by the approach of a
  sphere to a plane wall or stationary sphere.
\newblock \emph{Mathematika}, 16:\penalty0 37--49, 1969.

\bibitem[Cossu and Morino(2000)]{cossu2000instability}
C.~Cossu and L.~Morino.
\newblock On the instability of a spring-mounted circular cylinder in a viscous
  flow at low {R}eynolds numbers.
\newblock \emph{Journal of Fluids and Structures}, 14\penalty0 (2):\penalty0
  183--196, 2000.

\bibitem[Cox(1974)]{Cox67}
R~G Cox.
\newblock The motion of suspended particles almost in contact.
\newblock \emph{Int. J. Multiphase Flow}, 1:\penalty0 343--371, 1974.

\bibitem[Cumsille and Takahashi(2008{\natexlab{a}})]{CumsilleTakahashi}
P.~Cumsille and {T.} Takahashi.
\newblock Global strong solutions for the two-dimensional motion of an infinite
  cylinder in a viscous fluid.
\newblock \emph{Czechoslovak Mathematical Journal}, 58:\penalty0 961--992,
  2008{\natexlab{a}}.

\bibitem[Cumsille and Takahashi(2008{\natexlab{b}})]{CumsilleTakahashi2}
P.~Cumsille and T.~Takahashi.
\newblock Well posedness for the system modelling the motion of a rigid body of
  arbitrary form in an incompressible viscous fluid.
\newblock \emph{Czechoslovak mathematical journal}, 58\penalty0 (4):\penalty0
  961--992, 2008{\natexlab{b}}.

\bibitem[Dean and O'Neill(1963)]{Dean&Oneill64}
W.~R. Dean and M.~E. O'Neill.
\newblock A slow motion of viscous liquid caused by the rotation of a solid
  sphere.
\newblock \emph{Mathematika}, 10:\penalty0 13--24, 1963.

\bibitem[Desjardins and Esteban(1999)]{DejEste}
B.~Desjardins and M.~Esteban.
\newblock Existence of weak solutions for the motion of rigid bodies in a
  viscous fluid.
\newblock \emph{Archive for Rational Mechanics and Analysis}, 146\penalty0
  (1):\penalty0 59--71, 1999.

\bibitem[Desjardins and Esteban(2000)]{DejEste2}
B.~Desjardins and M.~J. Esteban.
\newblock On weak solutions for fluid-rigid structure interaction: compressible
  and incompressible models.
\newblock \emph{Comm. Partial Differential Equations}, 25\penalty0
  (7-8):\penalty0 1399--1413, 2000.

\bibitem[Dolci and Carmo(2019)]{dolci2019bifurcation}
D.~I. Dolci and B.~S. Carmo.
\newblock Bifurcation analysis of the primary instability in the flow around a
  flexibly mounted circular cylinder.
\newblock \emph{Journal of Fluid Mechanics}, 880:\penalty0 R5, 2019.

\bibitem[Ervedoza et~al.(2014)Ervedoza, Hillairet, and Lacave]{Ervedoza1}
S.~Ervedoza, M.~Hillairet, and C.~Lacave.
\newblock Long-time behaviour for the two-dimensional motion of a disk in a
  viscous fluid.
\newblock \emph{Communications in Mathematical Physics}, 329:\penalty0
  325--382, 2014.

\bibitem[Fani and Gallaire(2015)]{fani2015motion}
A.~Fani and F.~Gallaire.
\newblock The motion of a 2{D} pendulum in a channel subjected to an incoming
  flow.
\newblock \emph{Journal of Fluid Mechanics}, 764:\penalty0 5--25, 2015.

\bibitem[Feireisl(2003)]{Feireisl03}
E.~Feireisl.
\newblock On the motion of rigid bodies in a viscous incompressible fluid.
\newblock \emph{J. Evol. Equ.}, 3\penalty0 (3):\penalty0 419--441, 2003.
\newblock Dedicated to Philippe B\'enilan.

\bibitem[Feireisl and Ne\v{c}asov\'a(2011)]{FereislNecasova}
E.~Feireisl and S.~Ne\v{c}asov\'a.
\newblock On the long-time behaviour of a rigid body immersed in a viscous
  fluid.
\newblock \emph{Applicable Analysis}, 90(1):\penalty0 59--66, 2011.

\bibitem[Galdi(2011)]{galdi}
G.~P. Galdi.
\newblock \emph{An {I}ntroduction to the {M}athematical {T}heory of the
  {N}avier-{S}tokes {E}quations: {S}teady-{S}tate {P}roblems}.
\newblock Springer Science \& Business Media, 2011.

\bibitem[Gazzola et~al.(2024)Gazzola, Pata, and Patriarca]{gazzpatapat}
F.~Gazzola, V.~Pata, and C.~Patriarca.
\newblock Attractors for a fluid-structure interaction problem in a
  time-dependent phase space.
\newblock \emph{Journal of Functional Analysis}, 286\penalty0 (2), 2024.

\bibitem[Geissert et~al.(2013)Geissert, Götze, and Hieber]{Hieber_etal}
M.~Geissert, K.~Götze, and M.~Hieber.
\newblock ${L}^{p}$-theory for strong solutions to fluid-rigid body interaction
  in newtonian and generalized newtonian fluids.
\newblock \emph{Transactions of the American Mathematical Society},
  365\penalty0 (3):\penalty0 1393--1439, 2013.

\bibitem[G{\'e}rard-Varet and Hillairet(2010)]{GeVaHill}
D.~G{\'e}rard-Varet and M.~Hillairet.
\newblock Regularity issues in the problem of fluid structure interaction.
\newblock \emph{Archive for Rational Mechanics and Analysis}, 195:\penalty0
  375--407, 2010.

\bibitem[G{\'e}rard-Varet and Hillairet(2012)]{GeVaHill1}
D.~G{\'e}rard-Varet and M.~Hillairet.
\newblock Computation of the drag force on a sphere close to a wall: the
  roughness issue.
\newblock \emph{{ESAIM:M2AN}}, 46, 2012.

\bibitem[G{\'e}rard-Varet et~al.(2015)G{\'e}rard-Varet, Hillairet, and
  Wang]{GVHW}
D.~G{\'e}rard-Varet, M.~Hillairet, and C.~Wang.
\newblock The influence of boundary conditions on the contact problem in a 3{D}
  {N}avier-{S}tokes flow.
\newblock \emph{J. Math. Pures Appl. (9)}, 103\penalty0 (1):\penalty0 1--38,
  2015.

\bibitem[Glass and Sueur(2015)]{GlassSueur}
O.~Glass and F.~Sueur.
\newblock Uniqueness results for weak solutions of two-dimensional fluid--solid
  systems.
\newblock \emph{Archive for Rational Mechanics and Analysis}, 218\penalty0
  (2):\penalty0 907--944, 2015.

\bibitem[Grandmont and Maday(2000)]{Grandmont&Maday00}
C.~Grandmont and Y.~Maday.
\newblock Existence for an unsteady fluid-structure interaction problem.
\newblock \emph{M2AN Math. Model. Numer. Anal.}, 34\penalty0 (3):\penalty0
  609--636, 2000.

\bibitem[Gunzburger et~al.(2000)Gunzburger, Lee, and Seregin]{GunzLeeSeregin}
{M. D.} Gunzburger, {H-C.} Lee, and {G. A.} Seregin.
\newblock Global existence of weak solutions for viscous incompressible flows
  around a moving rigid body in three dimension.
\newblock \emph{Journal of Mathematical Fluid Mechanics}, 2:\penalty0 219--266,
  2000.

\bibitem[Happel and Brenner(1965)]{Happel&Brenner65}
J.~Happel and H.~Brenner.
\newblock \emph{Low {R}eynolds number hydrodynamics with special applications
  to particulate media}.
\newblock Prentice-Hall Inc., Englewood Cliffs, N.J., 1965.

\bibitem[Haraux(1985{\natexlab{a}})]{haraux1985non}
A.~Haraux.
\newblock Non-resonance for a strongly dissipative wave equation in higher
  dimensions.
\newblock \emph{Manuscripta Mathematica}, 53\penalty0 (1-2):\penalty0 145--166,
  1985{\natexlab{a}}.

\bibitem[Haraux(1985{\natexlab{b}})]{haraux1985two}
A.~Haraux.
\newblock Two remarks on dissipative hyperbolic problems.
\newblock \emph{Research Notes in Mathematics}, 122:\penalty0 161--179,
  1985{\natexlab{b}}.

\bibitem[Hesla(2004)]{Hesla}
T.I. Hesla.
\newblock Collisions of {Smooth} {Bodies} in {Viscous} {Fluids:} a
  {Mathematical} {Investigation}.
\newblock \emph{PhD thesis, University of Minnesota}, 146, 2004.

\bibitem[Hillairet(2007)]{Hillairet2007}
M.~Hillairet.
\newblock Lack of collision between solid bodies in {2D} incompressible viscous
  flow.
\newblock \emph{Communications in Partial Differential Equations}, 32\penalty0
  (9):\penalty0 1345--1371, 2007.

\bibitem[Hillairet and Sabbagh(2023)]{hillairet2023global}
M.~Hillairet and L.~Sabbagh.
\newblock Global solutions to coupled ({N}avier-){S}tokes {N}ewton systems in
  $\mathbb{R}^3$.
\newblock \emph{Asymptotic Analysis}, \penalty0 (Preprint):\penalty0 1--27,
  2023.

\bibitem[Hillairet and Takahashi(2009)]{HillTaka}
M.~Hillairet and T.~Takahashi.
\newblock Collisions in 3{D} fluid structure interactions problems.
\newblock \emph{SIAM Journal on Mathematical Analysis}, 40\penalty0
  (6):\penalty0 2451--2477, 2009.

\bibitem[Hillairet and Takahashi(2010)]{HillairetTakahashi10b}
M.~Hillairet and T.~Takahashi.
\newblock Blow up and grazing collision in viscous fluid solid interaction
  systems.
\newblock \emph{Ann. Inst. H. Poincar\'e Anal. Non Lin\'eaire}, 27\penalty0
  (1):\penalty0 291--313, 2010.

\bibitem[Hillairet and Takahashi(2021)]{HillTaka3}
M.~Hillairet and T.~Takahashi.
\newblock Existence of contacts for the motion of a rigid body into a viscous
  incompressible fluid with the tresca boundary conditions.
\newblock \emph{Tunisian journal of mathematics}, 3\penalty0 (3):\penalty0
  447--468, 2021.

\bibitem[Hillairet et~al.(2018)Hillairet, Seck, and Sokhna]{Sokhna}
M.~Hillairet, D.~Seck, and L.~Sokhna.
\newblock Note on the fall of an axisymmetric body in a perfect fluid over a
  horizontal ramp.
\newblock \emph{C. R. Math. Acad. Sci. Paris}, 356\penalty0 (11-12):\penalty0
  1156--1166, 2018.

\bibitem[Hoffmann and Starovoitov(1999)]{Hoffmann&Starovoitov99}
K.-H. Hoffmann and V.N. Starovoitov.
\newblock On a motion of a solid body in a viscous fluid. {T}wo-dimensional
  case.
\newblock \emph{Adv. Math. Sci. Appl.}, 9\penalty0 (2):\penalty0 633--648,
  1999.

\bibitem[Hoffmann and Starovoitov(2000)]{Hoffmann&Starovoitov00}
K.-H. Hoffmann and V.N. Starovoitov.
\newblock Zur {B}ewegung einer {K}ugel in einer z\"ahen {F}l\"ussigkeit.
\newblock \emph{Doc. Math.}, 5:\penalty0 15--21 (electronic), 2000.

\bibitem[Hogg(1994)]{Hogg}
A.~J. Hogg.
\newblock The inertial migration of non-neutrally buoyant spherical particles
  in two-dimensional shear flows.
\newblock \emph{Journal of fluid mechanics}, 272:\penalty0 285--318, 1994.

\bibitem[Houot and Munnier(2008)]{Munnier1}
J.~Houot and A.~Munnier.
\newblock On the motion and collisions of rigid bodies in an ideal fluid.
\newblock \emph{Asymptot. Anal.}, 56\penalty0 (3-4):\penalty0 125--158, 2008.

\bibitem[Jeffery(1922)]{Jeffery22}
G.B. Jeffery.
\newblock The motion of ellipsoidal particles immersed in a viscous fluid.
\newblock \emph{Proc. London Math. Soc.}, 102:\penalty0 161--179, 1922.

\bibitem[Jin et~al.(2023)Jin, Nečasová, Oschmann, and Roy]{Sarka}
B.~J. Jin, Š. Nečasová, F.~Oschmann, and A.~Roy.
\newblock Collision/no-collision results of a solid body with its container in
  a 3d compressible viscous fluid.
\newblock \emph{arXiv.org}, 2023.

\bibitem[Ladyzhensakaya(1963)]{Ladyzhenskaya1969}
O.A. Ladyzhensakaya.
\newblock \emph{The Mathematical Theory of Viscous Incompressible Flow}.
\newblock Gordon and Breach, 1963.

\bibitem[Li et~al.(2023)Li, Xu, and Zhang]{MR4643472}
H.~Li, L.~Xu, and P.~Zhang.
\newblock Stress blowup analysis when a suspending rigid particle approaches
  the boundary in {S}tokes flow: 2-dimensional case.
\newblock \emph{SIAM J. Math. Anal.}, 55\penalty0 (5):\penalty0 4493--4536,
  2023.

\bibitem[Matas et~al.(2004)Matas, Morris, and Guazzelli]{Matasetal}
J.-P. Matas, J.~F. Morris, and E.~Guazzelli.
\newblock Inertial migration of rigid spherical particles in poiseuille flow.
\newblock \emph{Journal of Fluid Mechanics}, 515:\penalty0 171 -- 195, 2004.

\bibitem[Munnier and Ramdani(2015)]{Munnier2}
A.~Munnier and K.~Ramdani.
\newblock Asymptotic analysis of a {N}eumann problem in a domain with cusp.
  {A}pplication to the collision problem of rigid bodies in a perfect fluid.
\newblock \emph{SIAM J. Math. Anal.}, 47\penalty0 (6):\penalty0 4360--4403,
  2015.

\bibitem[Obligado et~al.(2013)Obligado, Puy, and Bourgoin]{obligado2013bi}
M.~Obligado, M.~Puy, and M.~Bourgoin.
\newblock Bi-stability of a pendular disk in laminar and turbulent flows.
\newblock \emph{Journal of Fluid Mechanics}, 728:\penalty0 R2, 2013.

\bibitem[O'Neill and Stewartson(1967)]{ONeill&Stewartson67}
M.~E. O'Neill and K.~Stewartson.
\newblock On the slow motion of a sphere parallel to a nearby plane wall.
\newblock \emph{J. Fluid Mech.}, 27:\penalty0 705--724, 1967.

\bibitem[Patriarca(2022)]{patriarca}
C.~Patriarca.
\newblock Existence and uniqueness result for a fluid--structure--interaction
  evolution problem in an unbounded 2{D} channel.
\newblock \emph{Nonlinear Differential Equations and Applications NoDEA},
  29\penalty0 (4):\penalty0 1--38, 2022.

\bibitem[Patriarca et~al.(2022)Patriarca, Calamelli, Schito, Argentini, and
  Rocchi]{patriarca2022numerical}
C.~Patriarca, F.~Calamelli, P.~Schito, T.~Argentini, and D.~Rocchi.
\newblock A numerical characterization of the attractor for a fluid-structure
  interaction problem.
\newblock In \emph{Interactions between Elasticity and Fluid Mechanics}, pages
  175--192. EMS Press, 2022.

\bibitem[S.~Ervedoza and Tucsnak(2021)]{Ervedoza2}
D.~Maity S.~Ervedoza and M.~Tucsnak.
\newblock Large time behaviour for the motion of a solid in a viscous
  incompressible fluid.
\newblock \emph{Mathematische Annalen}, 2021.

\bibitem[Sabbagh(2019)]{sabbagh2019motion}
L.~Sabbagh.
\newblock On the motion of several disks in an unbounded viscous incompressible
  fluid.
\newblock \emph{Nonlinearity}, 32\penalty0 (6):\penalty0 2157, 2019.

\bibitem[San~Mart{\'\i}n et~al.(2002)San~Mart{\'\i}n, Starovoitov, and
  Tucsnak]{SMStarTuck}
J.~A. San~Mart{\'\i}n, V.N. Starovoitov, and M.~Tucsnak.
\newblock Global weak solutions for the two-dimensional motion of several rigid
  bodies in an incompressible viscous fluid.
\newblock \emph{Archive for Rational Mechanics and Analysis}, 161:\penalty0
  113--147, 2002.

\bibitem[Segr{\'e} and Silberberg(1962)]{SegreSilberberg62}
G.~Segr{\'e} and A.~Silberberg.
\newblock Behaviour of macroscopic rigid spheres in poiseuille flow part 2.
  experimental results and interpretation.
\newblock \emph{Journal of Fluid Mechanics}, 14:\penalty0 136 -- 157, 1962.

\bibitem[Starovoitov(2003)]{Star2}
V.N. Starovoitov.
\newblock Behavior of a rigid body in an incompressible viscous fluid near a
  boundary.
\newblock In \emph{Free Boundary Problems: Theory and Applications (Trento,
  Italy)}, volume 147, pages 313--327. Springer, 2003.

\bibitem[Starovoitov(2005)]{Star}
V.N. Starovoitov.
\newblock Nonuniqueness of a solution to the problem on motion of a rigid body
  in a viscous incompressible fluid.
\newblock \emph{Journal of Mathematical Sciences}, 130:\penalty0 4893--4898,
  2005.

\bibitem[Takahashi(2003)]{Takahashi}
T.~Takahashi.
\newblock Analysis of strong solutions for the equations modeling the motion of
  a rigid-fluid system in a bounded domain.
\newblock \emph{Advances in Difference Equations}, 8:\penalty0 1499--1532,
  2003.

\bibitem[Takahashi and Tucsnak(2004)]{TakahashiTucsnak}
{T.} Takahashi and M.~Tucsnak.
\newblock Global strong solutions for the two-dimensional motion of an infinite
  cylinder in a viscous fluid.
\newblock \emph{Journal of Mathematical Fluid Mechanics}, 6:\penalty0 53--77,
  2004.

\bibitem[Williamson and Govardhan(2004)]{williamson2004vortex}
C.~H.~K. Williamson and R.~Govardhan.
\newblock Vortex-induced vibrations.
\newblock \emph{Annual Review of Fluid Mechanics}, 36:\penalty0 413--455, 2004.

\end{thebibliography}
\end{document}